\documentclass[11pt]{amsart}
\usepackage{stmaryrd}
\usepackage{}
\usepackage{pifont}

\usepackage{amsmath}
\usepackage{amssymb}
\usepackage{latexsym}
\usepackage{amscd}
\usepackage{mathrsfs}
\usepackage[all]{xy}

\newdimen\AAdi%
\newbox\AAbo%
\def\AAk#1#2{\s_etbox\AAbo=\hbox{#2}\AAdi=\wd\AAbo\kern#1\AAdi{}}%
\def\AAr#1#2#3{\s_etbox\AAbo=\hbox{#2}\AAdi=\ht\AAbo\raise#1\AAdi\hbox{#3}}%
\font\tenmsb=msbm10 at 12pt \font\sevenmsb=msbm7 at 8pt
\font\fivemsb=msbm5 at 6pt
\newfam\msbfam
\textfont\msbfam=\tenmsb \scriptfont\msbfam=\sevenmsb
\scriptscriptfont\msbfam=\fivemsb

\newtheorem{theorem}{Theorem}

\newtheorem{remark}[theorem]{Remark}

\newtheorem{corollary}[theorem]{Corollary}

\newtheorem{lemma}[theorem]{Lemma}
\newtheorem{proposition}[theorem]{Proposition}

\numberwithin{equation}{section} \numberwithin{theorem}{section}

\renewcommand{\topmargin}{0cm}
\renewcommand{\oddsidemargin}{5mm}
\renewcommand{\evensidemargin}{5mm}
\renewcommand{\textwidth}{150mm}
\renewcommand{\textheight}{230mm}

\def\R{\mathbb R}

\def\N{\mathbb N}

\def\na{\nabla}

\def\f#1#2{\frac{#1}{#2}}

\def\a{\alpha}
\def\be{\beta}

\def\r{\Re_{I\!V}}

\def\p#1{\partial #1}

\def\de{\delta}
\def\De{\Delta}
\def\e{\eta}
\def\ep{\epsilon}

\def\G{\Gamma}
\def\g{\gamma}
\def\k{\kappa}
\def\la{\lambda}
\def\La{\Lambda}
\def\lan{\langle}
\def\ran{\rangle}

\def\Om{\Omega}
\def\th{\theta}
\def\Th{\Theta}

\def\Si{\Sigma}

\def\r{\rho}
\def\z{\zeta}

\begin{document}

\title
[Area-minimizing hypersurfaces in manifolds]
{Area-minimizing hypersurfaces in manifolds of Ricci curvature bounded below}

\author{Qi Ding}
\address{Shanghai Center for Mathematical Sciences, Fudan University, Shanghai 200438, China}
\email{dingqi@fudan.edu.cn}%

\begin{abstract}
In this paper, we study area-minimizing hypersurfaces in manifolds of Ricci curvature bounded below with Cheeger-Colding theory. Let $N_i$ be a sequence of smooth manifolds with Ricci curvature $\ge-n\k^2$ on $B_{1+\kappa'}(p_i)$ for constants $\k\ge0$, $\kappa'>0$, and volume of $B_1(p_i)$ has a positive uniformly lower bound. Assume $B_1(p_i)$ converges to a metric ball $B_1(p_\infty)$ in the Gromov-Hausdorff sense. For an area-minimizing hypersurface $M_i$ in $B_1(p_i)$ with $\p M_i\subset\p B_1(p_i)$, we prove the continuity for the volume function of area-minimizing hypersurfaces equipped with the induced Hausdorff topology. In particular, each limit $M_\infty$ of $M_i$ is area-minimizing in $B_1(p_\infty)$ provided $B_1(p_\infty)$ is a smooth Riemannian manifold. By blowing up argument, we get sharp dimensional estimates for the singular set of $M_\infty$ in $\mathcal{R}$, and $\mathcal{S}\cap M_\infty$. Here, $\mathcal{R}$, $\mathcal{S}$ are the regular and singular parts of $B_1(p_\infty)$, respectively.
\end{abstract}

\maketitle 

\section{Introduction}

In Euclidean space, area-minimizing hypersurfaces have been studied intensely several decades before (see \cite{Gi}\cite{LYa}\cite{S} for a systematical introduction).
The theory acts an important role in the famous Bernstein theorem for minimal graphs in Euclidean space.
Let $\Om$ be an open subset in $\R^{n+1}$, $\Si_i$ be a sequence of area-minimizing hypersurfaces in $\Om$ with $\p \Si_i\subset\p\Om$.
The compactness theorem tells us that there are an area-minimizing hypersurface $\Si$ in $\Om$ with $\p\Si\subset\p\Om$, and a subsequence $i_j$ such that $\Si_{i_j}$ converges to $\Si$ in the weak sense (see also Theorem 37.2 in \cite{S}). In particular, for any open $V\subset\subset\Om$ we have (see Lemma 9.1 in \cite{Gi} for the version of bounded variation)
\begin{equation}\aligned\label{SiijVEuc}
\lim_{j\rightarrow\infty}\mathcal{H}^n(\Si_{i_j}\cap V)=\mathcal{H}^n(\Si\cap V)\qquad \mathrm{provided} \ \mathcal{H}^n(\Si\cap\p V)=0.
\endaligned
\end{equation}

For any area-minimizing hypersurface $M$ in $\R^{n+1}$ (or a smoooth manifold), De Giorgi \cite{DG0}, Federer \cite{Fe}, Reifenberg \cite{Re} proved that the singular set of $M$ has Hausdorff dimension $\le n-7$.
Recently, Cheeger-Naber \cite{CN} and Naber-Valtorta \cite{NV} made important progress on quantitative
stratifications of the singular set of stationary varifolds. 
In particular, for an area-minimizing hypersurface $M$ they proved that
the second fundamental form of $M$ has \emph{a\,priori} estimates in $L^7_{weak}$ on $M$ \cite{NV}.
Needless to say, all area-minimizing hypersurfaces in smooth manifolds are stable. The theory of stable minimal surfaces is a powerful tool to study the topology of 3-dimensional manifolds,
see \cite{AR}\cite{FS}\cite{Liu}\cite{SY}\cite{SY0} for instance.

In general speaking, local calculations concerned on minimal hypersurfaces in manifolds usually contain sectional curvature of the ambient manifolds.
However, many properties of minimal hypersurfaces may have nothing to do with sectional curvature of the ambient manifolds.
In this paper, we study area-minimizing hypersurfaces in manifolds of Ricci curvature bounded below with Cheeger-Colding theory \cite{CCo1,CCo2,CCo3}.
For overcoming the difficulty of lacking sectional curvature condition, it is natural to take the possible limits in Ricci limit spaces, and then go back to the original problems by using the properties of such limits.
In this sense, it is worth to understand the limits of a sequence of area-minimizing hypersurfaces in a sequence of geodesic balls with Ricci curvature uniformly bounded below.
If the volume of unit geodesic balls goes to zero, the limit of the hypersurfaces may be equal to the whole ambient Ricci limit spaces.
So we assume the non-collapsing condition on ambient manifolds in the most cases, i.e., the volume of geodesic balls has a uniformly positive lower bound.
In a sequel, we will study Sobolev and Poincar$\mathrm{\acute{e}}$ inequalities on minimal graphs over manifolds \cite{D2}.

Let $S$ be a smooth minimal hypersurface in an $(n+1)$-dimensional complete Riemannian manifold $N$.
Using Jacobi fields, Heintze-Karcher \cite{HK} established the Laplacian comparison
for distance functions to $M$ outside cut loci in terms of Ricci curvature of ambient spaces.
Similar to Schoen-Yau's argument in \cite{SY2}, the comparison theorem holds globally in the distribution sense even $S$ instead by the support of an $n$-rectifiable stationary varifold in $N$.
Then we are able to estimate the positive lower bound of the volume of $S\cap B_1(p)\subset N$ using volume and the lower bound of Ricci curvature of $B_1(p)$ (compared \cite{IK}\cite{Mm}).
As an application, we obtain a non-existence result for complete area-minimizing hypersurfaces without the sectional curvature of ambient manifolds
(compared with Theorem 1 of Anderson \cite{An}; see Theorem \ref{Nonexistminimizinghypersurface}).

If $M$ is an area-minimizing hypersurface in $B_1(p)\subset N$ with $\p M\subset\p B_1(p)$, then the volume of $M$ in $B_1(p)$ has a uniform upper bound by a constant via the universal covers of ambient manifolds (see Lemma \ref{upbdareaM}).
With lower and upper bounds for the volume of area-minimizing hypersurfaces, we are able to consider their possible limits in Ricci limit spaces in the following sense.

Let $N_i$ be a sequence of $(n+1)$-dimensional smooth manifolds of Ricci curvature
$\ge-n\k^2$ on $B_{1+\k'}(p_i)$
for constants $\k\ge0$, $\k'>0$.
By Gromov precompactness theorem, up to a choice of the subsequence $\overline{B_1(p_i)}$ converges to a metric ball $\overline{B_1(p_\infty)}$ in the Gromov-Hausdorff sense.
For each $i$, let $\Phi_i:\,B_1(p_i)\rightarrow B_1(p_\infty)$ be an $\ep_i$-Hausdorff approximations with $\ep_i\rightarrow0$.
We further assume $\mathcal{H}^{n+1}(B_1(p_i))\ge v$ for some positive constant $v$.
Let $M_i$ be an area-minimizing hypersurface in $B_1(p_i)$ with $\p M_i\subset\p B_1(p_i)$.
Suppose that $\Phi_i(M_i)$ converges in the Hausdorff sense to a closed set $M_\infty$ in $\overline{B_1(p_\infty)}$ as $i\rightarrow\infty$.
Suppose $M_\infty\cap B_1(p_\infty)\neq\emptyset$.

Colding \cite{C}, Cheeger-Colding \cite{CCo1} proved the volume convergence of $B_1(p_i)$ under the Gromov-Hausdorff topology.
Based on their results, we get the continuity for the volume function of $M_i$ in the following sense.
\begin{theorem}\label{1.2}
For any open set $\Om\subset\subset B_1(p_\infty)$, let $\Om_i\subset B_1(p_i)$ be open with $\Phi_i(\Om_i)\rightarrow\Om$ in the Hausdorff sense. Then
\begin{equation}\aligned\label{HnMiOmiMOm000}
\limsup_{i\rightarrow\infty}\mathcal{H}^n\left(M_i\cap \overline{\Om_i}\right)\le&\mathcal{H}^n\left(M_\infty\cap \overline{\Om}\right)\\
\endaligned
\end{equation}
and
\begin{equation}\aligned\label{HnMiOmsiMOms000}
\mathcal{H}^n\left(M_\infty\cap B_s(\Om)\right)\le&\liminf_{i\rightarrow\infty}\mathcal{H}^n(M_i\cap B_s(\Om_i))\\
\endaligned
\end{equation}
for any $B_s(\Om)\subset\subset B_1(p_\infty)$ with $s>0$.
Moreover, $M_\infty$ is an area-minimizing hypersurface in $B_1(p_\infty)$ provided $\overline{B_1(p_\infty)}$ is a smooth Riemannian manifold.
\end{theorem}
\eqref{SiijVEuc} can be seen as a special version of \eqref{HnMiOmiMOm000}\eqref{HnMiOmsiMOms000} in Euclidean space.
From \eqref{HnMiOmiMOm000}\eqref{HnMiOmsiMOms000}, we immediately have
\begin{equation}\aligned\label{HnMBtplim000}
&\mathcal{H}^n\left(M_\infty\cap \overline{B_t(p_\infty)}\right)\ge\limsup_{i\rightarrow\infty}\mathcal{H}^{n}\left(M_i\cap \overline{B_t(p_i)}\right)\\
\ge&\liminf_{i\rightarrow\infty}\mathcal{H}^{n}(M_i\cap B_t(p_i))\ge\mathcal{H}^n(M_\infty\cap B_t(p_\infty))
\endaligned
\end{equation}
for any $t\in(0,1)$.
In particular, the inequality \eqref{HnMBtplim000} attains equality when $B_1(p_\infty)$ is a Euclidean ball.
Since $N_i$ may have unbounded sectional curvature, we can not directly prove Theorem \ref{1.2} by following the idea from the Euclidean case.
The proof of Theorem \ref{1.2} will be given by Theorem \ref{Conv0}, Lemma \ref{UpMinfty*} and Lemma \ref{LimMi}.
In $\S4$ we first prove \eqref{HnMBtplim000} and that $M_\infty$ is minimizing in $B_1(p_\infty)$ under the smooth condition on $\overline{B_1(p_\infty)}$.
Then using the results in $\S4$ and covering techniques, we can prove \eqref{HnMiOmiMOm000}\eqref{HnMiOmsiMOms000} in $\S5$
by combining Cheeger-Colding theory and the theory of area-minimizing hypersurfaces in Euclidean space.

The proof of \eqref{HnMBtplim000} for smooth $\overline{B_1(p_\infty)}$ will be divided into two parts (see Theorem \ref{Conv0}).
On the one hand, we prove area-minimizing $M_\infty$ and \eqref{HnMBtplim000} in the sense of Minkowski content, where we use the volume convergence of unit geodesic balls by Colding \cite{C}, Cheeger-Colding \cite{CCo1}.
On the other hand, for any $x\in M_\infty$ there is a constant $r_x>0$ so that $B_{r_x}(x)\cap M_\infty$ can be written as the boundary of an open set of finite perimeter in $B_{r_x}(x)$ using the volume convergence.
Then combining Theorem 2.104 in \cite{AFP} by Ambrosio-Fusco-Pallara, we can prove the equivalence of Hausdorff measure and Minkowski content for $M_\infty$. As an application, the local volume of area-minimizing hypersurfaces in a class of manifolds can be controlled by large-scale conditions (see Theorem \ref{IntReg}).

Let $\mathcal{R}$, $\mathcal{S}$ denote the regular and singular parts of $B_1(p_\infty)$, respectively.
Let $\mathcal{S}_{M_\infty}$ denote the singular set of $M_\infty$ in $\mathcal{R}$, i.e., a set containing all the points $x\in\mathcal{R}$ such that one of tangent cones of $M_\infty$ at $x$ is not flat. In fact, we can prove that all the possible tangent cones of $M_\infty$ at $x$ are not flat provided $x\in\mathcal{S}_{M_\infty}$.
By studying the local cone structure of $M_\infty$, we have the following sharp dimensional estimates for $\mathcal{S}_{M_\infty}$ and $M_\infty\cap\mathcal{S}$ (see Lemma \ref{codim7} and Lemma \ref{codim2}).
\begin{theorem}\label{dimestn-7n-2}
$\mathcal{S}_{M_\infty}$ has Hausdorff dimension $\le n-7$ for $n\ge7$, and it is empty for $n<7$;
$\mathcal{S}\cap M_\infty$ has Hausdorff dimension $\le n-2$.
\end{theorem}
Here, the codimension 7 is sharp if we choose $N_\infty$ as Euclidean space.
Suppose that $N_i$ splits off a line $\R$ isometrically, i.e., $N_i=\Si_i\times\R$ for some smooth manifold $\Si_i$. If we choose $M_i$ as $\Si_i\times\{0\}\subset\Si_i\times\R$, then the codimension 2 in the above theorem is sharp as the singular part of the limit of $\Si_i$ may have codimension 2.

\textbf{Remark.} 
We do not know whether $M_\infty$ is associated with the rectifiable current defined in \cite{AK} by Ambrosio-Kirchheim. If we assume further $M_i=\p E_i\cap B_1(p_i)$ with some set $E_i\subset B_1(p_i)$ for each $i$, then the results in Theorem \ref{dimestn-7n-2} can be generalized into the setting of RCD spaces in \cite{MS} by Mondino-Semola (about the same time as me on the arXiv website).
Theorem \ref{dimestn-7n-2} will act a crucial role for getting Sobolev inequality and Neumann-Poincar\'e inequality on minimal graphs over manifolds of Ricci curvature bounded below in \cite{D2}.

\emph{Acknowledgments.}
The author would like to thank J$\mathrm{\ddot{u}}$rgen Jost, Yuanlong Xin, Hui-Chun Zhang, Xi-Ping Zhu, Xiaohua Zhu for their interests and valuable discussions. The author would like to thank Gioacchino Antonelli and Daniele Semola for careful reading and valuable advice.
The author is partially supported by NSFC 11871156 and NSFC 11922106.

\section{Preliminary}

Let $(X,d)$ be a complete metric space.
For any subset $E\subset X$ and any constant $s\ge0$, we denote $\mathcal{H}^s(E)$ be the $s$-dimensional Hausdorff measure of $E$. Namely,
\begin{equation}\aligned\label{SH}
\mathcal{H}^s(E)=\lim_{\de\rightarrow0}\mathcal{H}^s_\de(E)
\endaligned
\end{equation}
with
\begin{equation}\aligned\label{mathcalHnep}
\mathcal{H}_\de^s(E)=\f{\omega_s}{2^s}\inf\left\{\sum_{i=1}^\infty (\mathrm{diam} U_i)^s\Big|\ E\subset\bigcup_{i=1}^\infty U_i\ \mathrm{for\ Borel\ sets}\ U_i\subset X,\ \mathrm{diam} U_i<\de\right\},
\endaligned
\end{equation}
where $\omega_s=\f{\pi^{s/2}}{\G(\f s2+1)}$, and $\G(r)=\int_0^\infty e^{-t}t^{r-1}dt$ is the gamma function for $0<r<\infty$.
In particular, for integer $s\ge1$, $\omega_s$ is the volume of the unit ball in $\R^s$, and $\omega_0=1$.
For each $p\in X$, let $B_r(p)$ denote the geodesic ball in $X$ with radius $r$ and centered at $p$. More general, for any set $K$ in $X$, let $\r_K=d(\cdot,K)=\inf_{y\in K}d(\cdot,y)$ be the distance function from $K$,
and $B_t(K)$ denote the $t$-neighborhood of $K$ in $X$ defined by $\{x\in X|\ d(x,K)<t\}$.
Moveover, $\r_K$ is a Lipschitz function on $X$ with the Lipschitz constant $\le1$.
Let $\Om\subset X$ be a measurable set of Hausdorff dimension $m\ge1$.
We define the upper and the lower $(m-1)$-dimensional \emph{Minkowski contents} of $K$ in $\Om$ by
\begin{equation}\aligned
\mathcal{M}^*(K,\Om)=&\limsup_{\de\rightarrow0}\f1{2\de}\mathcal{H}^{m}(\Om\cap B_\de(K)\setminus K)\\ \mathcal{M}_*(K,\Om)=&\liminf_{\de\rightarrow0}\f1{2\de}\mathcal{H}^{m}(\Om\cap B_\de(K)\setminus K).
\endaligned
\end{equation}
If $\mathcal{M}^*(K,\Om)=\mathcal{M}_*(K,\Om)$, the common value is denoted by $\mathcal{M}(K,\Om)$. 

Let $N$ be an $(n+1)$-dimensional complete Riemannian manifold with Ricci curvature $\ge-n\k^2$ on $B_R(p)$ for constants $\k\ge0$ and $R>0$.
For any integer $k\ge0$, let $V^k_s(r)$ denote the volume of a  geodesic ball with radius $r$ in a $k$-dimensional space form with constant sectional curvature $-s^2$.
In fact, $V^k_s(r)=k\omega_k s^{1-k}\int_0^r\sinh^{k-1}(st)dt$,
and $\left(V^{k}_{s}(r)\right)'=k\omega_ks^{1-k}\sinh^{k-1}(s r)$ for each $r>0$.
By Bishop-Gromov volume comparison,
\begin{equation}\aligned\label{BiVolN}
1\ge\f{\mathcal{H}^{n+1}(B_{r_1}(p))}{V^{n+1}_{\k}(r_1)}\ge\f{\mathcal{H}^{n+1}(B_{r_2}(p))}{V^{n+1}_{\k}(r_2)}
\endaligned
\end{equation}
for any $0<r_1\le r_2\le R$, and
\begin{equation}\aligned\label{HnpBr}
\mathcal{H}^{n}(\p B_{r}(p))\le \f{\left(V^{n+1}_{\k}(r)\right)'}{V^{n+1}_{\k}(r)}\mathcal{H}^{n+1}(B_{r}(p))\le\left(V^{n+1}_{\k}(r)\right)'=(n+1)\omega_{n+1}\f{\sinh^n(\k r)}{\k^n}
\endaligned
\end{equation}
for any $0<r\le R$. Here, we see $\sinh(\k r)/\k=r$ for $\k=0$.
Let us recall the isoperimetric inequality on $B_r(p)$ (see \cite{An1,B,Cr} for instance, or from the heat kernel \cite{Gri,LW}):
for any $r\in(0,R/2]$ there is a constant $\a_{n,\k r}>0$ depending only on $n,\k r$ such that for any open set $\Om\subset B_r(p)$, one has
\begin{equation}\label{isoperi}
\f{\mathcal{H}^{n}(\p\Om)}{\left(\mathcal{H}^{n+1}(\Om)\right)^{\f{n}{n+1}}}\ge \a_{n,\k r}\left(\f{\mathcal{H}^{n+1}(B_r(p))}{V^{n+1}_{\k}(r)}\right)^{\f1{n+1}}.
\end{equation}

Let exp$_p$ denote the exponential map from the tangent space $T_pN$ into $N$.
For any two constants $R,\tau>0$, let $\Si$ be an embedded $C^2$-hypersurface in $B_{R+\tau}(p)$ with $\p \Si\subset\p B_{R+\tau}(p)$ ($\p B_{R+\tau}(p)$, $\p \Si$ maybe empty).
By the uniqueness of the geodesics, if $\r_\Si$ is differentiable at $x\in B_{R}(p)\cap B_\tau(\Si)\setminus\Si$,
then there exist a unique $x_\Si\in \Si$ and a unique non-zero vector $v_x\in\R^{n+1}$ with $|v_x|=\r_\Si(x)$ such that
$\mathrm{exp}_{x_\Si}(v_x)=x$.
Let $\g_x(t)$ denote the geodesic $\mathrm{exp}_{x_\Si}(tv_x/|v_x|)$ from $t=0$ to $t=|v_x|$. In particular, $\r_\Si$ is smooth at $\g_x(t)$ for $t\in[0,|v_x|]$.
Let $H_x(t), A_x(t)$ denote the mean curvature (pointing out of $\{\r_\Si<t\}$), the second fundamental form of the level set $\{\r_\Si=t\}$ at $\g_x(t)$, respectively.
Let $\De_N$ denote the Laplacian of $N$ with respect to its Riemannian metric.
From Heintze-Karcher \cite{HK}, one has
\begin{equation}\label{Hxtge}
-\De_N\r_\Si=H_x(\r_\Si)\ge H_x(0)-n\k\tanh\left(\k\r_\Si\right)
\end{equation}
on the geodesic $\mathrm{exp}_{x_\Si}(tv_x/|v_x|)$ for $t\in[0,|v_x|]$.
In fact, from the variational argument
\begin{equation}\label{HtRic}
\f{\p H_x}{\p t}=|A_x|^2+Ric\left(\dot{\g}_x,\dot{\g}_x\right)\ge\f1n|H_x|^2-n\k^2,
\end{equation}
we can solve the above differential inequality \eqref{HtRic}, and obtain \eqref{Hxtge}.

For an open set $U\subset N$, we suppose that $U$ is properly embedded in $\R^{n+m}$ for some integer $m\ge1$.
For a set $S$ in $U$, $S$ is said to be \emph{countably $n$-rectifiable}
if $S\subset S_0\cup\bigcup_{j=1}^\infty F_j(\R^n)$, where $\mathcal{H}^n(S_0)=0$, and $F_j:\, \R^n\rightarrow \R^{n+m}$ are Lipschitz mappings for all integers $j\ge1$.
We further assume that $S$ is compact in $N$, and that there are a Radon measure $\mu$ in $N$ and a constant $\g>0$ such that $\mu$ is absolutely
continuous with respect to $\mathcal{H}^n$ and $\mu(B_r(x))\ge\g r^n$
for any $x\in S$, any $r\in(0,1)$. Then from Ambrosio-Fusco-Pallara (in \cite{AFP}, p. 110)
\begin{equation}\aligned\label{MSHS}
\mathcal{M}(S,N)=\mathcal{M}(S,U)=\mathcal{H}^n(S).
\endaligned
\end{equation}

Let $G_{n,m}$ denote the (Grassmann) manifold including all the $n$-dimensional subspace of $\R^{n+m}$.
An $n$-varifold $V$ in $U$ is a Radon measure on
$$G_{n,m}(U)=\{(x,T)|\, x\in U,\, T\in G_{n,m}\cap T_xN\}.$$
An $n$-rectifiable varifold in $U$ is an $n$-varifold in $U$ with support on countably $n$-rectifiable sets.
An $n$-varifold $V$ is said to be \emph{stationary} in $U$ if
\begin{equation}\aligned\nonumber
\int_{G_{n,m}(U)}\mathrm{div}_\omega YdV(x,\omega)=0
\endaligned
\end{equation}
for each $Y\in C^\infty_c(U,\R^{n+m})$ with $Y(x)\subset T_xN$ for each $x\in U$.
The singular set of $V$ is a set containing every point $x$ in spt$V\cap U$ such that one of tangent cone of $V$ at $x$ is not flat.
The notion of $n$-rectifiable stationary varifolds obviously generalizes the notion of minimal hypersurfaces. See \cite{LYa}\cite{S} for more details on varifolds and currents.

Let $\mathcal{D}^n(U)$ denote the set including all smooth $n$-forms on $U$ with compact supports in $U$. Denote $\mathcal{D}_n(U)$ be the set of $n$-currents in $U$, which are continuous linear functionals
on $\mathcal{D}^n(U)$. For each $T\in \mathcal{D}_n(U)$ and each open set $W$ in $U$, one defines the mass of $T$ on $W$ by
\begin{equation*}\aligned
\mathbb{M}(T\llcorner W)=\sup_{|\omega|_U\le1,\omega\in\mathcal{D}^n(U),\mathrm{spt}\omega\subset W}T(\omega)
\endaligned
\end{equation*}
with $|\omega|_U=\sup_{x\in U}\lan\omega(x),\omega(x)\ran^{1/2}$.
Let $\p T$ be the boundary of $T$ defined by $\p T(\omega')=T(d\omega')$ for any $\omega'\in\mathcal{D}^{n-1}(U)$.
For a countably $n$-rectifiable set $M\subset U$ with orientation $\xi$ (i.e., $\xi(x)$ is an $n$-vector representing $T_xM$ for $\mathcal{H}^n$-a.e. $x$), there is an $n$-current $\llbracket M\rrbracket\in\mathcal{D}_n(U)$ associated with $M$, i.e.,
\begin{equation*}\aligned
\llbracket M\rrbracket(\omega)=\int_M\lan \omega,\xi\ran, \qquad  \omega\in\mathcal{D}^n(U).
\endaligned
\end{equation*}
A countably $n$-rectifiable set $M$ (with orientation $\xi$) is said to be an \emph{area-minimizing} hypersurface in $U$ if the associated current $\llbracket M\rrbracket$ is a minimizing current in $U$. Namely,
$\mathbb{M}(\llbracket M\rrbracket\llcorner W)\le \mathbb{M}(T\llcorner W)$ whenever $W\subset\subset U$, $\p T=\p\llbracket M\rrbracket$ in $U$, spt$(T-\llbracket M\rrbracket)$ is compact in $W$ (see \cite{S} for instance).
In particular, $M$ does not contain any closed minimal hypersurface.
Similar to the Euclidean case, $M$ is one-sided, and smooth outside a closed set of Hausdorff dimension $\le n-7$ (\cite{DG0}\cite{Fe}\cite{Re}).

If $Z_1,Z_2$ are both metric spaces, then an admissible metric on the disjoint union $Z_1\coprod Z_2$ is a metric that extends the given metrics on $Z_1$ and $Z_2$. With this one can
define the Gromov-Hausdorff distance as
$$d_{GH}(Z_1,Z_2)=\inf\left\{d_H(Z_1,Z_2)\Big|\ \mathrm{adimissible\ metrics\ on}\ Z_1\coprod Z_2\right\},$$
where $d_H(Z_1,Z_2)$ is the Hausdorff distance of $Z_1,Z_2$.
For any $\ep>0$, a map $\Phi:\ Z_1\rightarrow Z_2$ is said to be an $\ep$-\emph{Hausdorff approximation} (see \cite{Fu} for example) if $Z_2$ is the $\ep$-neighborhood of the image $\Phi(Z_1)$ and
$$\left|d(x_1,x_2)-d(\Phi(x_1),\Phi(x_2))\right|\le\ep\qquad \ \mathrm{for\ every}\ x_1,x_2\in Z_1.$$
If $d_{GH}(Z_1,Z_2)\le\ep$, then there exists a $3\ep$-Hausdorff approximation $Z_1\rightarrow Z_2$. If there
exists an $\ep$-Hausdorff approximation $Z_1\rightarrow Z_2$, then $d_{GH}(Z_1,Z_2)\le 3\ep$.

Let $\{Z_i\}_{i=1}^\infty$ be a sequence of metric spaces with a point $p_i\in Z_i$, and $K_i$ be a subset of $Z_i$ for each $i$.
Let $Z_\infty$ be a metric space with a point $p_\infty\in Z_\infty$.
\begin{itemize}
  \item Case 1: $Z_\infty$ is compact with $\lim_{i\rightarrow\infty}d_{GH}(Z_i,Z_\infty)=0$. Namely, there is an $\ep_i$-Hausdorff approximations $\Phi_i:\ Z_i\rightarrow Z_\infty$ for some sequence $\ep_i\rightarrow0$.
Then from Blaschke theorem (see Theorem 7.3.8 in \cite{BBI} for instance), there is a closed subset $K_\infty$ of $Z_\infty$ such that $\Phi_{i}(K_i)$ converges to $K_\infty$ in the Hausdorff sense up to choose the subsequence.
  \item Case 2: $Z_\infty$ is non-compact with $(B_{R_i}(p_i),p_i)$ converges to a metric space $(Z_\infty,p_\infty)$ in the pointed Gromov-Hausdorff sense for some sequence $R_i\rightarrow\infty$.
Then there are sequences $r_i\rightarrow\infty$, $\ep_i\rightarrow0$
and a sequence of $\ep_{i}$-Hausdorff approximations $\Phi_{i}:\ B_{r_i}(p_i)\rightarrow B_{r_i}(p_\infty)$.
From Blaschke theorem, there is a closed set $K_\infty$ of $Z_\infty$ such that for each $0<r<\infty$, $\Phi_{i}\left(K_i\cap B_r(p_i)\right)$ converges to $K_\infty\cap B_r(p_\infty)$ in the Hausdorff sense up to choose the (diagonal) subsequence.
\end{itemize}
It is easy to see that $K_\infty$ depends on the choice of $\Phi_i$ in the sense of isometry group of $Z_\infty$.
For simplicity, we call that $K_i$ converges \textbf{in the induced Hausdorff sense} to $K_\infty$ in case 1,
and $(K_i,p_i)$ converges \textbf{in the induced Hausdorff sense} to $(K_\infty,p_\infty)$ in case 2 
unless we need emphasize the Hausdorff approximations $\Phi_i$.

Let $N_i$ be a sequence of $(n+1)$-dimensional smooth manifolds of Ricci curvature
\begin{equation}\aligned\label{Ric}
\mathrm{Ric}\ge-n\k^2\qquad \mathrm{on}\ B_{1+\k'}(p_i)
\endaligned
\end{equation}
for constants $\k\ge0$, $\k'>0$.
By Gromov's precompactness theorem (see \cite{GLP} for instance), up to choose the subsequence $B_1(p_i)$ converges to a metric ball $B_1(p_\infty)$
in the Gromov-Hausdorff sense.
From Cheeger-Colding theory \cite{CCo1,CCo2,CCo3},
there is a unique Radon measure $\nu_\infty$ on $B_1(p_\infty)$, which is a renormalized limit measure on $B_1(p_\infty)$ given by
\begin{equation}\aligned\label{nuinfty}
\nu_\infty(B_r(y))=\lim_{i\rightarrow\infty}\f{\mathcal{H}^{n+1}(B_r(y_i))}{\mathcal{H}^{n+1}(B_1(p_i))}
\endaligned
\end{equation}
for every ball $B_r(y)\subset B_1(p_\infty)$ and for each sequence $y_i\in B_1(p_i)$ converging to $y$. Then $\nu_\infty(B_1(p_\infty))=1$ and $\nu_\infty$ satisfies Bishop-Gromov comparison theorem in form as \eqref{BiVolN}.
We further assume that $B_1(p_i)$ satisfies a non-collapsing condition, i.e.,
\begin{equation}\aligned\label{Vol}
\mathcal{H}^{n+1}(B_1(p_i))\ge v\qquad \mathrm{for\ some\ constant}\ v>0.
\endaligned
\end{equation}
Then $\nu_\infty$ is just a multiple of the Hausdorff measure $\mathcal{H}^{n+1}$ and the volume convergence
\begin{equation}\aligned\label{VolCOV}
\mathcal{H}^{n+1}(B_1(p_\infty))=\lim_{i\rightarrow\infty}\mathcal{H}^{n+1}(B_1(p_i))
\endaligned
\end{equation}
holds from Colding \cite{C}, Cheeger-Colding \cite{CCo1}.

Let $\mathcal{R}$ be the regular set of $B_1(p_\infty)$. Namely, a point $x\in \mathcal{R}$ if and only if each tangent cone at $x$ of $B_1(p_\infty)$ is $\R^{n+1}$.
Let $\mathcal{S}=B_1(p_\infty)\setminus\mathcal{R}$ denote the singular set of $B_1(p_\infty)$. Then dim$(\mathcal{S})\le n-1$ from Cheeger-Colding \cite{CCo1}.
In general, $\mathcal{S}$ may not be closed in $B_1(p_\infty)$.
For any $\ep>0$, let $\mathcal{R}_\ep\supset\mathcal{R}$ and $\mathcal{S}_\ep$ be two subsets in $B_1(p_\infty)$ defined by
\begin{equation}\aligned\label{RSep}
\mathcal{R}_\ep=\left\{x\in B_1(p_\infty)\Big|\,\sup_{0<s\le r}s^{-1}d_{GH}(B_s(x),B_s(0))<\ep \ \mathrm{for\ some}\ r>0\right\},\  \mathcal{S}_\ep=B_1(p_\infty)\setminus\mathcal{R}_\ep.
\endaligned
\end{equation}
Here, $B_r(0)$ is the ball in $\R^{n+1}$ centered at the origin with radius $r$.
Let $\La_\ep$ be a positive function on $\mathcal{R}_\ep$ defined by
\begin{equation}\aligned\label{Laep}
\La_\ep(x)=\sup\{r>0|\, d_{GH}(B_{s}(x),B_{s}(0))<\ep s\ \mathrm{for\ each}\ 0<s\le r\}
\endaligned
\end{equation}
for any $x\in \mathcal{R}_\ep$.
From Bishop-Gromov volume comparison and Theorem 0.8 in \cite{C},
$\mathcal{S}_\ep$ is closed in $B_1(p_\infty)$, and $\La_\ep$ is a lower semicontinuous function on $\mathcal{R}_\ep$.

For each integer $0\le k\le n-1$ and constants $\ep>0,r>0$, let $\mathcal{S}^k_{\ep,r}$ be a set consisting of points $y\in B_1(p_\infty)$ such
that for any metric space $X$, $(0^{k+1},x)\in\R^{k+1}\times X$ and $s>r$, we have
$d_{GH}\left(B_s(y),B_s((0^{k+1},x))\right)\ge\ep s.$
Put
\begin{equation}\aligned\label{Sk}
\mathcal{S}^k=\cup_{\ep>0}\cap_{r>0}\mathcal{S}^k_{\ep,r},
\endaligned
\end{equation}
then $\mathcal{S}^{n-1}\setminus\mathcal{S}^{n-2}$ is the top stratum of the singular set and dim$(\mathcal{S}^k)\le k$ from Cheeger-Colding \cite{CCo1}. See Cheeger-Naber \cite{CN13} for further results on the singular sets.
For any $x\in \mathcal{S}^{n-1}\setminus\mathcal{S}^{n-2}$, there is a tangent cone $C_x$ of $B_1(p_\infty)$ at $x$, which splits off $\R^{n-1}$ isometrically.
Moreover, there is a round circle $\mathbb{S}^1_r$ in $\R^2$ with radius $r\in(0,1)$ such that $C_x=C\mathbb{S}^1_r\times\R^{n-1}$.

\section{Volume estimates for area-minimizing hypersurfaces}

From  the Laplacian comparison for distance functions from minimal submanifolds by Heintze-Karcher \cite{HK},
Itokawa-Kobayashi obtained the uniform lower bound for the volume of minimal hypersurfaces
in terms of the volume of their tubular neighborhoods in manifolds of nonnegative Ricci curvature (see Proposition 3.2 in \cite{IK}).
Let $N$ be an $(n+1)$-dimensional smooth complete Riemannian manifold.
Let $\na$ denote Levi-Civita connection of $N$, $\De_N$ denote the Laplacian of $N$.
Using Lemma \ref{DerMge*} in appendix I, we can get the lower bound of the volume of the support of $n$-rectifiable stationary varifolds in $N$ with Ricci curvature bounded below as follows.
\begin{lemma}\label{LOWERM000}
Suppose that $N$ has Ricci curvature $\ge-n\k^2$ on a geodesic ball $B_R(p)\subset N$ for some constant $\k\ge0$.
Let $V$ be an $n$-rectifiable stationary varifold in $B_{R}(p)$. Denote $M=\mathrm{spt} V\cap B_R(p)$. Then
\begin{equation}\aligned\label{tArea}
\mathcal{H}^{n+1}\left(B_t(M)\cap B_{s}(p)\right)\le\f{2t}{1-n\k t}\mathcal{H}^n\left(M\cap B_{t+s}(p)\right)
\endaligned
\end{equation}
for each $0<t\le\min\{\f1{n\k},s\}$ and $s+t<R$.
\end{lemma}
\begin{proof}
Let $\mathcal{R}_{M}$ denote the regular part of $M$, i.e., $\mathcal{R}_{M}=M\setminus \mathcal{S}_M$ with the singular part $\mathcal{S}_M$ of $M$.
Let $\r_M$ denote the distance function from $M$ in $N$.
Let $M_i$ be the smooth embedded hypersurface constructed related to $M$ as in Lemma \ref{DerMge*} of appendix I.
Let $\mathcal{C}_{M_i}$ denote the cut locus of $\r_{M_i}$ for each $i$.
For any $s<R$,
$\overline{\mathcal{C}_{M_i}}\cap B_{R-s}(M_i)\cap B_{s}(p)$ has Hausdorff dimension $\le n$ (see Corollary 4.12 in \cite{MM}). 
Hence by co-area formula, there is a zero $\mathcal{H}^1$-measure set $\La\subset(0,\infty)$ such that for any $t\in(0,R-s)\setminus\La$ and $i\in \mathbb{N}^+$,
$\{\r_{M_i}=t\}\cap B_{R-s}(M_i)\cap B_{s}(p)$ is smooth except a closed set with Hausdorff dimension $\le n-1$.
Let $\ep>0$ be small. From \eqref{equivMiM} there is an integer $i_0>0$ (depending on $s,t$) such that
$$\{\r_{M_i}=t\}\cap \overline{B_{s}(p)}=\{\r_{M}=t\}\cap \overline{B_{s}(p)}$$
for all $i\ge i_0$.
For \textbf{fixed} $s,t$ with $t\in(0,\infty)\setminus\La$ and $0<t\le s<R-t$, there exists a finite collection of balls $\{B_{s_j}(z_j)\}_{j=1}^{N_\ep}$ with $\overline{\mathcal{C}_M}\subset\left(\cup_{j=1}^{N_\ep} B_{s_j}(z_j)\right)$ such that $\{\r_{M}=t\}\cap \overline{B_s(p)}\cap\left(\cup_{j=1}^{N_\ep} B_{s_j}(z_j)\right)$ has $n$-dimensional Hausdorff measure $\le\ep$.
Hence, there is a closed set $\Om_{t,s}$ in $\{\r_{M}=t\}\cap B_s(p)\setminus \left(\cup_{j=1}^{N_\ep} B_{s_j}(z_j)\right)$ with $\p\Om_{t,s}\in C^\infty$
so that
\begin{equation}\aligned\label{HnOmitsrM}
\mathcal{H}^n\left(\{\r_{M}=t\}\cap B_s(p)\right)\le\mathcal{H}^n\left(\Om_{t,s}\right)+2\ep.
\endaligned
\end{equation}

For any $z\in\Om_{t,s}$ there is a unique $x_z\in \mathcal{R}_{M}\cap B_{t+s}(p)$ with $d(x_z,z)=t$.
Let $W$ be a small neighborhood of $x$ in $\Om_{t,s}$ with $\p W\in C^1$, then there is a neighborhood $K_W$ of $x_z$ in $\mathcal{R}_{M}$ such that
$$W=\left\{\mathrm{exp}_{y}(t\mathbf{n}_M(y))\in N|\  y\in K_W \right\},$$
where $\mathbf{n}_M$ denotes the unit normal vector field to $\mathcal{R}_{M}$.
Since the exponential map $\mathrm{exp}_{\cdot}(t\mathbf{n}_M(\cdot)):\, K_W\rightarrow W$ is 1-1 and differentiable,
then $\p K_W\in C^1$ from the inverse mapping theorem and $\p W\in C^1$.
Let $V_{t}$ denote the set defined by
\begin{equation}\aligned\nonumber
V_t=\left\{\mathrm{exp}_{y}(r\mathbf{n}_M(y))\in N|\  0\le r\le t,\ y\in K_W \right\}
\endaligned
\end{equation}
with the outward unit normal $\nu_{t}$ to $\p V_{t}$.
For any $y\in\p V_{t}\setminus V_{t}$, if there is a point $(x_y,\tau)\in\p K_W\times\R$ with $y=\mathrm{exp}_{x_y}(\tau\mathbf{n}_M(x_y))$, then $\na\r_M(y)=\f{\p}{\p \tau}\mathrm{exp}_{x_y}(\tau\mathbf{n}_M(x_y))$,
which implies
\begin{equation}\aligned\label{narMnuts}
\left\lan\na\r_M,\nu_{t}\right\ran=0\qquad \mathrm{on}\ \ \p V_{t}\setminus V_{t}.
\endaligned
\end{equation}
Under the assumption $0<t\le s<R-t$ and $t\in(0,\infty)\setminus\La$, we further assume $t\le\f1{n\k+\ep}$. From Lemma \ref{DerMge*}, we have
\begin{equation}\aligned
&\De_N\left(\r_M-\f1{n\k+\ep}\right)^2=2\left(\r_M-\f1{n\k+\ep}\right)\De_N \r_M+2\\
\ge&-2\left(\f1{n\k+\ep}-\r_M\right)n\k\tanh(\k\r_M)+2\ge0
\endaligned
\end{equation}
on $V_{t}$ in the distribution sense.
With \eqref{narMnuts}, integrating the above inequality by parts infers
\begin{equation}\aligned
0\le&\int_{V_{t}}\De_N\left(\r_M-\f1{n\k+\ep}\right)^2=\int_{\p V_{t}}\left\lan\na\left(\r_M-\f1{n\k+\ep}\right)^2,\nu_{t}\right\ran\\
=&2\left(t-\f1{n\k+\ep}\right)\mathcal{H}^n(W)
+\f2{n\k+\ep}\mathcal{H}^n(K_W),
\endaligned
\end{equation}
which implies
\begin{equation}\aligned\label{WKcont}
\mathcal{H}^n(W)\le\f1{1-(n\k+\ep)t}\mathcal{H}^n(K_W).
\endaligned
\end{equation}

Note that for any point $x\in M\cap B_{t+s}(p)$ with $d(x,\Om_{t,s})=t$, there are 2 points $z_x,z_x'\in\Om_{t,s}$ at most such that $d(x,z_x)=d(x,z_x')=t$.
Hence, by choosing a suitable covering of $\Om_{t,s}$, with \eqref{WKcont} we get
\begin{equation}\aligned\label{rMBRBtR}
\mathcal{H}^n\left(\Om_{t,s}\right)\le\f2{1-(n\k+\ep)t}\mathcal{H}^n\left(M\cap B_{t+s}(p)\right)
\endaligned
\end{equation}
for each $0<t\le s<R-t$ with $t\in(0,\infty)\setminus\La$ and $t\le\f1{n\k+\ep}$.
With co-area formula and \eqref{HnOmitsrM}, for each $0<t\le\min\{\f1{n\k+\ep},s\}$ and $s+t<R$ we have
\begin{equation}\aligned
\mathcal{H}^{n+1}\left(B_t(M)\cap B_{s}(p)\right)=&\int_0^t\mathcal{H}^n\left(\{\r_M=\tau\}\cap B_{s}(p)\right)d\tau
\le\int_0^t\left(\mathcal{H}^n\left(\Om_{\tau,s}\right)+2\ep\right)d\tau\\
\le&\int_0^t\left(\f2{1-(n\k+\ep)\tau}\mathcal{H}^n\left(M\cap B_{\tau+s}(p)\right)+2\ep\right)d\tau\\
\le&\f{2t}{1-(n\k+\ep)t}\mathcal{H}^n\left(M\cap B_{t+s}(p)\right)+2\ep t.
\endaligned
\end{equation}
Letting $\ep\rightarrow0$ completes the proof.
\end{proof}

If $\k>0$ and the manifold $N$ in Lemma \ref{LOWERM000} has Ric$\,\ge-n\k^2$ in $B_{1}(p)$,
then we choose $t=s=\min\{\f12,\f1{2n\k}\}$ in \eqref{tArea}, and obtain
\begin{equation}\aligned\label{nnnArea}
\mathcal{H}^{n+1}(B_{t}(p))\le\min\left\{1,\f1{n\k}\right\}\mathcal{H}^n(M\cap B_1(p)).
\endaligned
\end{equation}
Clearly, the constant $\min\left\{1,\f1{n\k}\right\}$ in the above estimation is not optimal.
If $N$ has nonnegative Ricci curvature in $B_1(p)$, then we choose $t=s=\f12$ in
\eqref{tArea}, and obtain
\begin{equation}\aligned
\mathcal{H}^{n+1}\left(B_{\f1{2}}(p)\right)\le\mathcal{H}^n\left(M\cap B_1(p)\right).
\endaligned
\end{equation}

Let $\Si$ be an $(n+1)$-dimensional complete simply connected non-compact manifold with nonnegative sectional curvature.
Anderson \cite{An} proved that
if there is a constant $\de_n>0$ depending only on $n$ such that $\limsup_{r\rightarrow\infty}r^{-1}\mathrm{diam}(\p B_r(p))<\de_n$, then $\Si$ admits no complete area-minimizing hypersurfaces.
Here, $\mathrm{diam}(\p B_r(p))=\inf_{x,y\in\p B_r(p)}d(x,y)$ denotes the (extrinsic) diameter of the sphere $\p B_r(p)$ in $\Si$. With Lemma \ref{LOWERM000}, we have the following non-existence result.
\begin{theorem}\label{Nonexistminimizinghypersurface}
Let $N$ be an $(n+1)$-dimensional complete simply connected non-compact manifold with nonnegative Ricci curvature.
If
$$\limsup_{r\rightarrow\infty}r^{-1}\mathrm{diam}(\p B_r(p))<\la_n$$
with $\la_n(1+\la_n)^n=1/(n+1)$, then $N$ admits no complete area-minimizing hypersurfaces.
\end{theorem}
\begin{proof}
Suppose that there is a complete connected area-minimizing hypersurface $M$ in $N$.
Let $\mathcal{R}_M$ denote the regular part of $M$.
Note that $\mathcal{R}_M$ is one-sided and connected.
Let $\mathbf{n}_M$ be the unit normal vector field to $\mathcal{R}_M$. Then for any $s>0$
$$\p B_s(M)\subset\{\mathrm{exp}_{x}(s\mathbf{n}_M(x))|\ x\in \mathcal{R}_{M}\}\cup\{\mathrm{exp}_{x}(-s\mathbf{n}_M(x))|\ x\in \mathcal{R}_{M}\}.$$
For each $\tau>0$, we define two open sets $B^+_{\tau}(M)$ and $B^-_{\tau}(M)$ by
\begin{equation}\aligned\nonumber
B^\pm_\tau(M)=B_\tau(M)\cap\{\mathrm{exp}_{x}(\pm s\mathbf{n}_M(x))|\ x\in \mathcal{R}_{M}, 0<s<r\}.
\endaligned
\end{equation}
Fix a point $p\in M$, and $r>0$. We claim
\begin{equation}\aligned\label{N+-tauM}
B^+_{\tau}(M)\cap B^-_{\tau}(M)\cap B_r(p)=\emptyset\qquad \mathrm{for\ each}\ \tau>0.
\endaligned
\end{equation}
If the claim \eqref{N+-tauM} fails, then there are a constant $\tau>0$ and a point $q\in\overline{B^+_{\tau}(M)}\cap \overline{B^-_{\tau}(M)}\cap B_r(p)$ but $B^+_{\tau}(M)\cap B^-_{\tau}(M)\cap B_r(p)=\emptyset$.
By definition, there are two points $x^+_q,x^-_q\in M$ such that
$$q=\mathrm{exp}_{x^+_q}\left(\tau\mathbf{n}_M(x^+_q)\right)=\mathrm{exp}_{x^-_q}\left(-\tau\mathbf{n}_M(x^-_q)\right).$$
Hence, there are smooth embedded curves $\g_\pm\subset B^\pm_{\tau}(M)$ connecting $p,q$.
Clearly, $\overline{\g_+\cup\g_-}$ is a closed piecewise smooth curve in $N$ with $M\cap\overline{\g_+\cup\g_-}=\{p\}$. Namely, the intersection number (mod 2) of $\overline{\g_+\cup\g_-}$ with $M$ is 1.
Therefore, $\overline{\g_+\cup\g_-}$ is homotopic nontrivial, which contradicts to the simply connected $N$. This gives the proof of the claim \eqref{N+-tauM}.

Denote $E_r=B^+_{2r}(M)\cap B_r(p)$. The claim \eqref{N+-tauM} infers
$$\p E_r=(M\cap B_r(p))\cup (B^+_{2r}(M)\cap\p B_r(p)).$$
Then $\p (M\cap B_r(p))=\p(B^+_{2r}(M)\cap\p B_r(p))$, and the area-minimizing hypersurface $M$ in $N$ implies
\begin{equation}\aligned
\mathcal{H}^{n}\left(M\cap B_r(p)\right)\le\mathcal{H}^n\left(B^+_{2r}(M)\cap\p B_r(p)\right).
\endaligned
\end{equation}
Analogously, we have $\mathcal{H}^{n}\left(M\cap B_r(p)\right)\le\mathcal{H}^n\left(B^-_{2r}(M)\cap\p B_r(p)\right)$.
With \eqref{N+-tauM}, we get
\begin{equation}\aligned\label{McapBrpupbound}
\mathcal{H}^{n}\left(M\cap B_r(p)\right)\le&\min\left\{\mathcal{H}^n\left(B^+_{2r}(M)\cap\p B_r(p)\right),\mathcal{H}^n\left(B^-_{2r}(M)\cap\p B_r(p)\right)\right\}\\
\le&\f12\mathcal{H}^n\left(\p B_r(p)\right).
\endaligned
\end{equation}
Let $\la=\limsup_{r\rightarrow\infty}r^{-1}\mathrm{diam}(\p B_r(p))$.
Then for each $\ep>0$ there is a constant $R_0>0$ such that
$$ B_R(p)\subset B_{(\la+\ep)R}(M)$$
for each $R\ge R_0$. From \eqref{tArea} we have
\begin{equation}\aligned\label{BRpcontr}
\mathcal{H}^{n+1}\left(B_R(p)\right)=&\mathcal{H}^{n+1}\left(B_{(\la+\ep)R}(M)\cap B_{R}(p)\right)
\le2(\la+\ep)R\mathcal{H}^n\left(M\cap B_{(1+\la+\ep)R}(p)\right).
\endaligned
\end{equation}
From \eqref{HnpBr} and \eqref{McapBrpupbound}, we have
\begin{equation}\aligned\label{MBRpcontr}
\mathcal{H}^n\left(M\cap B_{(1+\la+\ep)R}(p)\right)\le&\f12 \mathcal{H}^n(\p B_{(1+\la+\ep)R}(p))\\
\le&\f{n+1}{2(1+\la+\ep)R}\mathcal{H}^{n+1}\left(B_{(1+\la+\ep)R}(p)\right).
\endaligned\end{equation}
Combining \eqref{BiVolN}\eqref{BRpcontr}\eqref{MBRpcontr}, we get
\begin{equation}\aligned
\mathcal{H}^{n+1}\left(B_R(p)\right)\le&\f{(n+1)(\la+\ep)}{1+\la+\ep}\mathcal{H}^{n+1}\left(B_{(1+\la+\ep)R}(p)\right)\\
\le&(n+1)(\la+\ep)(1+\la+\ep)^n\mathcal{H}^{n+1}\left(B_{R}(p)\right).
\endaligned
\end{equation}
Since $\ep$ is arbitrary, then the above inequality implies
\begin{equation}\aligned
\la(1+\la)^n\ge\f1{n+1}.
\endaligned
\end{equation}
This completes the proof.
\end{proof}

Using the universal covers of ambient manifolds, the volume of area-minimizing hypersurfaces in geodesic balls has a upper bound by an explicit constant as follows.
\begin{lemma}\label{upbdareaM}
Let $N$ be an $(n+1)$-dimensional smooth complete Riemannian manifold with Ricci curvature $\ge-n\k^2$ in a geodesic ball $B_1(p)$ for some constant $\k\ge0$. If $M$ is an area-minimizing hypersurface in $B_1(p)$ with $\p M\subset\p B_1(p)$, then there is a constant $c_{n,\k}=\f12(n+1)\omega_{n+1}\k^{-n}\sinh^n\k$ so that
\begin{equation}\aligned\label{UNBOUNDn}
\mathcal{H}^n(M\cap B_t(p))\le c_{n,\k}t^n\qquad \mathrm{for\ any}\ t\in(0,1].
\endaligned
\end{equation}
\end{lemma}
\begin{proof}
By scaling, we only need to prove \eqref{UNBOUNDn} for $t=1$.
Let $\widetilde{N}$ be a universal cover of $N$ with the induced metric from $N$.
In other words, there is a Riemannian covering map $\pi$ from simply-connected manifold $\widetilde{N}$ to $N$.
Let $p'$ be a point in $\widetilde{N}$ with $\pi(p')=p$. Clearly, $\pi(B_1(p'))\subset B_1(p)$.
For any geodesic segment $\g:\ [0,1]\rightarrow \overline{B_1(p)}$ with $\g(0)=p$ and $\g(1)\in\p B_1(p)$, let $\g'$ be a component of $\pi^{-1}(\g)$ containing $p'$. Obviously, $\g'\subset \overline{B_1(p')}$ and $\g'(1)\in\p B_1(p')$.
Hence we get $\pi(B_1(p'))=B_1(p)$, and $\p B_1(p)\subset \pi(\p B_1(p'))$.
In particular, $\widetilde{N}$ is an $(n+1)$-dimensional smooth complete Riemannian manifold with Ricci curvature $\ge-n\k^2$ in $B_1(p')$.

Let $M$ be an area-minimizing hypersurface in $B_1(p)$ with $\p M\subset\p B_1(p)$. From $M\subset \pi(B_1(p'))$ and $\p M\subset \pi(\p B_1(p'))$ we can choose a subset set $M'$ of $\pi^{-1}(M)\subset\widetilde{N}$ such that $M'\subset B_1(p')$ and $\pi:\ M'\rightarrow M$ is 1-1. Then $M'$ is countably $n$-rectifiable, $\p M'\cap B_1(p')=\emptyset$, and
$\mathcal{H}^n(M)=\mathcal{H}^n(M').$
We claim that
\begin{center}
$M'$ is area-minimizing in $B_1(p')$.
\end{center}
Clearly, $M'$ is one-sided outside its singular set.
Let $T$ be the multiplicity one current in $B_1(p')$ with $M'=\mathrm{spt}T\cap B_1(p')$. Let $S$ be a current in $\widetilde{N}$
with $\p S=\p T$, spt$(S-T)\subset B_1(p')$.
Let $\pi^\#S$ be the pushing forward of $S$ defined by $\pi^\#S(\omega)=S(\z\pi^\#\omega)$ for any smooth $n$-form $\omega\in \mathcal{D}^n(B_1(p))$,
where $\z$ is any function in $C_c^\infty(\widetilde{N})$ with $\z=1$ in a neighborhood of spt$S\cap\mathrm{spt}(\pi^\#\omega)$ (see \cite{S} for further details).
Denote $\pi^\#T$ be the pushing forward of $T$.
Since $\pi:\ M'\rightarrow M$ is 1-1, then $\pi^\#T$ is the multiplicity one minimizing current in $B_1(p)$ with $M=\mathrm{spt}(\pi^\#T)\cap B_1(p)$. From $\pi(B_1(p'))=B_1(p)$ and spt$(S-T)\subset B_1(p')$, we get
spt$(\pi^\#S-\pi^\#T)=\mathrm{spt}(\pi^\#(S-T))=\pi(\mathrm{spt}(S-T))\subset\pi(B_1(p'))=B_1(p)$.
It's clear that $\p(\pi^\#S)=\pi^\#(\p S)=\pi^\#(\p T)=\p(\pi^\#T)$. Note that $\pi$ is a mapping with the Lipschitz constant $\mathbf{Lip}\pi=1$. Hence
$$\mathbb{M}(T)=\mathcal{H}^n(M)=\mathbb{M}(\pi^\#T)\le \mathbb{M}(\pi^\#S)\le\mathbf{Lip}\pi\,\mathbb{M}(S)=\mathbb{M}(S),$$
which means that $M'$ is area-minimizing in $B_1(p')$.

From the argument of the proof of Theorem \ref{Nonexistminimizinghypersurface}, there is a set $E$ in $B_1(p')$ such that $\p E\cap B_1(p')=M'$ and $\mathcal{H}^n\left(\p E\cap B_1(p')\right)\le\f12\mathcal{H}^n\left(\p B_1(p')\right)$. Then with \eqref{HnpBr} we have
\begin{equation}\aligned
\mathcal{H}^n(M)=\mathcal{H}^n(M')\le\f12\mathcal{H}^n(\p B_1(p'))\le\f12(n+1)\omega_{n+1}\f{\sinh^n\k}{\k^n},
\endaligned
\end{equation}
which completes the proof.
\end{proof}
\begin{remark}
In general, $\mathcal{H}^n(M)$ may not be sufficiently small when $\mathcal{H}^{n+1}(B_1(p))$ is sufficiently small in Lemma \ref{upbdareaM}.
For any $\ep>0$, let $N_\ep$ be a cylinder $\mathbb{S}^1(\ep)\times\R^n\subset\R^{n+2}$ with the metric induced from $\R^{n+2}$, and $M_\ep=\{\th_\ep\}\times\R^n$ for $\th_\ep\in\mathbb{S}^1(\ep)$. Denote $p_\ep=(\th_\ep,0^n)\in N_\ep$, and $B_r(p_\ep)$ be the geodesic ball in $N_\ep$ with radius $r$ and centered at $p_\ep$.
Then $\mathcal{H}^n(M_\ep\cap B_1(p_\ep))=\omega_n$ and $2\pi\ep\omega_n(1-\pi\ep)^n<\mathcal{H}^{n+1}(B_1(p_\ep))<2\pi\ep\omega_n$.
Since $N_\ep$ is flat everywhere, then for any minimal hypersurface $S_\ep\subset N_\ep$ with $p_\ep\in S_\ep$
we have $\mathcal{H}^n(S_\ep\cap B_r(p_\ep))\ge\omega_nr^n$ from the monotonicity of $r^{-n}\mathcal{H}^n(S_\ep\cap B_r(p_\ep))$.
Hence, $M_\ep$ is area-minimizing in $N_\ep$ for each $\ep>0$.
\end{remark}

Let us estimate the lower bound of the volume for a class of sets associated with area-minimizing hypersurfaces, which will be needed in Theorem \ref{Conv0}.
\begin{lemma}\label{lbd}
Let $N$ be an $(n+1)$-dimensional smooth non-compact Riemannian manifold with $Ric\ge-n\k^2$ on $B_1(p)$.
There is a constant $\a^*_{n,\k}>0$ depending only on $n,\k$ such that if $E$ is an open set in $B_1(p)\subset N$ with $p\in \p E$, and $\p E\cap B_1(p)$ is an area-minimizing hypersurface in $B_1(p)$, then
$$\mathcal{H}^{n+1}(E)\ge \a^*_{n,\k}\mathcal{H}^{n+1}(B_1(p)).$$
\end{lemma}
\begin{proof}
Let $E_{s}=E\cap\p B_s(p)$ for any $s\in(0,1]$, then
$$\p E_{s}=\p E\cap\p B_s(p)=\p(B_s(p)\cap\p E).$$
Since $\p E\cap B_1(p)$ is area-minimizing in $B_1(p)$, then
\begin{equation}\aligned
\mathcal{H}^{n}(E_{s})\ge\mathcal{H}^{n}(B_s(p)\cap\p E).
\endaligned
\end{equation}
From co-area formula,
$$\mathcal{H}^{n+1}(B_t(p)\cap E)=\int_0^t\mathcal{H}^{n}(E_{s})ds\qquad \mathrm{for\ any}\ t\in(0,1].$$
By Bishop-Gromov volume comparison \eqref{BiVolN} and the isoperimetric inequality \eqref{isoperi}, we have
\begin{equation}\aligned\label{ODEEsx}
\widetilde{\a}_{n,\k}\left(\mathcal{H}^{n+1}(B_1(p))\right)^{\f1{n+1}}&\left(\int_0^t\mathcal{H}^{n}(E_{s})ds\right)^{\f{n}{n+1}}\le\mathcal{H}^{n}(\p(B_t(p)\cap E))\\
\le&\mathcal{H}^{n}(E_{t})+\mathcal{H}^{n}(B_t(p)\cap \p E)\le2\mathcal{H}^{n}(E_{t})
\endaligned\end{equation}
for any $t\in(0,\f12]$.
Here, $\widetilde{\a}_{n,\k}=\a_{n,\k}\left(V_{\k}^{n+1}(1)\right)^{-1/(n+1)}$, and $V^{n+1}_\k(1)$ denotes the volume of a unit geodesic ball in an $(n+1)$-dimensional space form with constant sectional curvature $-\k^2$.
From $p\in \p E$, one has $\int_0^t\mathcal{H}^{n}(E_{s})ds>0$ for each $t\in(0,1]$. From the ODE \eqref{ODEEsx}, we get
\begin{equation}
\f{\p}{\p t}\left(\int_0^t\mathcal{H}^{n}(E_{s})ds\right)^{\f{1}{n+1}}\ge\f{\widetilde{\a}_{n,\k}}{2(n+1)}\left(\mathcal{H}^{n+1}(B_1(p))\right)^{\f1{n+1}},
\end{equation}
which gives
\begin{equation}\label{EsxB1p}
\left(\int_0^t\mathcal{H}^{n}(E_{s})ds\right)^{\f{1}{n+1}}\ge\f{\widetilde{\a}_{n,\k}t}{2(n+1)}\left(\mathcal{H}^{n+1}(B_1(p))\right)^{\f1{n+1}}
\end{equation}
for any $t\in(0,\f12]$.
In particular,
\begin{equation}
\mathcal{H}^{n+1}(E)\ge\int_0^{\f12}\mathcal{H}^{n}(E_{s})ds\ge\left(\f{\widetilde{\a}_{n,\k}}{4(n+1)}\right)^{n+1}\mathcal{H}^{n+1}(B_1(p)),
\end{equation}
which completes the proof.
\end{proof}
\textbf{Remark.} The set $E$ in the above lemma is said to be a \emph{minimal set} in some terminology, such as \cite{Gi}.

\section{Convergence for area-minimizing hypersurfaces in geodesic balls}

Let $N_i$ be a sequence of $(n+1)$-dimensional smooth Riemannian manifolds with $\mathrm{Ric}\ge-n\k^2$ on the metric ball $B_{1+\k'}(p_i)\subset N_i$ for constants $\k\ge0$, $\k'>0$.
Up to a choice of the subsequence, we assume that $\overline{B_1(p_i)}$ converges to a metric ball $\overline{B_1(p_\infty)}$ in the Gromov-Hausdorff sense.
Namely, there is a sequence of $\ep_i$-Hausdorff approximations $\Phi_i:\, B_1(p_i)\rightarrow B_1(p_\infty)$ for some sequence $\ep_i\rightarrow0$.
Let $\nu_\infty$ denote the renormalized limit measure on $B_1(p_\infty)$ obtained from the renormalized measures as \eqref{nuinfty}.
For any set $K$ in $\overline{B_1(p_\infty)}$, let $B_\de(K)$ be the $\de$-neighborhood of $K$ in $\overline{B_1(p_\infty)}$ defined by $\{y\in \overline{B_1(p_\infty)}|\ d(y,K)<\de\}$. Here, $d$ denotes the distance function on $\overline{B_1(p_\infty)}$.

Let us state the continuity for measure of $\de$-neighborhood of sets in the Gromov-Hausdorff topology, which will be needed in a sequel.
\begin{lemma}\label{UpMinfty}
Let $F_i$ be a sequence of sets in $B_1(p_i)$. If $F_i$ converges to a closed set $F_\infty\subset \overline{B_1(p_\infty)}$ in the induced Hausdorff sense, then for each $t\in(0,1)$ and $\de\in(0,1)$ one has
$$\nu_\infty\left(B_\de(F_\infty)\cap B_t(p_\infty)\right)=\lim_{i\rightarrow\infty}\mathcal{H}^{n+1}(B_\de(F_i)\cap B_t(p_i))/\mathcal{H}^{n+1}\left(B_{1}(p_i)\right).$$
\end{lemma}
\begin{proof}
For each fixed $\tau\in(0,\min\{t,\de\})$,
from Lemma \ref{nuinftyKthj} in appendix II, there is a sequence of mutually disjoint balls $\{B_{\th_j}(x_j)\}_{j=1}^\infty$ with
$x_j\in \overline{B_{\de-\tau}(F_\infty)}\cap \overline{B_{t-\tau}(p_\infty)}$ and $\th_j<\tau$
such that $\overline{B_{\de-\tau}(F_\infty)}\cap \overline{B_{t-\tau}(p_\infty)}\subset \bigcup_{1\le j\le k}B_{\th_j+2\th_k}(x_j)$ for the sufficiently large $k$, and
\begin{equation}\aligned\label{430}
\nu_\infty\left(\overline{B_{\de-\tau}(F_\infty)}\cap \overline{B_{t-\tau}(p_\infty)}\right)\le\sum_{j=1}^{\infty}\nu_\infty\left(B_{\th_j}(x_j)\right).
\endaligned
\end{equation}
Note that $\Phi_i(F_i)$ converges to $F_\infty$ in the Hausdorff sense.
For each $j\ge0$, there is a sequence of points $x_{i,j}\in B_{\de'}(F_i)$ with $\lim_{i\rightarrow\infty}x_{i,j}=x_j$.
From \eqref{nuinfty},
\begin{equation}\aligned
\lim_{i\rightarrow\infty}\mathcal{H}^{n+1}\left(B_{\th_j}(x_{i,j})\right)/\mathcal{H}^{n+1}\left(B_{1}(p_i)\right)=\nu_\infty\left(B_{\th_j}(x_{j})\right).
\endaligned
\end{equation}
Now we fix an integer $m>0$.
By the selection of $B_{\th_j}(x_j)$, we can require $B_{\th_j}(x_{i,j})\cap B_{\th_k}(x_{i,k})=\emptyset$ for $j\neq k$ and $j,k\le m$ provided $i$ is sufficiently large.
Hence combining $B_{\th_j}(x_{i,j})\subset B_\de(F_i)\cap B_t(p_i)$, \eqref{430} implies
\begin{equation}\aligned\label{wnthjn}
\sum_{j=1}^{m}\nu_\infty\left(B_{\th_j}(x_{j})\right)\le\mathcal{H}^{n+1}(B_\de(F_i)\cap B_t(p_i))/\mathcal{H}^{n+1}\left(B_{1}(p_i)\right)
\endaligned
\end{equation}
for the sufficiently large $i$.
Letting $i\rightarrow\infty$ first, then letting $m\rightarrow\infty$ infers
\begin{equation}\aligned\label{432}
\sum_{j=1}^{\infty}\nu_\infty\left(B_{\th_j}(x_{j})\right)\le\liminf_{i\rightarrow\infty}\mathcal{H}^{n+1}(B_\de(F_i)\cap B_t(p_i))/\mathcal{H}^{n+1}\left(B_{1}(p_i)\right).
\endaligned
\end{equation}
Combining \eqref{430}\eqref{432} we have
\begin{equation}\aligned\label{436}
\nu_\infty\left(\overline{B_{\de-\tau}(F_\infty)}\cap \overline{B_{t-\tau}(p_\infty)}\right)\le\liminf_{i\rightarrow\infty}\mathcal{H}^{n+1}(B_\de(F_i)\cap B_t(p_i))/\mathcal{H}^{n+1}\left(B_{1}(p_i)\right).
\endaligned
\end{equation}
Letting $\tau\rightarrow0$ implies
$$\nu_\infty\left(B_\de(F_\infty)\cap B_t(p_\infty)\right)\le\liminf_{i\rightarrow\infty}\mathcal{H}^{n+1}(B_\de(F_i)\cap B_t(p_i))/\mathcal{H}^{n+1}\left(B_{1}(p_i)\right).$$
Combining Lemma \ref{Cont*} in appendix II, we complete the proof.
\end{proof}

From now on, we further assume
\begin{equation}\aligned\label{VOL*}
\mathcal{H}^{n+1}(B_1(p_i))\ge v\qquad \mathrm{for\ each}\ i\ge1\ \mathrm{and\ some\ constant}\ v>0.
\endaligned
\end{equation}
For each $i$, let $M_i$ be an area-minimizing hypersurface in $B_1(p_i)$ with $\p M_i\subset\p B_1(p_i)$.
Suppose that $M_i$ converges to a closed set $M_\infty\subset \overline{B_1(p_\infty)}$ in the induced Hausdorff sense.
For any $0<t<1$ and $0<\de\le\min\{\f1{n\k},t-\de\}$, from \eqref{VolCOV}, Lemma \ref{LOWERM000} and Lemma \ref{UpMinfty} we have
\begin{equation}\aligned\label{Hn1Bdede*}
\mathcal{H}^{n+1}\left(B_{\de}(M_\infty)\cap B_{t-\de}(p_\infty)\right)=&\liminf_{i\rightarrow\infty}\mathcal{H}^{n+1}\left(B_\de(M_i)\cap B_{t-\de}(p_i)\right)\\
\le&\f{2\de}{1-n\k\de}\liminf_{i\rightarrow\infty}\mathcal{H}^n\left(M_i\cap B_{t}(p_i)\right).
\endaligned
\end{equation}
From the definition of Minkowski contents and \eqref{Hn1Bdede*}, for any $0<t'<t<1$ we immediately have
\begin{equation}\aligned\label{nu*MinftyBtp}
\mathcal{M}^*\left(M_\infty,B_{t'}(p_\infty)\right)\le\limsup_{\de\rightarrow0}\f{\mathcal{H}^{n+1}\left(B_{\de}(M_\infty)\cap B_{t-\de}(p_\infty)\right)}{2\de}
\le\liminf_{i\rightarrow\infty}\mathcal{H}^n\left(M_i\cap B_{t}(p_i)\right).
\endaligned
\end{equation}

The upper Minkowski content of the limit $M_\infty$ preserves the minimizing property under some suitable conditions in the following sense.
Let $S_i$ be a subset in $B_1(p_i)$ such that $S_i\setminus M_i$ is smooth embedded, $S_i=M_i$ outside $B_t(p_i)$
and $\p \llbracket S_i\rrbracket=\p\llbracket M_i\rrbracket$.
For each $i$, let $\mathcal{R}_{M_i}$ denote the regular part of $M_i$, and $\mathcal{R}_{S_i}=(S_i\setminus M_i)\cup(S_i\cap \mathcal{R}_{M_i})$.
Clearly, $\mathcal{R}_{S_i}$ is one-sided and connected. Let $\mathbf{n}_{S_i}$ be the unit normal vector field to $\mathcal{R}_{S_i}$.
For each $r\in(0,1-t)$, we define two open sets $B^\pm_{r}(S_i)$ by
\begin{equation}\aligned\nonumber
B^\pm_{r}(S_i)=B_r(S_i)\cap\{\mathrm{exp}_{x}(\pm s\mathbf{n}_{S_i}(x))|\ x\in \mathcal{R}_{S_i},\, 0<s<r\}.
\endaligned
\end{equation}
Clearly, $B^+_{r}(S_i)\cap B^-_{r}(S_i)\cap S_i=\emptyset$, and
$$B_{r}(S_i)\cap B_{1-r}(p_i)=(B^+_{r}(S_i)\cup B^-_{r}(S_i)\cup S_i)\cap B_{1-r}(p_i).$$
\begin{proposition}\label{MinMSinfty}
Suppose that there is a constant $\th\in(0,t)$ such that
$S_i,B^+_{\th}(S_i),B^-_{\th}(S_i)$ converge to closed sets $S_\infty,\G^+_{\th,\infty},\G^-_{\th,\infty}\subset \overline{B_1(p_\infty)}$ in the induced Hausdorff sense, respectively.
If $\G^+_{\th,\infty}\cap\G^-_{\th,\infty}\subset S_\infty$, then
\begin{equation}\aligned\label{MinftytSinfty}
\mathcal{M}^*\left(M_\infty,B_t(p_\infty)\right)\le\mathcal{M}^*\left(S_\infty,B_t(p_\infty)\right),\quad
\mathcal{M}_*\left(M_\infty,B_t(p_\infty)\right)\le\mathcal{M}_*\left(S_\infty,B_t(p_\infty)\right).
\endaligned
\end{equation}
\end{proposition}
\begin{proof}
For any $\tau\in(t',1)$ and $\de\in(0,\min\{\f1{2n\k},\tau,1-\tau\})$, from \eqref{tArea} and Lemma \ref{upbdareaM} we have
\begin{equation}\aligned\label{Hn+1BdeMitaupi}
\mathcal{H}^{n+1}\left(B_\de(M_i)\cap B_{\tau}(p_i)\right)\le\f{2\de}{1-n\k\de}\mathcal{H}^n\left(M_i\cap B_{\tau+\de}(p_i)\right)\le \f{2\de c_{n,\k}}{1-n\k\de}.
\endaligned
\end{equation}
From co-area formula, there is a constant $t_*\in[\f{1+3t'}4,\f{3+t'}4]$ such that
\begin{equation}\aligned\label{410*****}
\mathcal{H}^{n+1}\left(B_\de(M_i)\cap B_{t_*+\sqrt{\de}}(p_i)\setminus B_{t_*}(p_i)\right)\le\f12\de^{4/3}
\endaligned
\end{equation}
for the sufficiently small $0<\de<\f1{2n\k}$.
Combining \eqref{410*****} and co-area formula, there is a constant $\tilde{\de}\in[\de,\sqrt{\de}]$ such that
\begin{equation}\aligned
\mathcal{H}^{n}\left(B_\de(M_i)\cap\p B_{t_*+\tilde{\de}}(p_i)\right)\le&\f1{\sqrt{\de}-\de}\mathcal{H}^{n+1}\left(B_\de(M_i)\cap B_{t_*+\sqrt{\de}}(p_i)\setminus B_{t_*+\de}(p_i)\right)\\
\le&\f{\de^{4/3}}{2(\sqrt{\de}-\de)}\le\de^{5/6}.
\endaligned
\end{equation}
Note that $B_{\de}(S_i)=B_{\de}(M_i)$ outside $B_{t'}(p_i)$ for $0<\de<t'-t$. Hence
\begin{equation}\aligned\label{HnBdeSit*pi}
\mathcal{H}^{n}\left(B_\de(S_i)\cap\p B_{t_*+\tilde{\de}}(p_i)\right)=\mathcal{H}^{n}\left(B_\de(M_i)\cap\p B_{t_*+\tilde{\de}}(p_i)\right)\le\de^{5/6}.
\endaligned
\end{equation}

By the definition of $B^\pm_r(S_i)$, $S_i\cap B_{t_*+\tilde{\de}}(p_i)\subset \p\left(B^\pm_r(S_i)\cap B_{t_*+\tilde{\de}}(p_i)\right)$
and
$$\p\left(B^\pm_r(S_i)\cap B_{t_*+\tilde{\de}}(p_i)\right)=\left(\p(B^\pm_r(S_i))\cap B_{t_*+\tilde{\de}}(p_i)\right)\cup\left(B^\pm_r(S_i)\cap\p B_{t_*+\tilde{\de}}(p_i)\right).$$
From the minimizing property of $M_i$ and $S_i=M_i$ outside $B_t(p_i)$, we get
\begin{equation*}\aligned
&\mathcal{H}^n\left(M_i\cap B_{t_*+\tilde{\de}}(p_i)\right)\le\mathcal{H}^n\left(\p(B^+_r(S_i))\cap B_{t_*+\tilde{\de}}(p_i)\setminus S_i\right)
+\mathcal{H}^n\left(B^+_r(S_i)\cap\p B_{t_*+\tilde{\de}}(p_i)\right)\\
&\mathcal{H}^n\left(M_i\cap B_{t_*+\tilde{\de}}(p_i)\right)\le\mathcal{H}^n\left(\p(B^-_r(S_i))\cap B_{t_*+\tilde{\de}}(p_i)\setminus S_i\right)
+\mathcal{H}^n\left(B^-_r(S_i)\cap\p B_{t_*+\tilde{\de}}(p_i)\right)
\endaligned.
\end{equation*}
From $\G^+_{\th,\infty}\cap\G^-_{\th,\infty}\subset S_\infty$, there is an integer $i_0>0$ so that
$\p(B_r(S_i))\cap B_{t_*+\tilde{\de}}(p_i)=\left(\p(B^+_r(S_i))\cup\p(B^-_r(S_i))\right)\cap B_{t_*+\tilde{\de}}(p_i)\setminus S_i$ for all $r\in[\de^2,\de]$ and all $i\ge i_0$.
Hence, from the above two inequalities we get
\begin{equation}\aligned\label{2HMiBt*de}
2\mathcal{H}^n\left(M_i\cap B_{t_*+\tilde{\de}}(p_i)\right)\le\mathcal{H}^n\left(\p(B_r(S_i))\cap B_{t_*+\tilde{\de}}(p_i)\right)
+\mathcal{H}^n\left(B_r(S_i)\cap\p B_{t_*+\tilde{\de}}(p_i)\right)
\endaligned
\end{equation}
for all $r\in[\de^2,\de]$ and all $i\ge i_0$.

With co-area formula and \eqref{Hn+1BdeMitaupi}\eqref{HnBdeSit*pi}\eqref{2HMiBt*de}, for each $i\ge i_0$ we have
\begin{equation}\aligned
&\mathcal{H}^{n+1}\left(B_{\de-\de^2}(S_i)\cap B_{t_*+\tilde{\de}}(p_i)\right)=\int_0^{\de-\de^2}\mathcal{H}^n\left(\p(B_\tau(S_i))\cap B_{t_*+\tilde{\de}}(p_i)\right)d\tau\\
\ge&\int_{\de^2}^{\de-\de^2}\left(2\mathcal{H}^n\left(M_i\cap B_{t_*+\tilde{\de}}(p_i)\right)-\mathcal{H}^n\left(B_\tau(S_i)\cap\p B_{t_*+\tilde{\de}}(p_i)\right)\right)d\tau\\
\ge&2(\de-2\de^2)\mathcal{H}^n\left(M_i\cap B_{t_*+\tilde{\de}}(p_i)\right)-\de^{\f{11}6}\\
\ge&(1-2\de)(1-n\k\de)\mathcal{H}^{n+1}\left(B_\de(M_i)\cap B_{t_*}(p_i)\right)-\de^{\f{11}6}.
\endaligned
\end{equation}
Hence combining \eqref{410*****}, we obtain
\begin{equation}\aligned\label{Bde-de2Side43116}
&\mathcal{H}^{n+1}\left(B_{\de-\de^2}(S_i)\cap B_{t_*}(p_i)\right)+\f12\de^{4/3}\\
\ge&\mathcal{H}^{n+1}\left(B_{\de-\de^2}(S_i)\cap B_{t_*}(p_i)\right)+\mathcal{H}^{n+1}\left(B_\de(S_i)\cap B_{t_*+\sqrt{\de}}(p_i)\setminus B_{t_*}(p_i)\right)\\
\ge&\mathcal{H}^{n+1}\left(B_{\de-\de^2}(S_i)\cap B_{t_*+\sqrt{\de}}(p_i)\right)\\
\ge&(1-2\de)(1-n\k\de)\mathcal{H}^{n+1}\left(B_\de(M_i)\cap B_{t_*}(p_i)\right)-\de^{\f{11}6}.
\endaligned
\end{equation}
From \eqref{VolCOV} and Lemma \ref{UpMinfty}, letting $i\to\infty$ in \eqref{Bde-de2Side43116} implies
\begin{equation}\aligned
(1-2\de)(1-n\k\de)\mathcal{H}^{n+1}\left(B_\de(M_\infty)\cap B_{t_*}(p_\infty)\right)\le\mathcal{H}^{n+1}\left(B_\de(S_\infty)\cap B_{t_*}(p_\infty)\right)+2\de^{\f{11}6}.
\endaligned
\end{equation}
As $S_i=M_i$ outside $B_t(p_i)$,
\begin{equation}\aligned\label{deHn+1BdeMt'*}
(1-2\de)(1-n\k\de)\mathcal{H}^{n+1}\left(B_\de(M_\infty)\cap B_{t'}(p_\infty)\right)\le\mathcal{H}^{n+1}\left(B_\de(S_\infty)\cap B_{t'}(p_\infty)\right)+2\de^{\f{11}6}.
\endaligned
\end{equation}
Forcing $t'\rightarrow t$ in \eqref{deHn+1BdeMt'*} implies
\begin{equation}\aligned\label{deHn+1BdeMt'**}
(1-2\de)(1-n\k\de)\mathcal{H}^{n+1}\left(B_\de(M_\infty)\cap B_{t}(p_\infty)\right)\le\mathcal{H}^{n+1}\left(B_\de(S_\infty)\cap B_{t}(p_\infty)\right)+2\de^{\f{11}6}.
\endaligned
\end{equation}
We divide $2\de$ on both sides of \eqref{deHn+1BdeMt'**} and get \eqref{MinftytSinfty} by letting $\de\rightarrow0$.
\end{proof}

From Lemma \ref{LOWERM000} and Lemma \ref{upbdareaM}, we can get upper and positive lower bounds for the Hausdorff measure of $M_\infty$.
\begin{lemma}\label{UPPLOWMinfty}
There is a constant $c_{n,\k}^*>0$ depending only on $n,\k$ such that
\begin{equation}\aligned\label{HnBrpinfUP}
\mathcal{H}^n\left(M_\infty\cap B_r(p_\infty)\right)\le c_{n,\k}^*r^n\qquad \mathrm{for\ any}\ r\in(0,1].
\endaligned
\end{equation}
Moreover, if we further suppose $p_\infty\in M_\infty$, then there is a constant $\de_{n,\k,v}>0$ depending only on $n,\k,v$ such that
\begin{equation}\aligned\label{lowerbdMinftyr}
\mathcal{H}^n\left(M_\infty\cap B_r(p_\infty)\right)\ge \de_{n,\k,v}r^n&\qquad \mathrm{for\ any}\ r\in(0,1].
\endaligned
\end{equation}
\end{lemma}
\begin{proof}
For any $\ep\in(0,\f12)$ and $r\in(0,1-2\ep)$, let $\{B_{s_j}(x_j)\}_{j=1}^{N_\ep}$ be a covering of $M_\infty\cap \overline{B_r(p_\infty)}$ with all $s_j<\min\{\ep,(\k+1)^{-1}\}$ such that
\begin{equation}\aligned\label{HnMinftyCOVER}
\mathcal{H}^n\left(M_\infty\cap \overline{B_r(p_\infty)}\right)\le \omega_n\sum_{j=1}^{N_\ep}s_j^n+\ep.
\endaligned
\end{equation}
We assume $M_\infty\cap B_{s_j}(x_j)\neq\emptyset$ and $w_j\in M_\infty\cap B_{s_j}(x_j)$. Then $\{B_{2s_j}(w_j)\}_{j=1}^{N_\ep}$ is a covering of $M_\infty\cap\overline{B_r(p_\infty)}$.
By a covering lemma, up to a choice of the subsequence we may assume that $B_{\f25s_j}(w_j)$ and $B_{\f25s_{j'}}(w_{j'})$ are disjoint for all distinct $j,j'$.
Let $w_{i,j}\in M_i\cap B_r(p_i)$ be a sequence of points converging as $i\rightarrow\infty$ to $w_j$. Then we can assume that $B_{\f25s_j}(w_{i,j})$ and $B_{\f25s_{j'}}(w_{i,j'})$ are disjoint for all distinct $j,j'\le N_\ep$ and for all the sufficiently large $i$.
Hence, from Lemma \ref{LOWERM000} and $s_j\le(\k+1)^{-1}$ (up to a scaling of $N_i$)
\begin{equation}\aligned
\mathcal{H}^n\left(M_i\cap B_{\f{2s_j}5}(w_{i,j})\right)\ge \de_{n}s_j^n
\endaligned
\end{equation}
for some constant $\de_n>0$ depending only on $n$.
With \eqref{UNBOUNDn} we have
\begin{equation}\aligned
\de_n\sum_{j=1}^{N'_\ep}s_j^n\le\sum_{j=1}^{N'_\ep}\mathcal{H}^n\left(M_i\cap B_{\f{2s_j}5}(w_{i,j})\right)\le\mathcal{H}^n(M_i\cap B_{r+2\ep}(p_i))\le c_{n,\k}(r+2\ep)^n.
\endaligned
\end{equation}
With \eqref{HnMinftyCOVER}, there is a constant $c_{n,\k}^*>0$ depending only on $n,\k$ such that
\begin{equation}\aligned\label{upbdMinftyr}
\mathcal{H}^n\left(M_\infty\cap \overline{B_r(p_\infty)}\right)\le c_{n,\k}^* r^n\qquad \mathrm{for\ each}\ 0<r<1,
\endaligned
\end{equation}
which implies \eqref{HnBrpinfUP}.

We further suppose $p_\infty\in M_\infty$.
For any $0<\ep'<r<1$, (combining spherical measure is equivalent to Hausdorff measure up to a constant) there is a finite covering $\{B_{r_j}(z_j)\}_{j=1}^{N_\ep'}$ of $M_\infty\cap \overline{B_r(p_\infty)}$ with $B_{r_j}(z_j)\subset B_1(p_\infty)$ such that
\begin{equation}\aligned\label{ddddd}
\mathcal{H}^n\left(M_\infty\cap \overline{B_r(p_\infty)}\right)\ge c_n\sum_{j=1}^{N_\ep'}r_j^n-\ep,
\endaligned
\end{equation}
where $c_n$ is a constant depending only on $n$.
For each $j$, let $z_{i,j}\in N_i$ be a sequence of points with $B_{r_j}(z_{i,j})\subset B_1(p_i)$ and $z_{i,j}\rightarrow z_j$ as $i\rightarrow\infty$.
Then $\{B_{r_j}(z_{i,j})\}_{j=1}^{N_\ep'}$ is a covering of $M_i\cap B_{r}(p_i)$ for the sufficiently large $i$.
With \eqref{UNBOUNDn} and \eqref{ddddd}, we have
\begin{equation}\aligned
\mathcal{H}^n\left(M_\infty\cap \overline{B_r(p_\infty)}\right)\ge\f{c_n}{c_{n,\k}}\sum_{j=1}^{N_\ep}\mathcal{H}^n\left(M_i\cap B_{r_j}(z_{i,j})\right)-\ep
\ge\f{c_n}{c_{n,\k}}\mathcal{H}^n\left(M_i\cap B_{r}(p_i)\right)-\ep.
\endaligned
\end{equation}
Combining \eqref{tArea} (or \eqref{nnnArea}), there is a constant $\de_{n,\k,v}>0$ depending only on $n,\k,v$ such that
\begin{equation}\aligned
\mathcal{H}^n\left(M_\infty\cap \overline{B_r(p_\infty)}\right)\ge\de_{n,\k,v}r^n-\ep.
\endaligned
\end{equation}
We get \eqref{lowerbdMinftyr} by letting $\ep\rightarrow0$ in the above inequality.
This completes the proof.
\end{proof}

Now let us study the limits of a sequence of area-minimizing hypersurfaces in a sequence of manifolds converging to a smooth geodesic ball.
\begin{theorem}\label{Conv0}
Let $N_i$ be a sequence of $(n+1)$-dimensional smooth Riemannian manifolds with $\mathrm{Ric}\ge-n\k^2$ on the metric ball $B_{1+\k'}(p_i)\subset N_i$ with $\k\ge0$, $\k'>0$.
Suppose that $B_1(p_i)$ converges to an $(n+1)$-dimensional smooth manifold $\overline{B_1(p)}$ in the Gromov-Hausdorff sense.
Let $M_i$ be an area-minimizing hypersurface in $B_1(p_i)$ with $p_i\in M_i$ and $\p M_i\subset\p B_1(p_i)$, and $M_i$ converges to a closed set $M_\infty$ in $\overline{B_1(p)}$ in the induced Hausdorff sense. Then $M_\infty$ is area-minimizing in $B_1(p)$ and for any $t\in(0,1)$
\begin{equation}\aligned\label{HnMinfequMitpi}
&\mathcal{H}^n\left(M_\infty\cap \overline{B_t(p)}\right)\ge\limsup_{i\rightarrow\infty}\mathcal{H}^{n}\left(M_i\cap \overline{B_t(p_i)}\right)\\
\ge&\liminf_{i\rightarrow\infty}\mathcal{H}^{n}(M_i\cap B_t(p_i))\ge\mathcal{H}^n(M_\infty\cap B_t(p)).
\endaligned
\end{equation}
\end{theorem}
\begin{proof}
Let $x\in M_\infty\cap B_1(p)$.
There is a suitable small $0<r<\f12(1-d(p,x))$ such that $B_{r}(x)$ is diffeomorphic to an $(n+1)$-dimensional Euclidean ball $B_r(0)\subset\R^{n+1}$.
Let $x_i\in B_1(p_i)$ with $x_i\rightarrow x$.
From Theorem A.1.8 in \cite{CCo1}, there are open sets $U_i\subset B_r(x_i)$ such that
$U_i$ is homeomorphic to $B_r(0)$, and $U_i\supset B_{(1-\ep_i)r}(x_i)$ for some sequence $\ep_i\to0$ as $i\to\infty$.
Then there is an open set $E_i$ in $U_i$ with $\p E_i\cap U_i=M_i\cap U_i$.
Up to a choice of the subsequence, we assume that $E_i$ converges to a closed set $E_\infty$ in $\overline{B_r(x)}\subset B_1(p)$ in the induced Hausdorff sense.
From Lemma \ref{subset} in the appendix II, it follows that $\p E_\infty\cap B_r(x)\subset M_\infty$. We claim
\begin{equation}\aligned\label{CLAy}
M_\infty\cap B_r(x)=\p E_\infty\cap B_r(x).
\endaligned
\end{equation}
It remains to prove $y\in \p E_\infty$ for any $y\in M_\infty\cap B_r(x)$.
If it fails, i.e., there is a point $y\in M_\infty\cap B_r(x)\setminus \p E_\infty$. Then there is a positive constant $\de_0\in(0,r-d(x,y))$ such that $B_{\de_0}(y)\subset E_\infty$. Let $y_{i}\in \p E_{i}\cap B_1(p_i)$ with $y_i\rightarrow y$,
then for any $\de\in(0,\de_0]$, $E_i\cap B_\de(y_i)$ converges to $E_\infty\cap B_\de(y)=B_\de(y)$ in the induced Hausdorff sense.

From Lemma \ref{nuinftyKthj}, there is a sequence of mutually disjoint balls $\{B_{\th_j}(z_j)\}_{j=1}^\infty$ with $z_j\in \overline{B_{(1-\de)\de}(y)}\setminus B_{\de'}(M_\infty)$ and $\th_j<\min\{\de'/6,\de^2/6\}$
such that $\overline{B_{(1-\de)\de}(y)}\setminus B_{\de'}(M_\infty)\subset \bigcup_{1\le j\le k}B_{\th_j+2\th_k}(z_j)$ for the sufficiently large $k$, and
\begin{equation}\aligned\label{(1-de)deyMde'}
\mathcal{H}^{n+1}\left(\overline{B_{(1-\de)\de}(y)}\setminus B_{\de'}(M_\infty)\right)\le\sum_{j=1}^{\infty}\mathcal{H}^{n+1}\left(B_{\th_j}(z_j)\right).
\endaligned
\end{equation}
Denote $E_i^\de=E_i\cap B_\de(y_i)$.
For each $j$, there is a sequence of points $z_{j,i}\in N_i$ with $z_{j,i}\to z_j$.
Since $E_i^\de\to B_\de(y)$ and $\p E_i\cap B_\de(y_i)\to M_\infty\cap B_\de(y)$, for each large $N_*$ there is an integer $i_0>0$ such that $B_{\th_j}(z_{j,i})\subset E_i^\de$, $B_{\th_{j}}(z_{j,i})\cap B_{\th_{j'}}(z_{j',i})=\emptyset$ for each $1\le j\neq j'\le N_*$ and $i\ge i_0$. Hence, from \eqref{VolCOV} it follows that
\begin{equation}\aligned
\sum_{j=1}^{N_*}\mathcal{H}^{n+1}\left(B_{\th_j}(z_j)\right)=\lim_{i\to\infty}\sum_{j=1}^{N_*}\mathcal{H}^{n+1}\left(B_{\th_j}(z_{j,i})\right)
\le\liminf_{i\to\infty}\mathcal{H}^{n+1}\left(E_i^\de\right).
\endaligned
\end{equation}
Combining \eqref{(1-de)deyMde'} we have
\begin{equation}\aligned\label{(1-de)deyMde''}
\mathcal{H}^{n+1}\left(\overline{B_{(1-\de)\de}(y)}\setminus B_{\de'}(M_\infty)\right)\le\liminf_{i\to\infty}\mathcal{H}^{n+1}\left(E_i^\de\right).
\endaligned
\end{equation}
Combining \eqref{nu*MinftyBtp} and Lemma \ref{lbd}, there is a constant $\de'>0$ such that
\begin{equation}\aligned\label{qqqqq1}
\mathcal{H}^n\left(B_{\de'}(M_\infty)\cap B_{\de}(y)\right)\le \de\mathcal{H}^{n+1}\left(E_i^\de\right)\qquad \mathrm{for\ each}\ i\ge 1.
\endaligned
\end{equation}
With \eqref{(1-de)deyMde''}\eqref{qqqqq1}, for each $i\ge1$ we deduce
\begin{equation}\aligned
\mathcal{H}^{n+1}(B_{(1-\de)\de}(y))\le&\mathcal{H}^{n+1}\left(\overline{B_{(1-\de)\de}(y)}\setminus B_{\de'}(M_\infty)\right)+\mathcal{H}^{n+1}\left(B_{\de'}(M_\infty)\cap B_{\de}(y)\right)\\
\le& \liminf_{i\to\infty}\mathcal{H}^{n+1}\left(E_i^\de\right)+\de\mathcal{H}^{n+1}\left(E_i^\de\right).
\endaligned
\end{equation}
From \eqref{VolCOV} and Lemma \ref{lbd}, the above inequality is impossible for the sufficiently small $\de>0$.
Hence $y\in \p E_\infty$ and we have shown the claim \eqref{CLAy}.

From Theorem 4.4 in \cite{Gi} and \eqref{upbdMinftyr}\eqref{CLAy}, $M_\infty$ is countably $n$-rectifiable in $B_1(p)$.
Combining \eqref{MSHS} and \eqref{lowerbdMinftyr}, we deduce
\begin{equation}\aligned\label{HnMinftyBtpde}
\mathcal{H}^n\left(M_\infty\cap \overline{B_t(p)}\right)=\lim_{\de\rightarrow0}\f1{2\de}\mathcal{H}^{n+1}\left(B_\de\left(M_\infty\cap \overline{B_t(p)}\right)\right)
\endaligned
\end{equation}
for any $t\in(0,1)$.
For the fixed $t\in(0,1)$, let $S$ be a hypersurface in $B_1(p)$ such that $S\setminus M_\infty$ is smooth embedded,
$S=M_\infty$ outside $B_t(p)$ and $\p \llbracket S\rrbracket=\p\llbracket M_\infty\rrbracket$.
Let $\mathbf{n}_{S}$ be the unit normal vector field to the regular part $\mathcal{R}_S$ of $S$.
Clearly, there is a constant $\th>0$ such that
\begin{equation}\aligned\nonumber
\{\mathrm{exp}_{x}(s\mathbf{n}_S(x))|\ x\in \mathcal{R}_S,\, 0<s<\th\}\cap \{\mathrm{exp}_{x}(-s\mathbf{n}_S(x))|\ x\in \mathcal{R}_S,\, 0<s<\th\}=\emptyset.
\endaligned
\end{equation}
Moreover, there is a sequence of hypersurfaces $S_{i}\subset B_1(p_i)$ such that
$S_{i}\setminus M_i$ is smooth embedded, $S_{i}=M_i$ outside $B_t(p_i)$, $\p \llbracket S_{i}\rrbracket=\p\llbracket M_i\rrbracket$,
and $S_{i}$ converges to $S$ in the induced Hausdorff sense.
From Proposition \ref{MinMSinfty} and \eqref{HnMinftyBtpde}, we have
\begin{equation}\aligned
\mathcal{H}^n\left(M_\infty\cap \overline{B_t(p)}\right)\le\mathcal{H}^{n}\left(S\cap \overline{B_t(p)}\right).
\endaligned
\end{equation}
Namely, $M_\infty$ is area-minimizing in $B_t(p)$.
As $t$ is arbitrary in $(0,1)$, we conclude that $M_\infty$ is area-minimizing in $B_1(p)$.

From \eqref{Bde-de2Side43116} with $S_i=M_i$, we have
\begin{equation}\aligned\label{Hn12121}
\mathcal{H}^{n+1}(B_\de(M_i)\cap B_{t+\psi_\de}(p_i))+\psi_\de\ge2\de\mathcal{H}^{n}\left(M_i\cap \overline{B_t(p_i)}\right)
\endaligned
\end{equation}
for any $t\in(0,1)$ and the small $\de>0$,
where $\psi_\de$ is a general function depending only on $n,\de$ with $\lim_{\de\rightarrow0}\psi_\de=0$.
From Lemma \ref{Cont*} in the appendix II, we have
\begin{equation}\aligned\label{Hn12122}
\mathcal{H}^{n+1}\left(B_{\de}(M_\infty)\cap B_{t+\psi_\de}(p)\right)\ge\limsup_{i\rightarrow\infty}\mathcal{H}^{n+1}(B_\de(M_i)\cap B_{t+\psi_\de}(p_i)).
\endaligned
\end{equation}
Combining \eqref{Hn12121}\eqref{Hn12122}, we get
\begin{equation}\aligned\label{Hn12123}
\liminf_{\de\rightarrow0}\f1{2\de}\mathcal{H}^{n+1}\left(B_{\de}(M_\infty)\cap B_{t+\psi_\de}(p)\right)\ge\limsup_{i\rightarrow\infty}\mathcal{H}^{n}\left(M_i\cap \overline{B_t(p_i)}\right).
\endaligned
\end{equation}
Combining \eqref{nu*MinftyBtp}\eqref{HnMinftyBtpde} and \eqref{Hn12123}, for any $t'<t$ we have
\begin{equation}\aligned\label{HnMinftyBtpMipi}
&\mathcal{H}^n\left(M_\infty\cap \overline{B_{2t-t'}(p)}\right)\ge\liminf_{\de\rightarrow0}\f1{2\de}\mathcal{H}^{n+1}\left(B_{\de}(M_\infty)\cap B_{2t-t'}(p)\right)\\
\ge&\limsup_{i\rightarrow\infty}\mathcal{H}^{n}\left(M_i\cap \overline{B_t(p_i)}\right)
\ge\liminf_{i\rightarrow\infty}\mathcal{H}^{n}(M_i\cap B_t(p_i))\\
\ge&\mathcal{M}^*\left(M_\infty,B_{t'}(p)\right)\ge\mathcal{H}^n(M_\infty\cap B_{2t'-t}(p)).
\endaligned
\end{equation}
Forcing $t'\to t$ gets \eqref{HnMinfequMitpi}.
This completes the proof.
\end{proof}
\textbf{Remark.} If we further assume that $B_1(p)$ is a Euclidean ball $B_1(0)$, then $M_\infty$ is an area-minimizing hypersurface in $B_1(0)$.
Clearly, $\mathcal{H}^n(M_\infty\cap\p B_t(0))=0$ for each $0<t<1$. Then \eqref{HnMinfequMitpi} implies
\begin{equation}\aligned\label{VOLCov}
\mathcal{H}^n(M_\infty\cap B_t(0))=\lim_{i\rightarrow\infty}\mathcal{H}^{n}(M_i\cap B_t(p_i))=\lim_{i\rightarrow\infty}\mathcal{H}^{n}\left(M_i\cap \overline{B_t(p_i)}\right).
\endaligned
\end{equation}
\begin{proposition}\label{ContiOmega}
Let $N_i,B_1(p_i),M_i,B_1(p_\infty),M_\infty$ be the notions as in Theorem \ref{Conv0}.
Suppose that $B_1(p_\infty)$ is an $(n+1)$-dimensional Euclidean ball $B_1(0)$. Then for any open set $\Om\subset\subset B_1(0)$, any open $\Om_i\subset B_1(p_i)$ with $\Om_i\rightarrow\Om$ in the induced Hausdorff sense,
\begin{equation}\aligned
\mathcal{H}^n\left(M_\infty\cap \overline{\Om}\right)\ge\limsup_{i\rightarrow\infty}\mathcal{H}^{n}\left(M_i\cap \overline{\Om_i}\right).
\endaligned
\end{equation}
\end{proposition}
\begin{proof}
By contradiction and compactness of area-minimizing hypersurfaces in Euclidean space, for any $\ep>0$ there is a constant $\de_\ep$ (depending only on $n,\ep$) such that
\begin{equation}\aligned\label{1-epdeept}
\mathcal{H}^n\left(M_\infty\cap B_t(q)\right)\ge(1-\ep)\mathcal{H}^n(M_\infty\cap B_{(1+\de_\ep)t}(q))
\endaligned
\end{equation}
for any $B_{(1+\de_\ep)t}(q)\subset B_1(0)$.

From Lemma \ref{nuinftyKthj} in appendix II,
there is a sequence of mutually disjoint balls $\{B_{\th_j}(x_j)\}_{j=1}^\infty$ with $x_j\in M_\infty\cap \overline{\Om}$
and $\th_j<\inf_{x\in\p\Om}|1-x|$
such that $M_\infty\cap \overline{\Om}\subset \bigcup_{1\le j\le k}B_{\th_j+2\th_k}(x_j)$ for the sufficiently large $k$, $\th_j\ge \th_{j+1}$ for each $j\ge1$ and
\begin{equation}\aligned\label{thjn2k**}
\sum_{j=1}^{\infty}\th_j^n
=\lim_{k\rightarrow\infty}\sum_{j=1}^{k}(\th_{j}+2\th_k)^n.
\endaligned
\end{equation}
From Lemma \ref{UPPLOWMinfty},
\begin{equation}\aligned
\mathcal{H}^n\left(M_\infty\right)\ge\sum_{j=1}^{\infty}\mathcal{H}^n\left(M_\infty\cap B_{\th_j}(x_j)\right)\ge\de_{n,\k,v}\sum_{j=1}^{\infty}\th_j^n.
\endaligned
\end{equation}
Hence there is an integer $j_0>1$ such that $\sum_{j=j_0}^{\infty}\th_j^n<\ep$. Then with \eqref{thjn2k**}
\begin{equation}\aligned
\limsup_{k\rightarrow\infty}\sum_{j=j_0}^{k}(\th_{j}+2\th_k)^n\le \lim_{k\rightarrow\infty}\sum_{j=1}^{k}(\th_{j}+2\th_k)^n-\sum_{j=1}^{j_0-1}\th_j^n
=\sum_{j=1}^{\infty}\th_j^n-\sum_{j=1}^{j_0-1}\th_j^n=\sum_{j=j_0}^{\infty}\th_j^n<\ep.
\endaligned
\end{equation}
With Lemma \ref{UPPLOWMinfty}, for the sufficiently large $k$ we have
\begin{equation}\aligned
&\mathcal{H}^n\left(M_\infty\cap B_\ep(\Om)\right)\ge\sum_{j=1}^{j_0-1}\mathcal{H}^n\left(M_\infty\cap B_{\th_j}(x_j)\right)
+c_{n,\k}^*\sum_{j=j_0}^k(\th_j+2\th_k)^n-c_{n,\k}^*\ep\\
&\ge\sum_{j=1}^{j_0-1}\mathcal{H}^n\left(M_\infty\cap B_{\th_j}(x_j)\right)
+\sum_{j=j_0}^k\mathcal{H}^n\left(M_\infty\cap B_{\th_j+2\th_k}(x_j)\right)-c_{n,\k}^*\ep.
\endaligned
\end{equation}
For the sufficiently large $k$, $(1+\de_\ep)\th_j>\th_j+2\th_k$ for each $j=1,\cdots,j_0-1$.
Let $x_{j,i}\in B_1(p_i)$ with $x_{j,i}\to x_j$ as $i\to\infty$.
With \eqref{1-epdeept}, we have
\begin{equation}\aligned\label{MinftyovOm}
\mathcal{H}^n\left(M_\infty\cap B_\ep(\Om)\right)\ge(1-\ep)\limsup_{i\rightarrow\infty}\sum_{j=1}^k\mathcal{H}^n\left(M_i\cap B_{\th_j+2\th_k}(x_{j,i})\right)-c_{n,\k}^*\ep.
\endaligned
\end{equation}
Since
$M_\infty\cap \overline{\Om}\subset \bigcup_{1\le j\le k}B_{\th_j+2\th_k}(x_j)$,
then for any open set $\Om_i\subset B_1(p_i)$ with $\Om_i\rightarrow\Om$ in the induced Hausdorff sense, we have
\begin{equation}\aligned
M_i\cap \overline{\Om_i}\subset\bigcup_{j=1}^{k}B_{\th_j+2\th_k}(x_{j,i})
\endaligned
\end{equation}
for the sufficiently large $i>0$.
From \eqref{MinftyovOm}, we have
\begin{equation}\aligned\nonumber
\mathcal{H}^n\left(M_\infty\cap B_\ep(\Om)\right)\ge(1-\ep)\limsup_{i\rightarrow\infty}\mathcal{H}^n\left(M_i\cap \overline{\Om_i}\right)-c_{n,\k}^*\ep.
\endaligned
\end{equation}
Letting $\ep\to0$ completes the proof.
\end{proof}

In the $(n+1)$-dimensional Euclidean ball $B_1(0)$, any minimal hypersurface $S$ through the origin satisfies $\mathcal{H}^n(S)\ge\omega_n$ by the classical monotonicity of $r^{-n}\mathcal{H}^n(S\cap B_r(0))$. Combining \eqref{tArea} and Theorem \ref{Conv0}, we immediately have the following result.
\begin{corollary}\label{epdeHM}
For any $\ep>0$ there is a constant $\de_\ep>0$ depending only on $n,\ep$ such that if $B_1(p)$ is an $(n+1)$-dimensional smooth geodesic ball with Ricci curvature $\ge-\de_\ep$, $\mathcal{H}^{n+1}\left(B_1(p)\right)\ge(1-\de_\ep)\omega_{n+1},$
and $M$ is an area-minimizing hypersurface in $B_1(p)$ with $p\in M$, $\p M\subset\p B_1(p)$, then
\begin{equation}\aligned
\mathcal{H}^n(M)\ge(1-\ep)\omega_n,\ \ \mathcal{H}^{n+1}(B_t(M)\cap B_1(p))\ge2(1-\ep)t\omega_n \ \mathrm{for\ any}\ t\in(0,\de_\ep].
\endaligned
\end{equation}
\end{corollary}
Moreover, we have the volume estimate for area-minimizing hypersurfaces in a class of smooth manifolds as follows.
\begin{proposition}\label{EnlocAM}
For any $\ep\in(0,\f12]$ there is a constant $\de_\ep>0$ depending only on $n,\ep$ such that if $B_1(p)$ is an $(n+1)$-dimensional smooth geodesic ball with Ricci curvature $\ge-\de_\ep$, $\mathcal{H}^{n+1}\left(B_1(p)\right)\ge(1-\de_\ep)\omega_{n+1},$
and $M$ is an area-minimizing hypersurface in $B_1(p)$ with $p\in M$, $\p M\subset\p B_1(p)$ and
$\mathcal{H}^n(M)\le(1+\de_\ep)\omega_n,$
then for any subset $U\subset B_{1-\ep}(p)$ with $diam\, U=1$ we have
\begin{equation}\aligned\label{LR2211}
\mathcal{H}^n\left(M\cap U\right)\le(1+\ep)\omega_n2^{-n}.
\endaligned
\end{equation}
\end{proposition}
\begin{proof}
Assume that there are a constant $\ep_*\in(0,\f12]$, and a sequence of $(n+1)$-dimensional smooth geodesic balls $B_1(p_i)$ with Ricci curvature $\ge-\f1i$ and $\mathcal{H}^{n+1}\left(B_1(p_i)\right)\ge(1-\f1i)\omega_{n+1}$, and a sequence of area-minimizing hypersurfaces $M_i$ in $B_1(p_i)$ with $p_i\in M_i$, $\p M_i\subset\p B_1(p_i)$ and $\mathcal{H}^n(M_i)\le(1+\f1i)\omega_n$ such that
\begin{equation}\aligned\label{Mi||ep*}
\mathcal{H}^n\left(M_i\cap U_i\right)>(1+\ep_*)\omega_n2^{-n}
\endaligned
\end{equation}
for some sequence of sets $U_i\subset B_{1-\ep_*}(p_i)$ with diam$\, U_i=1$.
We fix a constant $t\in(0,\ep_*)$.
From Theorem \ref{Conv0}, without loss of generality, we can assume that $B_{1-t}(p_i)$ converges to a Euclidean ball $\overline{B_{1-t}(0)}$ in the Gromov-Hausdorff sense, $M_i$ converges to an area-minimizing hypersurface $M_{t,\infty}$ through $0$ in $B_{1-t}(0)$ in the induced Hausdorff sense, and $U_i$ converges to a closed set $U_{t,\infty}$ in $\overline{B_{1-\ep_*}(0)}$ in the induced Hausdorff sense with diam$\, U_{t,\infty}=1$.
From Proposition \ref{ContiOmega} and \eqref{Mi||ep*}, we have
\begin{equation}\aligned\label{Mi||ep**}
\mathcal{H}^n\left(M_{t,\infty}\cap U_{t,\infty}\right)\ge\limsup_{i\rightarrow\infty}\mathcal{H}^n\left(M_i\cap U_i\right)\ge(1+\ep_*)\omega_n2^{-n}.
\endaligned
\end{equation}
From \eqref{HnMinfequMitpi} and $\mathcal{H}^n(M_i)\le(1+\f1i)\omega_n$, we get $\mathcal{H}^n(M_{t,\infty}\cap B_{1-t}(0))\le\omega_n$.

From compactness of area-minimizing hypersurfaces, there is a sequence $t_i\to0$ such that $M_{t_i,\infty}$ converges to an area-minimizing hypersurface $M_\infty$ through $0$ in $B_1(0)$.
Then from $\mathcal{H}^n(M_{t_i,\infty}\cap B_{1-t_i}(0))\le\omega_n$ , we get $\mathcal{H}^n(M_{\infty}\cap B_1(0))\le\omega_n$, which implies flatness of $M_{\infty}$.
Without loss of generality, we assume that $U_{t_i,\infty}$ converges in the Hausdorff sense to a closed set $U_{\infty}$ in $\overline{B_{1-\ep_*}(0)}$ with diam$\, U_{\infty}=1$.
Combining Proposition \ref{ContiOmega} and \eqref{Mi||ep**}, we get
\begin{equation}\aligned
\mathcal{H}^n\left(M_\infty\cap U_{\infty}\right)\ge\limsup_{i\rightarrow\infty}\mathcal{H}^n\left(M_{t_i,\infty}\cap U_{t_i,\infty}\right)\ge(1+\ep_*)\omega_n2^{-n}.
\endaligned
\end{equation}
This contradicts to the flatness of $M_{\infty}$ with the isodiametric inequality (see 2.5 in Chapter 1 of \cite{S} or \cite{FH1}\cite{HR}). We complete the proof.
\end{proof}

For any integer $n\ge7$, (from Allard's regularity theorem) there is a positive constant $\th_n$ depending only on $n$ such that the densities of all the non-flat area-minimizing hypercones in $\R^{n+1}$ is no less than $1+\th_n$.
Namely, any area-minimizing non-flat hypercone $C$ in $\R^{n+1}$ satisfies
 \begin{equation}\aligned\label{DEFthn}
\lim_{r\rightarrow\infty}\f1{\omega_n r^n}\mathcal{H}^n(C\cap B_r(0))\ge1+\th_n.
\endaligned
\end{equation}
For any integer $2\le n\le6$, let $\th_n=\infty$.
\begin{lemma}\label{almostmono}
For any integer $n\ge2$ and any $0<\ep'<\ep<\th_n$ there is a constant $\de>0$ depending only on $n,\ep,\ep'$ such that if $B_1(p)$ is an $(n+1)$-dimensional smooth geodesic ball with Ricci curvature $\mathrm{Ric}\ge-\de$ and $\mathcal{H}^{n+1}\left(B_1(p)\right)\ge(1-\de)\omega_{n+1},$
$M$ is an area-minimizing hypersurface in $B_1(p)$ with $p\in M$, $\p M\subset\p B_1(p)$, $\mathcal{H}^n(M)\le(1+\ep')\omega_n$, then
\begin{equation}\aligned\label{WRegM}
\mathcal{H}^n(M\cap B_r(p))\le(1+\ep)\omega_nr^n\qquad \mathrm{for\ any}\ 0<r<1.
\endaligned
\end{equation}
\end{lemma}
\begin{proof}
Let us prove \eqref{WRegM} by contradiction.
Assume that there are constants $0<\ep'<\ep<\th_n$, a sequence $0<r_i<1$, a sequence of $(n+1)$-dimensional smooth geodesic balls $B_1(p_i)$ with Ricci curvature $\ge-\f1i$ and $\mathcal{H}^{n+1}\left(B_1(p_i)\right)\ge(1-\f1i)\omega_{n+1}$, and a sequence of area-minimizing hypersurfaces $M_i$ in $B_1(p_i)$ with $p_i\in M_i$ and $\p M_i\subset\p B_1(p_i)$ such that
\begin{equation}\aligned\label{AssumeMi}
\mathcal{H}^n(M_i)\le(1+\ep')\omega_n
\endaligned
\end{equation}
and
\begin{equation}\aligned\label{AssumeMiBripi}
\mathcal{H}^n(M_i\cap B_{r_i}(p_i))>(1+\ep)\omega_nr_i^n.
\endaligned
\end{equation}
Let us consider the density function $\Th_i$ defined by
$$\Th_i(r)=\f{\mathcal{H}^n(M_i\cap B_r(p_i))}{\omega_nr^n}$$
for any $r\in(0,1]$.
From \eqref{AssumeMi} and \eqref{AssumeMiBripi}, there is a sequence of numbers $t_i\in(r_i,1)$ such that
$\Th_i(t_i)=1+\ep$ and $\Th_i(r)\le1+\ep$ for all $r\in(t_i,1)$.
Moreover, $(1+\ep)\omega_nt_i^n\le(1+\ep')\omega_n$ implies
\begin{equation}\aligned\label{tiepep'}
t_i\le\left(\f{1+\ep'}{1+\ep}\right)^{1/n}.
\endaligned
\end{equation}

Scale $B_1(p_i),M_i$ by the factor $\f{1}{t_i}$ w.r.t. $p_i$.
Namely, we define $B_{1/t_i}(\xi_i)=\f{1}{t_i}B_1(p_i)$, $M_i^*=\f{1}{t_i}M_i\subset B_{1/t_i}(\xi_i)$,
then $M_i^*$ is area-minimizing in $B_{1/t_i}(\xi_i)$.
Up to a choice of the subsequence, $(B_{t_i^{-1/2}}(\xi_i),\xi_i)$ converges to a metric space $(N_\infty^*,\xi^*)$ in the pointed Gromov-Hausdorff sense.
From \eqref{tiepep'}, $\mathrm{Ric}\ge-\f1i$ on $B_1(p_i)$ and $\mathcal{H}^{n+1}\left(B_1(p_i)\right)\ge(1-\f1i)\omega_{n+1}$, we deduce that $N_\infty^*$ contains a metric ball $B_{s_*}(\xi^*)$, which is a Euclidean ball with $s_*=\left(\f{1+\ep}{1+\ep'}\right)^{1/(2n)}$.
With Theorem \ref{Conv0}, we can assume that $M_i^*\cap B_{s_*}(\xi_i)$ converges to an area-minimizing hypersurface $M_\infty^*$ in $B_{s_*}(\xi^*)$.
From \eqref{HnMinfequMitpi} and the choice of $t_i$, we have
\begin{equation}\aligned\label{Ri111}
r^{-n}\mathcal{H}^n\left(M_\infty^*\cap B_{r}(\xi^*)\right)\le \mathcal{H}^n\left(M_\infty^*\cap B_{1}(\xi^*)\right)=(1+\ep)\omega_n
\endaligned
\end{equation}
for all $r\in(1,s_*)$.
The monotonicity of $r^{-n}\mathcal{H}^n\left(M_\infty^*\cap B_r(\xi^*)\right)$ implies that $M_\infty^*$ is a cone.
Then \eqref{Ri111} contradicts to the definition of $\th_n$.
This completes the proof.
\end{proof}

Combining Proposition \ref{EnlocAM} and Lemma \ref{almostmono}, the local property of area-minimizing hypersurfaces in a class of manifolds can be controlled by large-scale conditions in the following sense.
\begin{theorem}\label{IntReg}
For any integer $n\ge2$, and any constant $\ep\in(0,\f12]$, there is a constant $\de>0$ depending only on $n,\ep$ such that if
$B_1(p)$ is an $(n+1)$-dimensional smooth geodesic ball with Ricci curvature $\ge-\de$, $\mathcal{H}^{n+1}\left(B_1(p)\right)\ge(1-\de)\omega_{n+1},$
and $M$ is an area-minimizing hypersurface in $B_1(p)$ with $p\in M$, $\p M\subset\p B_1(p)$ and $\mathcal{H}^n(B_1(p)\cap M)<(1+\de)\omega_n$, then for any subset $U\subset B_{1-\ep}(p)$ we have
\begin{equation}\aligned
\mathcal{H}^n(M\cap U)\le(1+\ep)\omega_{n}2^{-n}(diam\, U)^n.
\endaligned
\end{equation}
In particular, $M$ is smooth in $B_{1-\ep}(p)$ provided we choose $\ep<\th_n$.
\end{theorem}
Note that the singular set of every area-minimizing hypersurface in a manifold has the codimension 7 at most. Then combining Theorem \ref{Conv0}, Proposition \ref{EnlocAM} and Lemma \ref{almostmono},
we can prove the following result by contradiction.
\begin{theorem}\label{PARTReg}
For any integer $n\ge2$, and any constant $\ep\in(0,\f12]$, there is a constant $\de>0$ depending only on $n,\ep$ such that if
$B_1(p)$ is an $(n+1)$-dimensional smooth geodesic ball with Ricci curvature $\ge-\de$, $\mathcal{H}^{n+1}\left(B_1(p)\right)\ge(1-\de)\omega_{n+1},$
and $M$ is an area-minimizing hypersurface in $B_1(p)$ with $\p M\subset\p B_1(p)$, then there is a collection of balls $\{B_{s_i}(x_i)\}_{i=1}^{N_\ep}$ with $\sum_{i=1}^{N_\ep}s_i^{n-7+\ep}<\ep$ so that for any subset $U\subset B_{1-\ep}(p)\setminus\cup_{i=1}^{N_\ep}B_{s_i}(x_i)$ with $diam\, U<\de$ we have
\begin{equation}\aligned
\mathcal{H}^n(M\cap U)\le(1+\ep)\omega_{n}2^{-n}(diam\, U)^n.
\endaligned
\end{equation}
\end{theorem}

\section{Continuity for the volume function of area-minimizing hypersurfaces}

Let $N_i$ be a sequence of $(n+1)$-dimensional complete smooth manifolds with \eqref{Ric} and \eqref{Vol}.
Up to choose the subsequence, we assume that $\overline{B_1(p_i)}$ converges to a metric ball $\overline{B_1(p_\infty)}$ in the Gromov-Hausdorff sense.
Let $M_i$ be an area-minimizing hypersurface in $B_1(p_i)\subset N_i$ with $\p M_i\subset\p B_1(p_i)$.
Suppose that $M_i$ converges to a closed set $M_\infty\subset \overline{B_1(p_\infty)}$ in the induced Hausdorff sense.
\begin{lemma}\label{LowerMinfty}
For any $\de>0$ there is a constant $r_\de\in(0,1)$ such that if $0<r\le r_\de$ and $x\in M_\infty\cap B_{1-3r}(p_\infty)$ with
$d_{GH}\left(B_{2r}(x),B_{2r}(0)\right)<r_\de r,$
then for any sequence $M_i\ni x_i\rightarrow x$, there holds
\begin{equation}\aligned\label{1-deBrMinf}
(1-\de)\limsup_{i\rightarrow\infty}\mathcal{H}^n\left(M_i\cap \overline{B_{(1+r_\de)r}(x_i)}\right)\le\mathcal{H}^n\left(M_\infty\cap \overline{B_{r}(x)}\right).
\endaligned
\end{equation}
\end{lemma}
\begin{proof}
We only need to show that there is a constant $r_\de'\in(0,1)$ such that if $0<r\le r_\de'$ and $x\in M_\infty\cap B_{1-3r}(p_\infty)$ with
$d_{GH}\left(B_{2r}(x),B_{2r}(0)\right)<r_\de' r,$
then for any sequence $M_i\ni x_i\rightarrow x$
\begin{equation}\aligned\label{1-deBrMinf*}
(1-\de/2)\limsup_{i\rightarrow\infty}\mathcal{H}^n\left(M_i\cap \overline{B_{(1+s)r}(x_i)}\right)\le\mathcal{H}^n\left(M_\infty\cap \overline{B_{(1+s)r}(x)}\right)\quad \mathrm{for\ any}\ 0<s\le r_\de'.
\endaligned
\end{equation}
For the sufficiently small $r_\de''>0$ we can assume
\begin{equation}\aligned
\mathcal{H}^n\left(M_\infty\cap \overline{B_{(1+r_\de'')r}(x)}\right)\le\f{1-\de/2}{1-\de}\mathcal{H}^n\left(M_\infty\cap \overline{B_{r}(x)}\right).
\endaligned
\end{equation}
We choose $r_\de=\min\{r_\de',r_\de''\}$, then \eqref{1-deBrMinf} holds.

Let us prove \eqref{1-deBrMinf*} by contradiction. We assume that there are a constant $\ep_0>0$,
3 sequences of positive constants $r_j,r_j^*\le r_j'\rightarrow0$, a sequence of points $z_j\in M_\infty\cap B_{1-3r_j}(p_\infty)$, and a sequence of points $M_i\ni z_{i,j}\rightarrow z_j$ such that
$
d_{GH}\left(B_{2r_j}(z_j),B_{2r_j}(0)\right)<r_j'r_j
$
and
\begin{equation}\aligned\label{ep0CONTR}
(1-\ep_0)\limsup_{i\rightarrow\infty}\mathcal{H}^n\left(M_i\cap \overline{B_{(1+r_j^*)r_j}(z_{i,j})}\right)>\mathcal{H}^n\left(M_\infty\cap \overline{B_{(1+r_j^*)r_j}(z_j)}\right).
\endaligned
\end{equation}
For any $0<\tau<\ep$, for each $j$ let $\{U_{j,k}\}_{k=1}^{N_{\tau,\ep,j}}$ be a finite covering of $M_{\infty}\cap \overline{B_{(1+r_j^*)r_j}(z_j)}$ with $\mathrm{diam} U_{j,k}\le\tau r_j$ such that
\begin{equation}\aligned\label{MBrjNtauepj**}
\mathcal{H}^n\left(M_{\infty}\cap \overline{B_{(1+r_j^*)r_j}(z_j)}\right)> \f{\omega_n}{2^n}\sum_{k=1}^{N_{\tau,\ep,j}} (\mathrm{diam} U_{j,k})^n-\ep r_j^n.
\endaligned
\end{equation}
Without loss of generality, we assume that all the $U_{j,k}$ are open sets.

There is a subsequence $i_j$ so that $d_{GH}\left(B_{2r_j}(z_{i_j,j}),B_{2r_j}(0)\right)<r_j'r_j$ and
\begin{equation}\aligned\label{MijsipMeprj}
\mathcal{H}^n\left(M_{i_j}\cap \overline{B_{(1+r_j^*)r_j}(z_{i_j,j})}\right)\ge\limsup_{i\rightarrow\infty}\mathcal{H}^n\left(M_i\cap \overline{B_{(1+r_j^*)r_j}(z_{i,j})}\right)-\ep r_j^n.
\endaligned
\end{equation}
Up to a choice of the subsequence, for each $j$ there is a collection of open sets $U_{j,k}^*\subset B_{3r_j/2}(z_{i_j,j})$ so that $M_{i_j}\cap \overline{B_{(1+r_j^*)r_j}(z_{i_j,j})}\subset\bigcup_{k=1}^{N_{\tau,\ep,j}}U_{j,k}^*$ and $\mathrm{diam} U_{j,k}^*=\mathrm{diam} U_{j,k}$.
From Theorem \ref{PARTReg}, for each $j$
there are $\tau_\ep>0$ and a collection of balls $\{B_{s_{j,k}}(x_{j,k})\}_{k=1}^{N_\ep}\subset B_{3r_j/2}(z_{i_j,j})$ with $\sum_{k=1}^{N_\ep}s_{j,k}^n<\ep r_j^n$ so that for any subset $U\subset B_{3r_j/2}(z_{i_j,j})\setminus\cup_{k=1}^{N_\ep}B_{s_{j,k}}(x_{j,k})$ with diam$U<\tau_\ep r_j$ we have
\begin{equation}\aligned\label{1-epMijU}
(1-\ep)\mathcal{H}^n(M_{i_j}\cap U)\le\omega_{n}2^{-n}(\mathrm{diam}\, U)^n.
\endaligned
\end{equation}
Up to a choice of the subsequence of $i_j$, we can require $\tau_\ep$ depending only on $n,\ep,r_j$.
Hence we can assume $\tau<\tau_\ep$.
Let $F_{j,\ep}=\bigcup_{k=1}^{N_\ep} B_{s_{j,k}}(x_{j,k})$, then combining Lemma \ref{upbdareaM}
\begin{equation}\aligned\label{HNtaurjEjep}
&\mathcal{H}^{n+1}(M_{i_j}\cap B_{2\tau r_j}(F_{j,\ep}))\le \sum_{k=1}^{N_\ep}\mathcal{H}^{n+1}\left(M_{i_j}\cap B_{s_{j,k}+2\tau r_j}(x_{j,k})\right)\\
\le& c_{n,\k}\sum_{k=1}^{N_\ep}(s_{j,k}+2\tau r_j)^n\le c_{n,\k}\sum_{k=1}^{N_\ep}c_n(s_{j,k}^n+\tau^n r_j^n)\le c_{n,\k}c_n(\ep+N_\ep\tau^n) r_j^n,
\endaligned
\end{equation}
where $c_n$ is a constant depending only on $n$.
Denote $\mathcal{I}_{\tau,\ep,j}=\{k=1,\cdots,N_{\tau,\ep,j}|\, U_{j,k}^*\cap F_{j,\ep}=\emptyset\}$.
Note $\mathrm{diam} U_{j,k}^*=\mathrm{diam} U_{j,k}\le\tau r_j$.
Then $M_{i_j}\cap\overline{B_{(1+r_j^*)r_j}(z_{i_j,j})}\subset\bigcup_{k\in \mathcal{I}_{\tau,\ep,j}}U_{j,k}^*\cup B_{2\tau r_j}(F_{j,\ep})$.
From \eqref{MBrjNtauepj**}\eqref{1-epMijU}\eqref{HNtaurjEjep},
\begin{equation*}\aligned
&\mathcal{H}^n\left(M_{\infty}\cap \overline{B_{(1+r_j^*)r_j}(z_j)}\right)> \f{\omega_n}{2^n}\sum_{k\in \mathcal{I}_{\tau,\ep,j}} (\mathrm{diam} U_{j,k}^*)^n-\ep r_j^n\\
\ge&(1-\ep)\sum_{k\in \mathcal{I}_{\tau,\ep,j}} \mathcal{H}^n(M_{i_j}\cap U_{j,k}^*)-\ep r_j^n-c_{n,\k}c_n(\ep+N_\ep\tau^n) r_j^n+\mathcal{H}^{n+1}(M_{i_j}\cap B_{2\tau r_j}(F_{j,\ep}))\\
\ge&(1-\ep)\mathcal{H}^n\left(M_{i_j}\cap \overline{B_{(1+r_j^*)r_j}(z_{i_j,j})}\right)-\ep r_j^n-c_{n,\k}c_n(\ep+N_\ep\tau^n) r_j^n.
\endaligned
\end{equation*}
Combining \eqref{MijsipMeprj}, we get
\begin{equation}\aligned
&\mathcal{H}^n\left(M_{\infty}\cap \overline{B_{(1+r_j^*)r_j}(z_j)}\right)\\
\ge&(1-\ep)\limsup_{i\rightarrow\infty}\mathcal{H}^n\left(M_i\cap \overline{B_{(1+r_j^*)r_j}(z_{i,j})}\right)-2\ep r_j^n-c_{n,\k}c_n(\ep+N_\ep\tau^n)r_j^n.
\endaligned
\end{equation}
Letting $\tau\to0$ first, and then $\ep\to0$ implies
\begin{equation}\aligned
\mathcal{H}^n\left(M_{\infty}\cap \overline{B_{(1+r_j^*)r_j}(z_j)}\right)\ge\limsup_{i\rightarrow\infty}\mathcal{H}^n\left(M_i\cap \overline{B_{(1+r_j^*)r_j}(z_{i,j})}\right),
\endaligned
\end{equation}
which contradicts to \eqref{ep0CONTR}. This completes the proof.
\end{proof}

Now let us study the upper semicontinuity for the volume function of area-minimizing hypersurfaces equipped with the induced Hausdorff topology.
Let $\mathcal{R}$ and $\mathcal{S}$ denote the regular set and the singular set of $B_1(p_\infty)$, respectively.
For any $\ep>0$, let $\mathcal{R}_\ep$ and $\mathcal{S}_\ep$ denote the sets as \eqref{RSep}.
Note that $\mathcal{S}_\ep\subset\mathcal{S}$ and $\mathcal{S}_\ep$ is closed in $B_1(p_\infty)$ with Hausdorff dimension of $\mathcal{S}_\ep$ $\le n-1$.
For each $0<t<1$, there are finite balls $\{B_{r_{y_j}}(y_j)\}_{j=1}^{N_\ep}\subset B_1(p_\infty)$ with
$\mathcal{S}_\ep\cap\overline{B_{t}(p_\infty)}\subset\bigcup_{j=1}^{N_\ep}B_{r_{y_j}}(y_j)$ so that
\begin{equation}\aligned\label{EepUPB}
r_{y_j}<\ep,\qquad \mathrm{and}\qquad \mathcal{H}^n_\ep(E_{\ep,t})\le\omega_n\sum_{j=1}^{N_\ep}r_{y_j}^n<\ep\quad \mathrm{with}\ E_{\ep,t}\triangleq\bigcup_{j=1}^{N_\ep}B_{r_{y_j}}(y_j).
\endaligned
\end{equation}
Let $\La_\ep$ be the function defined as \eqref{Laep}.
Since $\La_\ep$ is lower semicontinuous on $\mathcal{R}_\ep$,
then
\begin{equation}\aligned\label{Laeptxinfty*}
\La_{\ep,t}^*\triangleq\inf\left\{\La_\ep(x)|\, x\in M_\infty\cap\overline{B_{t}(p_\infty)}\setminus E_{\ep,t}\right\}>0.
\endaligned
\end{equation}
\begin{lemma}\label{UpMinfty*}
Let $\Om$ be an open set in $B_t(p_\infty)$.
If $\Om_i\subset B_1(p_i)$ are open with $\Om_i\rightarrow\Om$ in the induced Hausdorff sense, then for any $B_s(\Om)\subset B_t(p_\infty)$ with $s>0$, we have
$$\mathcal{H}^n\left(M_\infty\cap B_s(\Om)\right)\le
\min\left\{\liminf_{i\rightarrow\infty}\mathcal{H}^n(M_i\cap B_s(\Om_i)),\mathcal{M}_*\left(M_\infty, B_s(\Om)\right)\right\}.$$
\end{lemma}
\begin{proof}
Let $\ep\in(0,s)$.
From Lemma \ref{UPPLOWMinfty} and Lemma \ref{nuinftyKthj} in appendix II,
there is a sequence of mutually disjoint balls $\{B_{\th_j}(x_j)\}_{j=1}^\infty$ with $x_j\in M_\infty\cap \overline{B_{s-\ep}(\Om)}\setminus E_{\ep,t}$
and $\th_j<\min\{\La_{\ep,t}^*,\ep/3\}$
such that
\begin{equation}\aligned\label{MinftyBt'Eep}
\mathcal{H}^n_\ep\left(M_\infty\cap \overline{B_{s-\ep}(\Om)}\setminus E_{\ep,t}\right)\le\omega_n\sum_{j=1}^{\infty}\th_j^n.
\endaligned
\end{equation}
For each integer $j\ge 1$, there is a sequence of points $x_{i,j}\in M_i$ with $\lim_{i\rightarrow\infty}x_{i,j}=x_j$.
By Corollary \ref{epdeHM} and the definition of $\La_\ep$ in \eqref{Laep}, for any integer $k\ge1$ there are a positive function $\psi_\ep$ (independent of $k$)
with $\lim_{\ep\rightarrow0}\psi_\ep=0$
and an integer $i_k$ (depending on $k$) such that
\begin{equation}\aligned\label{430*}
(1+\psi_\ep)\mathcal{H}^n\left(M_i\cap B_{\th_j}(x_{i,j})\right)&\ge\omega_n\th_j^{n},\\
(1+\psi_\ep)\mathcal{H}^{n+1}\left(B_\th(M_i)\cap B_{\th_j}(x_{i,j})\right)&\ge2\th\omega_n\th_j^{n}\quad \mathrm{for\ any}\ \th\in(0,\psi_\ep)\\
\endaligned
\end{equation}
for each $i\ge i_k$ and each $1\le j\le k$. By the selection of $B_{\th_j}(x_j)$, we can require $B_{\th_{j_1}}(x_{i,j_1})\cap B_{\th_{j_2}}(x_{i,j_2})=\emptyset$ for all $j_1\neq j_2$ and the sufficiently large $i$. Hence there is a constant $\de_k>0$ (depending on $k$) such that from \eqref{430*}
\begin{equation}\aligned\label{wnthjn*}
\omega_n\sum_{j=1}^{k}\th_j^{n}\le&(1+\psi_\ep)\mathcal{H}^n(M_i\cap B_s(\Om_i)),\\
2\th\omega_n\sum_{j=1}^{k}\th_j^{n}\le&(1+\psi_\ep)\mathcal{H}^{n+1}\left(B_\th(M_i)\cap B_s(\Om_i)\right)
\endaligned
\end{equation}
for each $i\ge i_k$ and each $\th\in(0,\min\{\psi_\ep,\de_k\})$.
With Lemma \ref{UpMinfty},
letting $i\rightarrow\infty$, then $\th\rightarrow0$ infers
\begin{equation}\aligned
\omega_n\sum_{j=1}^{k}\th_j^{n}\le&(1+\psi_\ep)\liminf_{i\rightarrow\infty}\mathcal{H}^n(M_i\cap B_s(\Om_i)),\\
\omega_n\sum_{j=1}^{k}\th_j^{n}\le&(1+\psi_\ep)\mathcal{M}_*\left(M_\infty, B_s(\Om)\right).
\endaligned
\end{equation}
Letting $k\rightarrow\infty$ infers
\begin{equation}\aligned\label{432*}
\omega_n\sum_{j=1}^{\infty}\th_j^{n}\le(1+\psi_\ep)\min\left\{\liminf_{i\rightarrow\infty}\mathcal{H}^n(M_i\cap B_s(\Om_i)),\mathcal{M}_*\left(M_\infty, B_s(\Om)\right)\right\}.
\endaligned
\end{equation}
By the definition of $\mathcal{H}^n_\ep$ in \eqref{mathcalHnep}, \eqref{EepUPB} and \eqref{MinftyBt'Eep}, we have
\begin{equation}\aligned\label{433*}
\mathcal{H}^n_\ep\left(M_\infty\cap \overline{B_{s-\ep}(\Om)}\right)
\le\omega_n\sum_{j=1}^{N_\ep}r_{y_j}^n+\mathcal{H}^n_\ep\left(M_\infty\cap \overline{B_{s-\ep}(\Om)}\setminus E_{\ep,t}\right)
<\ep+\omega_n\sum_{j=1}^{\infty}\th_j^n.
\endaligned
\end{equation}
Combining \eqref{432*}\eqref{433*} we have
\begin{equation}\aligned\label{436*}
\mathcal{H}^n_\ep\left(M_\infty\cap \overline{B_{s-\ep}(\Om)}\right)\le\ep+(1+\psi_\ep)\min\left\{\liminf_{i\rightarrow\infty}\mathcal{H}^n(M_i\cap B_s(\Om_i)),\mathcal{M}_*\left(M_\infty, B_s(\Om)\right)\right\}.
\endaligned
\end{equation}
Note that the Hausdorff measure $\mathcal{H}^n$ is a Radon measure. Letting $\ep\to0$ suffices to complete the proof.
\end{proof}
From Lemma \ref{UpMinfty*}, we immediately have
\begin{equation}\aligned\label{HnBtpinfMiMINF}
\mathcal{H}^n\left(M_\infty\cap B_t(p_\infty)\right)\le
\min\left\{\liminf_{i\rightarrow\infty}\mathcal{H}^n(M_i\cap B_t(p_i)),\mathcal{M}_*\left(M_\infty, B_{t}(p_\infty)\right)\right\}
\endaligned
\end{equation}
for each $t\in(0,1)$.

Using Lemma \ref{LowerMinfty}, we are able to show the lower semicontinuity for the volume function of area-minimizing hypersurfaces equipped with the induced Hausdorff topology (see Corollary \ref{ContiOmega} for the special case).
\begin{lemma}\label{LimMi}
For any open set $\Om\subset\subset B_1(p_\infty)$, if $\Om_i\subset B_1(p_i)$ are open with $\Om_i\rightarrow\Om$ in the induced Hausdorff sense, then
$$\limsup_{i\rightarrow\infty}\mathcal{H}^n\left(M_i\cap \overline{\Om_i}\right)\le\mathcal{H}^n\left(M_\infty\cap \overline{\Om}\right).$$
\end{lemma}
\begin{proof}
Suppose $\overline{\Om}\subset B_{t}(p_\infty)$ for some $t\in(0,1)$. Let $\ep<1-t$, $E_{\ep,t}$ be as defined in \eqref{EepUPB} and $\La_{\ep,t}^*$ be as defined in \eqref{Laeptxinfty*}.
From Lemma \ref{UPPLOWMinfty} and Lemma \ref{nuinftyKthj} in appendix II,
there is a sequence of mutually disjoint balls $\{B_{\th_j}(x_j)\}_{j=1}^\infty$ with $x_j\in M_\infty\cap \overline{\Om}\setminus E_{\ep,t}$
and $\th_j<\min\{\La_{\ep,t}^*,\ep/3\}$
such that $M_\infty\cap \overline{\Om}\setminus E_{\ep,t}\subset \bigcup_{1\le j\le k}B_{\th_j+2\th_k}(x_j)$ for the sufficiently large $k$, and
\begin{equation}\aligned
\mathcal{H}^n\left(M_\infty\right)\ge\sum_{j=1}^{\infty}\mathcal{H}^n\left(M_\infty\cap B_{\th_j}(x_j)\right)\ge\de_{n,\k,v}\sum_{j=1}^{\infty}\th_j^n
=\de_{n,\k,v}\lim_{k\rightarrow\infty}\sum_{j=1}^{k}(\th_{j}+2\th_k)^n.
\endaligned
\end{equation}
Hence there is an integer $j_0>1$ such that $\sum_{j=j_0}^{\infty}\th_j^n<\ep$. Then
\begin{equation}\aligned
\limsup_{k\rightarrow\infty}\sum_{j=j_0}^{k}(\th_{j}+2\th_k)^n\le \lim_{k\rightarrow\infty}\sum_{j=1}^{k}(\th_{j}+2\th_k)^n-\sum_{j=1}^{j_0-1}\th_j^n
=\sum_{j=1}^{\infty}\th_j^n-\sum_{j=1}^{j_0-1}\th_j^n=\sum_{j=j_0}^{\infty}\th_j^n<\ep.
\endaligned
\end{equation}
With Lemma \ref{UPPLOWMinfty}, for the sufficiently large $k$ we have
\begin{equation}\aligned\label{HnMinfgecnk*ep}
&\mathcal{H}^n\left(M_\infty\cap B_{\ep}(\Om)\right)\ge\sum_{j=1}^{j_0-1}\mathcal{H}^n\left(M_\infty\cap B_{\th_j}(x_j)\right)
+c_{n,\k}^*\sum_{j=j_0}^k(\th_j+2\th_k)^n-c_{n,\k}^*\ep\\
&\ge\sum_{j=1}^{j_0-1}\mathcal{H}^n\left(M_\infty\cap B_{\th_j}(x_j)\right)
+\sum_{j=j_0}^k\mathcal{H}^n\left(M_\infty\cap B_{\th_j+2\th_k}(x_j)\right)-c_{n,\k}^*\ep.
\endaligned
\end{equation}
From Lemma \ref{LowerMinfty}, for the suitable small $\ep>0$ there is a constant $\xi_\ep>1$ such that
for each $j$ and each sequence $M_i\ni x_{j,i}\rightarrow x_j$ as $i\rightarrow\infty$, we have
\begin{equation}\aligned\label{abcd111}
(1-\ep)\limsup_{i\rightarrow\infty}\mathcal{H}^n\left(M_i\cap \overline{B_{\xi_\ep\th_j}(x_{j,i})}\right)\le\mathcal{H}^n\left(M_\infty\cap \overline{B_{\th_j}(x_j)}\right).
\endaligned
\end{equation}
For the sufficiently large $k$, $\xi_\ep\th_j>\th_j+2\th_k$ for each $1\le j\le j_0-1$. With \eqref{HnMinfgecnk*ep}, we have
\begin{equation}\aligned\label{M*tpinftyep}
\mathcal{H}^n\left(M_\infty\cap B_{\ep}(\Om)\right)\ge(1-\ep)\limsup_{i\rightarrow\infty}\sum_{j=1}^k\mathcal{H}^n\left(M_i\cap B_{\th_j+2\th_k}(x_{j,i})\right)-c_{n,\k}^*\ep.
\endaligned
\end{equation}
Recalling $\mathcal{S}_\ep\cap\overline{B_{t}(p_\infty)}\subset E_{\ep,t}=\bigcup_{j=1}^{N_\ep}B_{r_{y_j}}(y_j)$. Now we fix an integer $k$ sufficiently large so that
\begin{equation}\aligned\label{M01ep}
M_\infty\cap \overline{\Om}\subset\bigcup_{j=1}^{k}B_{\th_j+2\th_k}(x_j)\cup\bigcup_{j=1}^{N_\ep}B_{r_{y_j}}(y_j).
\endaligned
\end{equation}
Let $y_{j,i}\in M_i$ with $y_{j,i}\to y_j$ as $i\to\infty$.
Then from \eqref{M01ep}, we have
\begin{equation}\aligned\label{5.37}
M_i\cap \overline{\Om_i}\subset\bigcup_{j=1}^{k}B_{\th_j+2\th_k}(x_{j,i})\cup\bigcup_{j=1}^{N_\ep}B_{r_{y_j}}(y_{k,i})
\endaligned
\end{equation}
for the sufficiently large $i>0$.
Combining Lemma \ref{UPPLOWMinfty} and \eqref{EepUPB}\eqref{M*tpinftyep}\eqref{5.37}, we have
\begin{equation}\aligned\nonumber
\mathcal{H}^n\left(M_\infty\cap B_{\ep}(\Om)\right)\ge &
(1-\ep)\limsup_{i\rightarrow\infty}\sum_{j=1}^k\mathcal{H}^n\left(M_i\cap B_{\th_j+2\th_k}(x_{j,i})\right)-c_{n,\k}^*\ep\\
+& \limsup_{i\rightarrow\infty}\sum_{j=1}^{N_\ep}\mathcal{H}^n\left(M_i\cap B_{r_{y_j}}(y_{j,i})\right)-c_{n,\k}^*\sum_{j=1}^{N_\ep}r_{y_j}^n\\
\ge&(1-\ep)\limsup_{i\rightarrow\infty}\mathcal{H}^n\left(M_i\cap \overline{\Om_i}\right)-c_{n,\k}^*\ep-c_{n,\k}^*\omega_n^{-1}\ep.
\endaligned
\end{equation}
Letting $\ep\to0$ completes the proof.
\end{proof}
Uniting \eqref{HnBtpinfMiMINF} and Lemma \ref{LimMi}, we immediately have the following result.
\begin{theorem}\label{LimMiMib}
For each integer $i\ge1$,
let $M_i$ be an area-minimizing hypersurface in $B_1(p_i)$ with $\p M_i\subset\p B_1(p_i)$.
If $M_i$ converges to a closed set $M_\infty\subset \overline{B_1(p_\infty)}$ in the induced Hausdorff sense, then for any $t\in(0,1)$ we have
\begin{equation}\aligned
\mathcal{H}^n\left(M_\infty\cap B_t(p_\infty)\right)\le& \liminf_{i\rightarrow\infty}\mathcal{H}^n\left(M_i\cap B_t(p_i)\right)\\
\le&\limsup_{i\rightarrow\infty}\mathcal{H}^n\left(M_i\cap \overline{B_{t}(p_i)}\right)\le\mathcal{H}^n\left(M_\infty\cap \overline{B_{t}(p_\infty)}\right).
\endaligned
\end{equation}
\end{theorem}
In particular, if $\mathcal{H}^n\left(M_\infty\cap\p B_{t}(p_\infty)\right)=0$ for some $t\in(0,1)$, then the above theorem implies
\begin{equation}\aligned\label{EQUIVMEA}
\mathcal{H}^n\left(M_\infty\cap B_t(p_\infty)\right)=\liminf_{i\rightarrow\infty}\mathcal{H}^n\left(M_i\cap B_t(p_i)\right).
\endaligned
\end{equation}
From \eqref{Hn1Bdede*} and Lemma \ref{LimMi}, for any $\tau\in(0,1)$ we have
\begin{equation*}\aligned
\f{1-n\k\de}{2\de}\mathcal{H}^{n+1}\left(B_\de(M_\infty)\cap B_{\tau}(p_\infty)\right)\le
\liminf_{i\rightarrow\infty}\mathcal{H}^n\left(M_i\cap B_{\tau+\de}(p_i)\right)\le\mathcal{H}^n\left(M_\infty\cap \overline{B_{\tau+\de}(p_\infty)}\right),
\endaligned
\end{equation*}
which implies
\begin{equation}\aligned\label{M*Btau}
\mathcal{M}^*\left(M_\infty, B_{\tau}(p_\infty)\right)\le\mathcal{H}^n\left(M_\infty\cap \overline{B_{\tau}(p_\infty)}\right).
\endaligned
\end{equation}
Combining Lemma \ref{UpMinfty*} and \eqref{M*Btau}, we get the following result.
\begin{corollary}\label{M**Btau}
For each integer $i\ge1$,
let $M_i$ be an area-minimizing hypersurface in $B_1(p_i)$ with $\p M_i\subset\p B_1(p_i)$.
If $M_i$ converges to a closed set $M_\infty\subset \overline{B_1(p_\infty)}$ in the induced Hausdorff sense, then for any $\tau\in(0,1)$
\begin{equation*}\aligned
\mathcal{H}^n\left(M_\infty\cap B_{\tau}(p_\infty)\right)\le\mathcal{M}_*\left(M_\infty, B_{\tau}(p_\infty)\right)\le\mathcal{M}^*\left(M_\infty, B_{\tau}(p_\infty)\right)
\le\mathcal{H}^n\left(M_\infty\cap \overline{B_{\tau}(p_\infty)}\right).
\endaligned
\end{equation*}
\end{corollary}

\section{Singular sets related to limits of area-minimizing hypersurfaces}

Let $C_*$ be a 2-dimensional metric cone with cross section $\G$ and vertex at $o_*$,
where $\G$ is a 1-dimensional round circle with radius $\le1$.
Let $C=C_*\times\R^{n-1}$ with the standard product metric and $o=(o_*,0^{n-1})\in C$.
Let $\mathcal{C}_{n+1}$ denote the set including all the cones $C$ defined above.
Clearly, the Euclidean space $\R^{n+1}$ with the standard flat metric belongs to $\mathcal{C}_{n+1}$.
For each cone $C\in \mathcal{C}_{n+1}$, $C$ has the flat Euclidean metric on its regular part.
Let us prove a monotonicity formula for all the limits of area-minimizing hypersurfaces in the metric cone $C$ provided $C\in\mathcal{C}_{n+1}$ as follows.
\begin{lemma}\label{pM+Minfty}
Let $Q_i$ be a sequence of $(n+1)$-dimensional complete Riemannian manifolds with Ricci curvature $\ge-(n-1)R_i^{-2}$ on $B_{R_i}(q_i)\subset Q_i$ for some sequence $R_i\rightarrow\infty$. Suppose that $(Q_i,q_i)$ converges to a cone $(C,o)$ in $\mathcal{C}_{n+1}$ in the pointed Gromov-Hausdorff sense.
Let $\Si_i$ be a sequence of area-minimizing hypersurfaces in $B_1(q_i)\subset Q_i$ with $\p \Si_i\subset\p B_1(q_i)$, which converges in the induced Hausdorff sense to a closed set $\Si$ in $C$.
Then for any $0<t'<t<1$, we have
\begin{equation}\aligned\label{t1t2HnMpinfty}
t^{-n}\mathcal{H}^n\left(\Si\cap B_{t}(o)\right)\ge (t')^{-n}\mathcal{H}^n\left(\Si\cap \overline{B_{t'}(o)}\right).
\endaligned
\end{equation}
\end{lemma}
\begin{proof}
Let $\mathcal{S}_C$ denote the singular set of $C$.
From Theorem \ref{Conv0}, there is a multiplicity one rectifiable $n$-stationary varifold $V$ in $B_1(o)\setminus \mathcal{S}_C$ with supp$V\cap B_1(o)=\Si\cap B_1(o)$
outside $\mathcal{S}_C$.
Note that $B_1(o)\setminus \mathcal{S}_C$ has flat standard Euclidean metric.
Denote $\Si^*=\Si\cap B_1(o)\setminus\mathcal{S}_C$.
Let $\r$ be the distance function from $o$ in $C$, and $\p_\r$ denote the unit radial vector perpendicular to $\p B_\r(o)$.
For any smooth function $f$ with compact support in $C\setminus \mathcal{S}_C$, any $0<t_1<t_2<1$, we have the mean value inequality (see \cite{CM} or \cite{S} for instance)
\begin{equation}\aligned\label{MVfSi}
&t_2^{-n}\int_{\Si\cap B_{t_2}(o)}f-t_1^{-n}\int_{\Si\cap B_{t_1}(o)}f\\
=&\int_{\Si\cap B_{t_2}(o)\setminus B_{t_1}(o)}f|(\p_\r)^N|^2\r^{-n}+\int_{t_1}^{t_2}\tau^{-n-1}\int_{\Si\cap B_\tau(o)}\lan \na_\Si f,\p_\r\ran d\tau,
\endaligned
\end{equation}
where $\na_\Si$ denotes the Levi-Civita connection of the regular part of $\Si^*$, $(\p_\r)^N$ denotes the projection of $\p_\r$ to the normal bundle of the regular part of $\Si^*$. Since $\mathcal{S}_C$ has Hausdorff dimension $\le n-1$, then $\mathcal{H}^n\left(\Si\cap\mathcal{S}_C\right)=0$.

Let $\r_{\mathcal{S}_C}$ denote the distance function from $\mathcal{S}_C$.
For any small $0<\ep<t_1$, let $\phi_\ep$ be a Lipschitz function on $C$ defined by $\phi_\ep=1$ on $\{\r_{\mathcal{S}_C}\ge2\ep\r\}$, $\phi_\ep=\r_{\mathcal{S}_C}/(\ep\r)-1$ on $\{\ep\r\le\r_{\mathcal{S}_C}<2\ep\r\}$, and $\phi_\ep=0$ on $\{\r_{\mathcal{S}_C}<\ep\r\}$.
Then
\begin{equation}\aligned
\lan \na_C \phi_\ep,\p_\r\ran =0\qquad \mathrm{a.e.\ \ on} \ C\setminus \mathcal{S}_C,
\endaligned
\end{equation}
where $\na_C$ denotes the Levi-Civita connection of $C\setminus \mathcal{S}_C$.
Let $\e_\ep$ be a Lipschitz function on $C$ defined by $\e_\ep=1$ on $\{\r\ge2\ep\}$, $\e_\ep=\r/\ep-1$ on $\{\ep\le\r<2\ep\}$, and $\e_\ep=0$ on $\{\r<\ep\}$.
Note that $\phi_\ep \e_\ep$ is Lipschitz on $\Si^*$ for a.e. $\ep>0$. Then for $\tau\in[t_1,t_2]$
\begin{equation}\aligned
&\left|\int_{\Si\cap B_\tau(o)}\lan \na_\Si(\phi_\ep \e_\ep),\p_\r\ran \right|=\left|\int_{\Si\cap B_\tau(o)}\phi_\ep\lan \na_\Si\e_\ep,\p_\r\ran \right|\\
\le&\int_{\Si\cap B_{2\ep}(o)\setminus B_\ep(o)}\f1{\ep}\le\f1{\ep}\mathcal{H}^n(\Si\cap B_{2\ep}(o)).
\endaligned
\end{equation}
With Lemma \ref{UPPLOWMinfty}, we get
\begin{equation}\aligned
\lim_{\ep\rightarrow0}\int_{\Si\cap B_\tau(o)}\lan \na_\Si(\phi_\ep \e_\ep),\p_\r\ran=0.
\endaligned
\end{equation}
We choose the function $f$ in \eqref{MVfSi} approaching to $\phi_\ep\e_\ep$, and then $\ep\to0$ implies
\begin{equation}\aligned\label{MontSi}
&t_2^{-n}\mathcal{H}^n(\Si\cap B_{t_2}(o))-t_1^{-n}\mathcal{H}^n(\Si\cap B_{t_1}(o))
=&\int_{\Si^*\cap B_{t_2}(o)\setminus B_{t_1}(o)}|(\p_\r)^N|^2\r^{-n}.
\endaligned
\end{equation}
For any $0<t'<t<1$, let $t_2=t$, $t_1\to t'$ with $t_1>t'$, we complete the proof.
\end{proof}

Let $N_i$ be a sequence of $(n+1)$-dimensional complete smooth manifolds with \eqref{Ric} and \eqref{Vol}.
Up to choose the subsequence, we assume that $\overline{B_1(p_i)}$ converges to a metric ball $\overline{B_1(p_\infty)}$ in the Gromov-Hausdorff sense.
Let $M_i$ be an area-minimizing hypersurface in $B_1(p_i)\subset N_i$ with $\p M_i\subset\p B_1(p_i)$.
Suppose that $M_i$ converges to a closed set $M_\infty\subset \overline{B_1(p_\infty)}$ in the induced Hausdorff sense. Let $\mathcal{S}^{n-2}$ be the subset of $B_1(p_\infty)$ defined in \eqref{Sk}.
Using Lemma \ref{pM+Minfty}, we have the following cone property.
\begin{theorem}\label{Covcone}
For $y_\infty\in M_\infty\cap B_1(p_\infty)\setminus\mathcal{S}^{n-2}$, let $s_j$ be a sequence with $s_j\to0^+$ as $j\to\infty$ so that $\f1{s_j}(B_1(p_\infty),y_\infty)$ converges to a metric cone $(C,o)$ in the pointed Gromov-Hausdorff sense with $C\in\mathcal{C}_{n+1}$,
and $\f1{s_j}(M_\infty,y_\infty)$ converges in the induced Hausdorff sense to $(M^*,o)$ for some closed set $M^*\subset C$. Then for any sequence $\r_k\rightarrow0^+$, there is a subsequence $\r_{k'}$ of $\r_k$ such that $\f1{\r_{k'}}(M^*,o)$ converges in the induced Hausdorff sense to a metric cone $(C^*,o)$ with $C^*\subset C$.
Moreover, there is a sequence $r_j\rightarrow0^+$ such that $\f1{r_j}(B_1(p_\infty),y_\infty)$ converges to $(C,o)$ in the pointed Gromov-Hausdorff sense,
and $\f1{r_j}(M_\infty,y_\infty)$ converges in the induced Hausdorff sense to $(C^*,o)$.
\end{theorem}
\begin{proof}
Given $t>1$, by taking the diagonal subsequence, (up to choose the subsequences) we may assume $2ts_j\r_k\le 1-d(p_\infty,y_\infty)$ and
\begin{equation}\aligned\label{GHrjykl000}
d_{GH}\left(B_{ts_j\r_{k}}(y_{\infty}),B_{ts_j\r_{k}}(o)\right)<s_j\r_{k}/k
\endaligned
\end{equation}
for each integer $j\ge k>0$. Here, $B_{r}(o)$ is the ball in $C$ centered at its vertex $o$ with the radius $r$.
Let $\Phi_{k,j}:\,B_{ts_j\r_k}(y_{\infty})\to B_{ts_j\r_k}(o)\subset C$ denote a $3s_j\r_{k}/k$-Hausdorff approximation.
Up to a choice of subsequence of $j$, we can assume
\begin{equation}\aligned\label{eq510add000}
d_{H}\left(s_j^{-1}\Phi_{k,j}(M_{\infty}\cap B_{ts_j\r_{k}}(y_{\infty})),M^*\cap B_{t\r_{k}}(o)\right)<\r_{k}/k
\endaligned
\end{equation}
for any $j\ge k$.
There is a sequence of points $y_i\in M_i$ so that $y_i\to y_\infty$ as $i\to\infty$.
By taking the diagonal subsequence, (up to choose the subsequences) we can assume
\begin{equation}\aligned\label{GHrjykl}
d_{GH}\left(B_{ts_j\r_k}(y_{i}),B_{ts_j\r_k}(y_\infty)\right)<s_j\r_k/k,
\endaligned
\end{equation}
and
\begin{equation}\aligned\label{eq510add}
d_{H}\left(\Phi_{k,j,i}(M_{i}\cap B_{ts_j\r_k}(y_{i})),M_\infty\cap B_{ts_j\r_k}(y_\infty)\right)<s_j\r_k/k
\endaligned
\end{equation}
for any $i\ge j\ge k>0$,
where $\Phi_{k,j,i}$ is a $3s_j\r_k/k$-Hausdorff approximation from $B_{ts_j\r_k}(y_{i})$ to $B_{ts_j\r_k}(y_{\infty})\subset B_1(p_\infty)$.

Up to a choice of subsequence of $\r_k$, we assume that $\f1{\r_{k}}(M^*,o)$ converges in the induced Hausdorff sense to $(C^*,o)$ for some closed subset $C^*\subset C$ with $o\in C^*$. 
Using the monotonicity formula \eqref{t1t2HnMpinfty}, we get
\begin{equation}\aligned\label{limrkrto0+}
\lim_{k\to\infty}\r_{k}^{-n}\mathcal{H}^n\left(M^*\cap B_{\r_{k}}(o)\right)=\lim_{r\to0^+}r^{-n}\mathcal{H}^n\left(M^*\cap B_{r}(o)\right).
\endaligned
\end{equation}
From \eqref{GHrjykl000} and \eqref{GHrjykl}, we obtain
\begin{equation}\aligned\label{dGHrkskt0001}
\lim_{i\to\infty}d_{GH}\left(s_i^{-1}\r_k^{-1}B_{ts_i\r_k}(y_{j}),B_{t}(o)\right)=\lim_{k\to\infty}d_{GH}\left(s_k^{-1}\r_k^{-1}B_{ts_k\r_k}(y_{k}),B_{t}(o)\right)=0.
\endaligned
\end{equation}
From \eqref{eq510add000}\eqref{eq510add} and $\f1{\r_{k}}(M^*,o)\to(C^*,o)$, we obtain
\begin{equation}\aligned\label{dHrkskt0001}
&\lim_{i\to\infty}d_{H}\left(s_i^{-1}\r_k^{-1}\Phi_{k,i}(\Phi_{k,i,i}(M_i\cap B_{ts_i\r_{k}}(y_i))),\r_k^{-1}M^*\cap B_{t}(o)\right)\\
=&\lim_{k\to\infty}d_{H}\left(s_k^{-1}\r_k^{-1}\Phi_{k,k}(\Phi_{k,k,k}(M_k\cap B_{ts_k\r_{k}}(y_k))),C^*\cap B_{t}(o)\right)=0.
\endaligned
\end{equation}
By Lemma \ref{UPPLOWMinfty} and Theorem \ref{Conv0},
there is a multiplicity one rectifiable $n$-stationary varifold $V$ in the regular part of $C\cap B_t(o)$ with spt$V\triangleq C^*\cap B_t(o)$  such that
$s_i^{-1}\r_i^{-1}(M_{i}\cap B_{ts_i\r_{i}}(y_i))$ converges in the induced Hausdorff sense to $ C^*\cap B_t(o)$.
From Theorem \ref{LimMiMib} and \eqref{dGHrkskt0001}\eqref{dHrkskt0001}, given $0<t_1<t_2<t$ for each $k$ we have
\begin{equation}\aligned
\limsup_{i\to\infty}s_i^{-n}\mathcal{H}^n(M_{i}\cap\overline{B_{t_2s_{i}\r_k}(p_i)})\le\mathcal{H}^n( M^*\cap\overline{B_{t_2\r_k}(o)})
\endaligned
\end{equation}
and
\begin{equation}\aligned
\mathcal{H}^n( M^*\cap B_{t_1\r_k}(o))\le\liminf_{i\to\infty}s_i^{-n}\mathcal{H}^n(M_{i}\cap B_{t_1s_{i}\r_k}(p_i)).
\endaligned
\end{equation}
With \eqref{limrkrto0+}, by taking the diagonal subsequence, (up to choose the subsequences) we have
\begin{equation}\aligned\label{cpdejMkjadd}
t_2^{-n}\liminf_{k\to\infty}s_k^{-n}\r_k^{-n}\mathcal{H}^n\left(M_{k}\cap\overline{B_{t_2s_{k}\r_k}(p_k)}\right)\le t_1^{-n} \limsup_{k\to\infty}s_k^{-n}\r_k^{-n}\mathcal{H}^n(M_{k}\cap B_{t_1s_{k}\r_k}(p_k)),
\endaligned
\end{equation}
and still have
\begin{equation}\aligned\label{dGHrkskt0002}
&\lim_{k\to\infty}d_{GH}\left(s_k^{-1}\r_k^{-1}B_{ts_k\r_k}(y_{k}),B_{t}(o)\right)\\
=&\lim_{k\to\infty}d_{H}\left(s_k^{-1}\r_k^{-1}\Phi_{k,k}(\Phi_{k,k,k}(M_k\cap B_{ts_k\r_{k}}(y_k))),C^*\cap B_{t}(o)\right)=0
\endaligned
\end{equation}
from \eqref{dGHrkskt0001}\eqref{dHrkskt0001}.
Combining Theorem \ref{LimMiMib}, \eqref{cpdejMkjadd} and \eqref{dGHrkskt0002}, we conclude that 
$$t_2^{-n}\mathcal{H}^n(C^*\cap B_{t_2}(o))\le t_1^{-n}\mathcal{H}^n\left(C^*\cap \overline{B_{t_1}(o)}\right).$$
From Lemma \ref{pM+Minfty}, $\tau^{-n}\mathcal{H}^n(C^*\cap B_{\tau}(o))$ is a constant on $(0,t)$.
From arbitrariness of $t$ and the uniqueness of solutions of minimal hypersurface equation, we conclude that $C^*$ is a metric cone in $C$.

From  \eqref{GHrjykl000} and  \eqref{eq510add000}, $\f1{s_j\r_j}B_{ts_j\r_j}(y_\infty)$ converges to $B_t(o)$ in the Gromov-Hausdorff sense, and $\f1{s_j\r_j}(M_\infty\cap B_{ts_j\r_j}(y_\infty))$ converges to $C^*\cap B_t(o)$ in the induced Hausdorff sense. There is a subsequence $\tau_j$ of $s_j\r_j$ so that $\f1{\tau_j}B_{t^2\tau_j}(y_\infty)$ converges to $B_{t^2}(o)$ in the Gromov-Hausdorff sense, and $\f1{\tau_j}(M_\infty\cap B_{t^2\tau_j}(y_\infty))$ converges to $C^*\cap B_{t^2}(o)$ in the induced Hausdorff sense. Then we extract the subsequence of $\tau_j$ for the case $t^3$ (namely, $B_{t^3}(o)$ and $C^*\cap B_{t^3}(o)$), and continue to do this process inductively for $t^l$ until $l\to\infty$. 
By taking the diagonal subsequence, there is a subsequence $r_j\to0^+$ as $j\to\infty$ so that $\f1{r_j}(B_1(p_\infty),y_\infty)$ converges to $(C,o)$ in the pointed Gromov-Hausdorff sense,
and $\f1{r_j}(M_\infty,y_\infty)$ converges in the induced Hausdorff sense to $(C^*,o)$.
This completes the proof.
\end{proof}
\textbf{Remark.}
In general, I do not know whether $M^*$ is a metric cone in $C$ in Theorem \ref{Covcone}. I would like to thank Gioacchino Antonelli and Daniele Semola for indicating this problem, which I overlooked in the last version.

Using Lemma \ref{almostmono}, we have the uniqueness of the limit for the density function of $M_\infty$ in a special case, which acts a key role in studying the regularity of $M_\infty$ in $B_1(p_\infty)$.
\begin{lemma}\label{4EQUIV}
For any $x\in M_\infty\cap \mathcal{R}$, if $\liminf_{r\rightarrow0}r^{-n}\mathcal{H}^n(M_\infty\cap B_{r}(x))=\omega_n$, then $\lim_{r\rightarrow0}r^{-n}\mathcal{H}^n(M_\infty\cap B_{r}(x))=\omega_n$, and for any sequence $s_i\rightarrow0^+$ there is a subsequence $s_{i_j}$ so that $s_{i_j}^{-1}(M_\infty,x)$ converges to $(\R^n,0)$ in $\R^{n+1}$ in the induced Hausdorff sense.
\end{lemma}
\begin{proof}
From the assumption, there is a positive sequence $r_j\rightarrow0$ so that
$\lim_{j\rightarrow\infty}r_j^{-n}\mathcal{H}^n(M_\infty\cap B_{r_j}(x))=\omega_n$.
Let $\th_n$ be the constant defined as \eqref{DEFthn} such that
\begin{equation}\aligned\label{Cep*}
\mathcal{H}^n(C\cap B_{1}(0))\ge(1+\th_n)\omega_n
\endaligned
\end{equation}
for each non-flat area-minimizing hypercone $C$ in $\R^{n+1}$ with the vertex at the origin.
Let $x_i\in M_i$ with $x_i\to x$ as $i\to\infty$.
Given a small constant $\ep\in(0,\de_n)$ with $(1-\ep^2)^{-n-2}\le1+\ep/2$. Since $M_i$ converges to $M_\infty\subset \overline{B_1(p_\infty)}$ in the induced Hausdorff sense, from Theorem \ref{LimMiMib} we have
\begin{equation}\aligned
\mathcal{H}^n(M_i\cap B_{(1-\ep^2)r_j}(x_i))<(1-\ep^2)^{-1}\mathcal{H}^n(M_\infty\cap B_{r_j}(x))
\endaligned
\end{equation}
for all sufficiently large $i$. Hence, there is an integer $j_\ep$ so that for each integer $j\ge j_\ep$
\begin{equation}\aligned
\mathcal{H}^n(M_i\cap B_{(1-\ep^2)r_j}(x_i))<(1-\ep^2)^{-2}\omega_nr_j^n\le(1+\ep/2)\omega_n\left((1-\ep^2)r_j\right)^n
\endaligned
\end{equation}
for all sufficiently large $i$ (depending on $j$).
From Lemma \ref{almostmono}, (up to a choice of $j_\ep$) for each $s\in(0,(1-\ep^2)r_j]$ with $j\ge j_\ep$ we have
\begin{equation}\aligned\label{MiBsxiep<}
\mathcal{H}^n(M_i\cap B_{s}(x_i))<(1+\ep)\omega_ns^n
\endaligned
\end{equation}
for all sufficiently large $i$ (depending on $j$). From \eqref{MiBsxiep<} and Theorem \ref{LimMiMib} again, we have
\begin{equation}\aligned
s^{-n}\mathcal{H}^n(M_\infty\cap B_{s}(x))\le s^{-n}\liminf_{i\to\infty}\mathcal{H}^n(M_i\cap B_{s}(x_i))\le(1+\ep)\omega_n.
\endaligned
\end{equation}
Letting $s\to0$ and then $\ep\to0$ imply
$$\lim_{r\rightarrow0}s^{-n}\mathcal{H}^n(M_\infty\cap B_{s}(x))\le\omega_n.$$
Combining Theorem \ref{Conv0} and Theorem \ref{LimMiMib}, for any sequence $s_i\rightarrow0^+$ there is a subsequence $s_{i_j}$
so that $s_{i_j}^{-1}(M_\infty,x)$ converges to $(\R^n,0)$ in $\R^{n+1}$ in the induced Hausdorff sense. This completes the proof.
\end{proof}

Inspired by Allard's regularity theorem and the work on the singular sets by Cheeger-Colding \cite{CCo1},
we define some approximate sets related to the singular sets of area-minimizing hypersurfaces in metric spaces as follows.
Let $\mathcal{S}_{M_\infty,\ep,r}$ denote the subsets in the regular set $\mathcal{R}\subset B_1(p_\infty)$ containing all the points $y$ satisfying
$$\mathcal{H}^n\left(M_\infty\cap B_{s}(y)\right)\ge(1+\ep)\omega_n s^n\qquad \mathrm{for\ some}\ s\in(0,r].$$
Note that $M_\infty\subset\overline{B_1(p_\infty)}$. Let
$$\mathcal{S}_{M_\infty,\ep}=\bigcap_{0<r\le1}\mathcal{S}_{M_\infty,\ep,r},
\qquad\mathcal{S}_{M_\infty}=\bigcup_{\ep>0}\mathcal{S}_{M_\infty,\ep}=\bigcup_{\ep>0}\bigcap_{0<r\le1}\mathcal{S}_{M_\infty,\ep,r}.$$
From Lemma \ref{4EQUIV}, every tangent cone of $M_\infty$ at each point in $\mathcal{S}_{M_\infty}$ is not a hyperplane.
Let $\th_n>0$ be the constant defined in \eqref{DEFthn}.
From Lemma \ref{4EQUIV} again, for any $x\in\mathcal{S}_{M_\infty,\ep}$ with $\ep<\th_n$ there is a constant $r_x>0$ such that
\begin{equation}\aligned
\mathcal{H}^n\left(M_\infty\cap B_{s}(x)\right)\ge(1+2\ep/3)\omega_n s^n\qquad \mathrm{for\ any}\ s\in(0,r_x].
\endaligned
\end{equation}
Noting that $\mathcal{S}_{M_\infty,\ep}$ may not be closed as $\mathcal{R}$ may not be closed in $B_1(p_\infty)$.
However, from Theorem \ref{IntReg} and Theorem \ref{LimMiMib}, there is a constant $r_x^*>0$ such that
\begin{equation}\aligned\label{MinfBsy1ep2}
\mathcal{H}^n\left(M_\infty\cap B_{s}(y)\right)\ge(1+\ep/2)\omega_n s^n\qquad \mathrm{for\ any}\ s\in(0,r_x^*], \ y\in B_{r_x^*}(x)\cap\mathcal{S}_{M_\infty,\ep}.
\endaligned
\end{equation}
We call $\mathcal{S}_{M_\infty}$ \emph{the singular set} of $M_\infty$ in $\mathcal{R}$.
Denote $\mathcal{R}_{M_\infty,\ep,r}=M_\infty\cap\mathcal{R}\setminus\mathcal{S}_{M_\infty,\ep,r}$. If $y\in\mathcal{R}_{M_\infty,\ep,r}$, then there is a number $s\in(0,r]$ so that
$$\mathcal{H}^n\left(M_\infty\cap B_{s}(y)\right)<(1+\ep)\omega_n s^n.$$
Set
$$\mathcal{R}_{M_\infty}=M_\infty\cap\mathcal{R}\setminus\mathcal{S}_{M_\infty}=\bigcap_{\ep>0}\bigcup_{0<r\le1}\mathcal{R}_{M_\infty,\ep,r}.$$
From Lemma \ref{4EQUIV}, for each $x\in \mathcal{R}_{M_\infty}$, any tangent cone of $M_\infty$ at $x$ is a hyperplane in $\R^{n+1}$.

Now let us prove Theorem \ref{dimestn-7n-2}, which is divided into the following two lemmas.
\begin{lemma}\label{codim7}
$\mathcal{S}_{M_\infty}$ has Hausdorff dimension $\le n-7$ for $n\ge7$, and it is empty for $n<7$.
\end{lemma}
\begin{proof}
For the case $n\ge 7$, we only need to show $\mathrm{dim}\mathcal{S}_{M_\infty,\ep}\le n-7$ for each $\ep>0$.
Suppose there exists a constant $\be>n-7$ (maybe not integer) such that $\mathcal{H}^\be\left(\mathcal{S}_{M_\infty,\ep}\right)>0$.
Then $\mathcal{H}^\be_\infty\left(\mathcal{S}_{M_\infty,\ep}\right)=\lim_{t\rightarrow\infty}\mathcal{H}^\be_t\left(\mathcal{S}_{M_\infty,\ep}\right)>0$ from Lemma 11.2 in \cite{Gi}.
By the argument of Proposition 11.3 in \cite{Gi}, there is a point $x_0\in \mathcal{S}_{M_\infty,\ep}$ and a sequence $r_j\rightarrow0$ such that
\begin{equation}\aligned
\mathcal{H}^\be_\infty\left(\mathcal{S}_{M_\infty,\ep}\cap B_{r_j}(x_0)\right)>2^{-\be-1}\omega_\be r_j^\be.
\endaligned
\end{equation}
Denote $(N_j^*,x_j)=\f1{r_j}(B_1(p_\infty),x_0)$, and $M^j_\infty=\f1{r_j}M_\infty$. Then for the ball $B_{1}(x_j)\subset N_j^*$ we have
\begin{equation}\aligned
\mathcal{H}^\be_\infty\left(\mathcal{S}_{M^j_\infty,\ep}\cap B_{1}(x_j)\right)>2^{-\be-1}\omega_\be.
\endaligned
\end{equation}
Without loss of generality, by Theorem \ref{Conv0} we assume that $(M^j_\infty,x_j)$ converges as $j\rightarrow\infty$ in the induced Hausdorff sense to  $(M^*_\infty,0)$ in $\R^{n+1}$, where $M^*_\infty$ is an area-minimizing hypersurface through the origin in $\R^{n+1}$.
If $y_j\in \mathcal{S}_{M^j_\infty,\ep}\cap B_{1}(x_j)$ and $y_j\rightarrow y_*\in M^*_\infty$, then from \eqref{MinfBsy1ep2} and the definition of $M^j_\infty$ one has
$$\mathcal{H}^n\left(M^j_\infty\cap B_{s}(y_j)\right)\ge(1+\ep/2)\omega_n s^n$$
for any $0<s\le1$ and the sufficiently large $j$. Combining Theorem \ref{Conv0} and \eqref{EQUIVMEA}, taking the limit in the above inequality implies
$$\mathcal{H}^n\left(M^*_\infty\cap B_{s}(y_*)\right)\ge(1+\ep/2)\omega_n s^n.$$
Hence, we conclude that $y_*$ is a singular point of $M^*_\infty$. From the proof of Lemma 11.5 in \cite{Gi}, we have
\begin{equation}\aligned
\mathcal{H}^\be_\infty\left(\mathcal{S}_{M^*_\infty,\ep}\cap B_{1}(0)\right)>2^{-\be-1}\omega_\be.
\endaligned
\end{equation}
Now we can use Theorem 11.8 in \cite{Gi} to get a contradiction. Hence, $\mathcal{S}_{M_\infty}$ has Hausdorff dimension $\le n-7$ for $n\ge7$. For $n<7$, we can use the above argument to show $\mathcal{S}_{M_\infty}$ empty. This completes the proof.
\end{proof}

\begin{lemma}\label{codim2}
$\mathcal{S}\cap M_\infty$ has Hausdorff dimension $\le n-2$.
\end{lemma}
\begin{proof}
Let us prove it by contradiction.
Assume there is a constant $\de\in(0,1]$ such that
$\mathcal{H}^{n-2+\de}(\mathcal{S}\cap M_\infty)\ge\de$.
From the definition of $\mathcal{S}^k$ in \eqref{Sk} and dim$(\mathcal{S}^k)\le k$ for each nonnegative integer $k\le n-1$,
the set $\Lambda_{M_\infty}\triangleq M_\infty\cap\mathcal{S}^{n-1}$ satisfies $\mathcal{H}^{n-2+\de}(\La_{M_\infty})\ge\de$.
From Proposition 11.3 in \cite{Gi},
for $\mathcal{H}^{n-2+\de}$-almost every point $y\in\La_{M_\infty}$, there is a sequence $r_i\rightarrow0$ such that
\begin{equation}\aligned\label{ri2-n-deLaMinfty}
\lim_{i\rightarrow\infty}r_i^{2-n-\de}\mathcal{H}^{n-2+\de}_\infty(\La_{M_\infty}\cap B_{r_i}(y))>0.
\endaligned
\end{equation}
By Gromov's compactness theorem, we can assume that $\f1{r_i}(B_1(p_\infty),y)$ converges to a metric cone $(Y_y,o_y)$ in the pointed Gromov-Hausdorff sense,
$\f1{r_i}(M_\infty,y)$ converges to $(M_{y,\infty},o_y)$ with closed $M_{y,\infty}$ in $Y_y$, and $\f1{r_i}\left(\La_{M_\infty},y\right)$ converges to $(\La_{y,M_\infty},o_y)$ with closed $\La_{y,M_\infty}\subset M_{y,\infty}$ in the induced Hausdorff sense. We write $Y_y=\R^{n-1}\times CS^1_r$ for some metric cone $CS^1_r$ with the cross section $S^1_r$, where $S^1_r$ is the round circle with the radius $r\in(0,1)$.
By the definition of $\Lambda_{M_\infty}$, $\La_{y,M_\infty}$ is contained in the singular set $\R^{n-1}\times\{o_r\}$ of $Y_y$, where $o_r$ is the vertex of $CS^1_r$.
From \eqref{ri2-n-deLaMinfty} and the proof of Lemma 11.5 in \cite{Gi}, we get
\begin{equation}\aligned\label{Hn-2deyMinfrr}
\mathcal{H}^{n-2+\de}_\infty(\La_{y,M_\infty}\cap B_{1}(0^{n-1},o_r))>0.
\endaligned
\end{equation}

From Proposition 11.3 in \cite{Gi} again,
there are a point $z\in\La_{y,M_\infty}\setminus\{o_y\}\subset\R^{n-1}\times\{o_r\}$ and a sequence $\tau_i\rightarrow0$ such that
\begin{equation}\aligned\label{abssss}
\lim_{i\rightarrow\infty}\tau_i^{2-n-\de}\mathcal{H}^{n-2+\de}_\infty(\La_{y,M_\infty}\cap B_{\tau_i}(z))>0.
\endaligned
\end{equation}
Up to choose the subsequence, we assume that $\f1{\tau_i}(M_{y,\infty},z)$ converges to a metric cone $(M_{z,y,\infty},z)$ in $\R^{n-1}\times CS^1_r$ from Theorem \ref{Covcone}, and
$\f1{\tau_i}(\La_{y,M_\infty},z)$ converges to a closed set $(\La_{z,y,\infty},z)$ in $\R^{n-1}\times\{o_r\}\subset\R^{n-1}\times CS^1_r$.
Moreover, $M_{z,y,\infty}$ and $\La_{z,y,\infty}$ both split off the same line (through $z$) isometrically.
From \eqref{Hn-2deyMinfrr} and the proof of Lemma 11.5 in \cite{Gi}, we get
\begin{equation}\aligned
\mathcal{H}^{n-2+\de}_\infty(\La_{z,y,M_\infty}\cap B_{1}(z))>0.
\endaligned
\end{equation}
By dimension reduction argument, there is a metric cone $C_y\subset\R^{n-1}\times CS^1_r$ with $\R^{n-1}\times \{o_r\}\subset C_y$,
and $C_y$ splits off a factor $\R^{n-1}$ isometrically.
Moreover, there are a sequence of $(n+1)$-dimensional complete Riemannian manifolds $Q_i$ with Ricci curvature $\ge-(n-1)R_i^{-2}$ on $B_{R_i}(y_i)\subset Q_i$ for some sequence $R_i\rightarrow\infty$ such that $(Q_i,y_i)$ converges to the cone $(Y_y,o_y)$ in the pointed Gromov-Hausdorff sense,
and a sequence of area-minimizing hypersurfaces $\Si_i$ in $B_{R_i}(y_i)$ with $\p \Si_i\subset\p B_{R_i}(y_i)$ such that $(\Si_i,y_i)$ converges in the induced Hausdorff sense to $(C_y,o_y)$. In particular, $\p C_y=\emptyset$.

There is a 1-dimensional cone $C_y'\subset CS^1_r$ with $\p C_y'=\emptyset$ such that $C_{y}=\R^{n-1}\times C_y'$.
Let $\Om$ be a domain (connected open set) in $CS^1_r$ with boundary in $C_y'$. Then $\p\Om\cap\p B_1(o_r)$ are two points, denoted by $\a^+,\a^-$.
Then there is a minimizing geodesic $\g$ in $\overline{\Om}$ connecting $\a^+,\a^-$, and there is a constant $\th_r>0$ so that the length of $\g$ satisfies $\mathcal{H}^1(\g)\le2-\th_r$. Let $U$ denote the bounded domain in $\Om\subset CS^1_r$ enclosed by $C_y'$ and $\g$.
For any $\ep\in(0,1)$, let $\ep\g$ denote a minimizing geodesic connecting $\ep\a^+,\ep\a^-$, and $\ep U$ denote the bounded domain in $CS^1_r$ enclosed by $C_y'$ and $\ep\g$.
For any $t>0$, we define a domain $\Om_{t,\ep}$ in $Y_y\setminus(\R^{n-1}\times \{o_r\})$ by
$$\Om_{t,\ep}=\{(\xi,x)\in \R^{n-1}\times \Om|\, \ep<d(x,o_r)<1,|\xi|<t\}.$$
From \eqref{MSHS} and the proof of Proposition \ref{MinMSinfty},
$C_y$ is an area-minimizing hypersurface in $\overline{\Om_{t,\ep}}$ for any $t>0,\ep\in(0,1)$.
Let $S_*=\{(\xi,x)\in \R^{n-1}\times CS^1_r|\, x\in\g\cup\ep\g,|\xi|<t\}$, and $S^*=\{(\xi,x)\in \R^{n-1}\times CS^1_r|\, x\in \overline{U}\setminus\ep U,|\xi|=t\}$.
Let $S=S_*\cup S^*$ be a set in $\overline{\Om_{t,\ep}}$, then
\begin{equation}\aligned
\p S=&\left\{(\xi,x)\in \R^{n-1}\times \overline{\Om}|\, \ep\le d(x,o_r)\le1,|\xi|=t\right\}\\
&\cup\left\{(\xi,x)\in \R^{n-1}\times \overline{\Om}|\, d(x,o_r)=\ep\ or\ 1,|\xi|\le t\right\}=C_y\cap\p\Om_{t,\ep}.
\endaligned
\end{equation}
Minimizing $C_y$ in $\overline{\Om_{t,\ep}}$ implies
\begin{equation}\aligned
2(1-\ep)\omega_{n-1}t^{n-1}=&\mathcal{H}^{n}\left(C_y\cap \overline{\Om_{t,\ep}}\right)\le\mathcal{H}^{n}(S)=\mathcal{H}^{n}(S_*)+\mathcal{H}^{n}(S^*)\\
\le&(1+\ep)(2-\th_r)\omega_{n-1}t^{n-1}+\pi(n-1)\omega_{n-1}t^{n-2}.
\endaligned
\end{equation}
The above inequality is impossible for the suitable large $t>0$ and the suitable small $\ep>0$.
This completes the proof.
\end{proof}

\section{Appendix I}

Compared with Schoen-Yau's argument on Laplacian of distance functions from fixed points (see Proposition 1.1 in \cite{SY2}),
the inequality \eqref{Hxtge} holds in the distribution sense as follows.
\begin{lemma}\label{DerMge*}
Let $N$ be an $(n+1)$-dimensional complete Riemannnian manifolds with Ricci curvature $\ge-n\de^2_N$ on $B_R(p)$. If $V$ is an $n$-rectifiable stationary varifold in $B_{R}(p)$ with $M=\mathrm{spt} V\cap B_R(p)$, then
\begin{equation}\aligned\label{DerMge}
\De_N\r_M\le n\de_N\tanh\left(\de_N\r_M\right)\qquad on\ B_{R-t}(p)\cap B_t(M)\setminus M
\endaligned
\end{equation}
for all $0<t<R$ in the distribution sense.
\end{lemma}
\begin{proof}
For $\mathcal{H}^n$-a.e. $x\in \mathrm{spt}V\cap B_R(p)$, the tangent cone of $V$ at $x$ is a hyperplane in $\R^{n+1}$ with multiplicity $\th_x>0$.
Let $\th$ denote the multiplicity function of $V$, then $\th(x)=\th_x$ for $\mathcal{H}^n$-a.e. $x$.
Let $\mu_V$ denote the Radon measure associated with $V$ defined by $\mu_V=\mathcal{H}^n\llcorner\th$, then $\lim_{r\rightarrow0}r^{-n}\mu_V(B_r(x))=\th_x\omega_n$.
Note that the density of stationary varifolds in Euclidean space is upper semicontinuous (see \cite{S}). Similarly, we can let $\th$ be the density of $V$, 
and $\th$ is upper semicontinuous on $\mathrm{spt}V\cap B_R(p)$.
By Allard's regularity theorem, $M$ is smooth in a neighborhood of $x$. From constancy theorem (see 
\cite{S}), $M$ is a minimal hypersurface in a neighborhood of $x$.
Hence, the singular set $\mathcal{S}_M$ of $M$ is closed in $B_R(p)$.

For any $i\in \mathbb{N}^+$, there is a covering $\{B_{r_{i,j}}(x_{i,j})\}_{j=1}^{l_i}$ of $\mathcal{S}_M$ with $\lim_{i\rightarrow\infty}l_i=\infty$ such that
$\omega_n\sum_{j=1}^{l_i} r_{i,j}^n<2^{-i}$.
Hence there is a sequence of smooth embedded hypersurfaces $M_i$ converging to $M$ in the Hausdorff sense with $M_i=M$ outside $\bigcup_{j=1}^{l_i} B_{r_{i,j}}(x_{i,j})$. Let $\r_{M_i}$ be the distance function from $M_i$.
Let $K$ be a closed set in $B_{R-t}(p)\cap B_t(M)\setminus M$ for some fixed $t\in(0,R)$.
We claim
\begin{equation}\label{equivMiM}
\r_{M_i}=\r_{M}\qquad \mathrm{on} \ \ K
\end{equation}
for the sufficiently large $i$ depending on $n,t,R,K$.
Let us prove \eqref{equivMiM}. For any $x\in K$, there is a point $y_x\in M$ such that $\r_M(x)=d(x,y_x)$. Then any tangent cone of $M$ at $y_x$ is on one side of $\R^n$ in $\R^{n+1}$. By maximum principle of (singular) minimal hypersurfaces, the tangent cone of $M$ at $y_x$ is $\R^n$, which implies $y_x\in \mathcal{R}_{M}= M\setminus \mathcal{S}_M$. Then $y_x\varsubsetneq \bigcup_{j=1}^{l_i} B_{r_{i,j}}(x_{i,j})$ for the sufficiently large $i$ depending on $n,t,R,K$, which implies
$y_x\in M_i$ and
$d(x,y)>\r_M(x)$ for any $y\in\bigcup_{j=1}^{l_i} B_{r_{i,j}}(x_{i,j})$. Hence it follows that $\r_{M_i}(x)=d(x,y_x)=\r_M(x)$, and this confirms the claim \eqref{equivMiM}.

Let $\mathcal{C}_{M}$ denote the cut locus of $\r_{M}$, the set
including all the points which are joined to $M$ by two minimizing geodesics at least.
From the above argument and \eqref{Hxtge}, we immediately have
\begin{equation}\label{DeMideN}
-\De\r_{M}\ge-n\de_N\tanh\left(\de_N\r_{M}\right)\qquad \mathrm{on}\ B_{R-t}(p)\cap B_t(M)\setminus(M\cup\mathcal{C}_{M})
\end{equation}
for all $0<t<R$.
Let $\mathcal{C}_{M_i}$ denote the cut locus of $\r_{M_i}$ for each $i$.
For any closed set $K\subset B_{R-t}(p)\cap B_t(M)\setminus M$, $\mathcal{C}_{M_i}\cap K$ is a closed set of Hausdorff dimension $\le n$ and $\mathcal{C}_{M_i}\cap K$ is smooth up to a zero set of Hausdorff dimension $\le n$ (see \cite{MM} for instance).
From \eqref{equivMiM}, $\mathcal{C}_{M}\cap K=\mathcal{C}_{M_i}\cap K$ for the sufficiently large $i$. Hence, $\mathcal{C}_{M}$ is closed in $B_{R-t}(p)\cap B_t(M)\setminus M$ of Hausdorff dimension $\le n$ and $\mathcal{C}_{M}$ is smooth up to a zero set of Hausdorff dimension $\le n$.
Let $\mathcal{E}_s\triangleq B_s(\mathcal{C}_{M})$ denote the $s$-neighborhood of $\mathcal{C}_{M}$ in $N$, then (as the Hausdorff measure is Borel-regular)
\begin{equation}\aligned\label{VOlKt0}
\lim_{s\rightarrow0}\mathcal{H}^{n+1}(\mathcal{E}_s)=0.
\endaligned
\end{equation}
From co-area formula and Theorem 4.4 in \cite{Gi}, $\p\mathcal{E}_s$ is countably $n$-rectifiable for almost all $s>0$.
For any point $y\in\p\mathcal{E}_s$, there is a unique normalized geodesic $\g_y\subset N$ connecting $\g(0)\in M$ and $\g_y(\r_M(y)+s)\in \mathcal{C}_{M}$ with $y=\g_y(\r_M(y))$.
In particular, (compared with the argument of the proof of Proposition 1.1 in \cite{SY2})
\begin{equation}\aligned\label{nuijge0}
\lan \nu_{s},\na\r_{M}\ran\ge0 \qquad \mathcal{H}^n-a.e.\ \ \mathrm{on}\ \p \mathcal{E}_s,
\endaligned
\end{equation}
where $\nu_s$ is the inner unit normal vector $\mathcal{H}^n$-a.e. to $\p \mathcal{E}_s$.

Let $U$ be an open set in $B_{R-t}(p)\cap B_t(M)\setminus M$.
Let $\phi$ be a nonnegative Lipschitz function on $U$ with $\phi=0$ on $\p U$, then with \eqref{VOlKt0}
\begin{equation}\aligned
\int_U\left\lan\na\r_{M},\na\phi\right\ran=\lim_{s\rightarrow0}\int_{U\setminus \mathcal{E}_s}\left\lan\na\r_{M},\na\phi\right\ran.
\endaligned
\end{equation}
With \eqref{DeMideN} and \eqref{nuijge0}, integrating by parts implies
\begin{equation}\aligned
\int_{U\setminus \mathcal{E}_s}\left\lan\na\r_{M},\na\phi\right\ran=&\int_{U\cap\p \mathcal{E}_s}\phi\lan \nu_s,\na\r_{M}\ran-\int_{U\setminus \mathcal{E}_s}\phi\De\r_{M}\\
\ge&-n\de_N\int_{U\setminus \mathcal{E}_s}\phi\tanh\left(\de_N\r_M\right).
\endaligned
\end{equation}
Letting $s\rightarrow0$ in the above inequality infers
\begin{equation}\aligned
\int_U\left\lan\na\r_{M},\na\phi\right\ran\ge-n\de_N\int_U\phi\tanh\left(\de_N\r_{M}\right).
\endaligned
\end{equation}
This completes the proof.
\end{proof}
In general, the inequality \eqref{DerMge} cannot hold true in $B_{R-t}(p)\cap B_t(M)$. For example, let $N$ be the standard Euclidean space, $M$ be the $\R^n\times\{0\}\subset\R^{n+1}=N$, then $\r_M(x_1,\cdots,x_{n+1})=|x_{n+1}|$. Clearly, $|x_{n+1}|$ is not a superharmonic function on $\R^{n+1}$ in the distribution sense.

\section{Appendix II}

Let $N_i$ be a sequence of $(n+1)$-dimensional smooth Riemannian manifolds with $\mathrm{Ric}\ge-n\k^2$ on the metric ball $B_{1+\k'}(p_i)\subset N_i$ for constants $\k\ge0$, $\k'>0$.
Up to choose the subsequence, we assume that $\overline{B_1(p_i)}$ converges to a metric ball $\overline{B_1(p_\infty)}$ in the Gromov-Hausdorff sense.
Namely, there is a sequence of $\ep_i$-Hausdorff approximations $\Phi_i:\, B_1(p_i)\rightarrow B_1(p_\infty)$ for some sequence $\ep_i\rightarrow0$.
Let $\nu_\infty$ denote the renormalized limit measure on $B_1(p_\infty)$ obtained from the renormalized measures as \eqref{nuinfty}.
For any set $K$ in $\overline{B_1(p_\infty)}$, let $B_\de(K)$ be the $\de$-neighborhood of $K$ in $\overline{B_1(p_\infty)}$ defined by $\{y\in \overline{B_1(p_\infty)}|\ d(y,K)<\de\}$.
Here, $d$ denotes the distance function on $\overline{B_1(p_\infty)}$.

Let us introduce several useful results as follows. For the completeness and self-sufficiency, we give the proofs here.
\begin{lemma}\label{nuinftyKthj}
Let $K$ be a closed subset of $B_1(p_\infty)$ with $\inf_{x\in K}d(x,\p B_1(p_\infty))=\ep_0>0$.
For each $\ep\in(0,\ep_0/3]$, there is a sequence of mutually disjoint balls $\{B_{\th_j}(x_j)\}_{j=1}^\infty$ with $x_j\in K$ and $\th_j<\ep$
such that $K\subset \bigcup_{1\le j\le k}B_{\th_j+2\th_k}(x_j)$ for the sufficiently large $k$, and
\begin{equation}\aligned\label{nuinfKUP}
\nu_\infty(K)\le\sum_{j=1}^{\infty}\nu_\infty\left(B_{\th_j}(x_j)\right).
\endaligned
\end{equation}
Moreover, if $\mathcal{H}^m(K)<\infty$ and $\inf\{r^{-m}\mathcal{H}^m(K\cap B_r(x))|\, B_r(x)\subset B_1(p_\infty)\}>0$ for some integer $0<m\le n+1$, then
\begin{equation}\aligned\label{Hkepkjinfthjk}
\mathcal{H}^m_\ep(K)\le\omega_m\sum_{j=1}^{\infty}\th_j^m.
\endaligned
\end{equation}
\end{lemma}
\begin{proof}
For each $\ep>0$, let $\mathcal{U}_\ep$ be a collection of balls defined by
$$\mathcal{U}_\ep=\left\{B_r(x)\subset B_1(p_\infty)\big|\ x\in K,\ 0<r\le\ep\right\}.$$
Now we adopt the idea in the proof of Vitali's covering lemma.
First, we take $B_{\th_1}(x_1)\subset\mathcal{U}_\ep$ such that $\th_1$ is the largest radius of balls belonging to $\mathcal{U}_\ep$.
Suppose that $B_{\th_1}(x_1),B_{\th_2}(x_2),\cdots$, $B_{\th_{k-1}}(x_{k-1})$ have already been chosen.
Then we select $B_{\th_k}(x_k)\subset\mathcal{U}_\ep$ such that $\th_k$ is the largest radius of balls belonging to $\mathcal{U}_\ep$
with $B_{\th_k}(x_k)\cap B_{\th_j}(x_j)=\emptyset$ for each $1\le j\le k-1$.
Hence we can obtain an infinite sequence of mutually disjoint balls $\{B_{\th_j}(x_j)\}_{j\ge1}$.
From the choice of $\th_j$, $\lim_{j\rightarrow\infty}\th_j=0$ and $\th_j\ge \th_{j+1}$ for each $j\ge1$.
Then there is an integer $N_\th>0$ such that $\th_{k}<\ep$ for all $k\ge N_\th$.
For a point $x\in K\setminus\bigcup_{j=1}^{k}B_{\th_j}(x_j)$ with $k\ge N_\th$, $\bigcup_{j=1}^{k}\overline{B_{\th_j}(x_j)}\cap \overline{B_{\th_k}(x)}\neq\emptyset$
according to the choice of $\{B_{\th_j}(x_j)\}_{j=1}^k$, which implies $x\in\bigcup_{j=1}^{k}B_{\th_j+2\th_k}(x_j)$. Hence,
\begin{equation}\aligned\label{detauFinfty}
K\subset\bigcup_{j=1}^{k}B_{\th_j+2\th_k}(x_j).
\endaligned
\end{equation}
Combining \eqref{detauFinfty} and Bishop-Gromov comparison, there is a constant $\be_{n,\k}>0$ depending only on $n,\k$ such that
\begin{equation}\aligned\label{nuinf*}
\nu_\infty(K)\le&\sum_{j=1}^{k}\nu_\infty\left(B_{\th_{j}+2\th_k}(x_j)\right)\le\sum_{j=1}^{k}\f{V^{n+1}_\k(\th_j+2\th_k)}{V^{n+1}_\k(\th_j)}\nu_\infty\left(B_{\th_j}(x_j)\right)\\
<&\sum_{j=1}^{k}\left(1+\be_{n,\k}\f{\th_k}{\th_j}\right)\nu_\infty\left(B_{\th_j}(x_j)\right),
\endaligned
\end{equation}
where $V^{n+1}_\k(r)$ denotes the volume of a geodesic ball with radius $r$ in an $(n+1)$-dimensional space form with constant sectional curvature $-\k^2$.
For any $\de>0$, from $\nu_\infty(B_1(p_\infty))=1$ there is a constant $N_\th'\ge N_\th$ so that
$$\sum_{j=N_\th'}^\infty \nu_\infty\left(B_{\th_j}(x_j)\right)<\de.$$
Combining $\th_j\ge \th_{j+1}$, we have
\begin{equation}\aligned\label{knuinf000*}
&\sum_{j=1}^{k}\f{\th_k}{\th_j}\nu_\infty\left(B_{\th_j}(x_j)\right)
\le\sum_{j=1}^{N_\th'-1}\f{\th_k}{\th_j}\nu_\infty\left(B_{\th_j}(x_j)\right)+\sum_{j=N_\th'}^k\f{\th_k}{\th_j}\nu_\infty\left(B_{\th_j}(x_j)\right)\\
\le&\f{\th_k}{\th_{N_\th'}}\sum_{j=1}^{N_\th'-1}\nu_\infty\left(B_{\th_j}(x_j)\right)+\sum_{j=N_\th'}^k\nu_\infty\left(B_{\th_j}(x_j)\right)
\le\f{\th_k}{\th_{N_\th'}}+\de
\endaligned
\end{equation}
for each $k\ge N_\th'$. With $\lim_{j\rightarrow\infty}\th_j=0$, it is easy to see
\begin{equation}\aligned\label{knuinf000}
\lim_{k\rightarrow\infty}\sum_{j=1}^{k}\f{\th_k}{\th_j}\nu_\infty\left(B_{\th_j}(x_j)\right)=0.
\endaligned
\end{equation}
Combining \eqref{nuinf*}\eqref{knuinf000} we have
\begin{equation}\aligned
\nu_\infty(K)\le\lim_{k\rightarrow\infty}\sum_{j=1}^{k}\left(1+\be_{n,\k}\f{\th_k}{\th_j}\right)\nu_\infty\left(B_{\th_j}(x_j)\right)
=\sum_{j=1}^{\infty}\nu_\infty\left(B_{\th_j}(x_j)\right).
\endaligned
\end{equation}

Now we assume $\mathcal{H}^m(K)<\infty$ for some integer $0<m\le n+1$, and $\de_K\triangleq\inf\{r^{-m}\mathcal{H}^m(K\cap B_r(x))|\, B_r(x)\subset B_1(p_\infty)\}>0$.
Then from \eqref{knuinf000*}\eqref{knuinf000}, we have
\begin{equation}\aligned\label{ktoinfthjk=0}
\lim_{k\rightarrow\infty}\sum_{j=1}^{k}\th_k\th_j^{m-1}\le\de_K^{-1}\lim_{k\rightarrow\infty}\sum_{j=1}^{k}\f{\th_k}{\th_j}\mathcal{H}^m\left(B_{\th_j}(x_j)\right)=0.
\endaligned
\end{equation}
There is a constant $\be_m$ depending on $m$ such that
\begin{equation}\aligned\label{ktoinfthjk=0*}
\mathcal{H}^m_\ep(K)\le&\omega_m\sum_{j=1}^{k}(\th_{j}+2\th_k)^m\le\omega_m\sum_{j=1}^{k}\left(1+\be_m\f{\th_k}{\th_j}\right)\th_j^m.
\endaligned
\end{equation}
With \eqref{ktoinfthjk=0}\eqref{ktoinfthjk=0*}, we get \eqref{Hkepkjinfthjk}.
\end{proof}

\begin{lemma}\label{Cont*}
Let $F_i$ be a subset in $B_1(p_i)\subset N_i$ such that $\Phi_i(F_i)$ converges in the Hausdorff sense to a closed set $F_\infty$ in $B_1(p_\infty)$, then $$\nu_\infty(F_\infty)\ge\limsup_{i\rightarrow\infty}\mathcal{H}^{n+1}(F_i)/\mathcal{H}^{n+1}(B_1(p_i)).$$
\end{lemma}
\begin{proof}
Note that $\nu_\infty$ is an outer measure defined as (1.9) in \cite{CCo1}.
For any $\ep>0$, by finite covering lemma there exist a constant $\de>0$,
and a covering
$\{B_{r_j}(x_{j})\}_{j=1}^{k_\de}$ of $F_\infty$ with finite $k_\de\in\N$ and $r_j\in(0,\de)$ such that
\begin{equation}\aligned\label{10.1}
\nu_\infty(F_\infty)\ge\sum_{j=1}^{k_\de}\nu_\infty(B_{r_j}(x_j))-\ep.
\endaligned
\end{equation}
Without loss of generality, we assume that $\bigcup_{j=1}^{k_\de}B_{r_{j}}(x_{j})$ contains $B_\ep(F_\infty)$ up to choose the sufficiently small $\ep>0$,
where $B_\ep(F_\infty)$ is the $\ep$-neighborhood of $F_\infty$.
Since $B_1(p_i)$ converges to $B_1(p_\infty)$ in the Gromov-Hausdorff sense,
then there is a sequence $x_{i,j}\in B_1(p_i)$ with $x_{i,j}\rightarrow x_j$.
Combining $\Phi_i(F_i)\rightarrow F_\infty$ in the Hausdorff sense, we get $F_i\subset \bigcup_{j=1}^{k_\de}B_{r_j}(x_{i,j})$.
With
\begin{equation}\aligned
\lim_{i\rightarrow\infty}\mathcal{H}^{n+1}(B_{r_j}(x_{i,j}))/\mathcal{H}^{n+1}(B_1(p_i))=\nu_\infty(B_{r_j}(x_j))
\endaligned
\end{equation}
for each $j$, we get
\begin{equation}\aligned\label{10.2}
\mathcal{H}^{n+1}(F_i)/\mathcal{H}^{n+1}(B_1(p_i))\le\sum_{j=1}^{k_\de}\mathcal{H}^{n+1}(B_{r_j}(x_{i,j}))/\mathcal{H}^{n+1}(B_1(p_i))\le\ep+\sum_{j=1}^{k_\de}\nu_\infty(B_{r_j}(x_j))
\endaligned
\end{equation}for the sufficiently large $i$.
Combining \eqref{10.1}\eqref{10.2} we have
$$\nu_\infty(F_\infty)+2\ep\ge\limsup_{i\rightarrow\infty}\mathcal{H}^{n+1}(F_i)/\mathcal{H}^{n+1}(B_1(p_i)).$$
Letting $\ep\rightarrow0$ completes the proof.
\end{proof}

\begin{lemma}\label{subset}
Let $F_i$ be a subset in $B_1(p_i)\subset N_i$ such that $\Phi_i(F_i)$ and $\Phi_i(\p F_i)$ converge in the Hausdorff sense to closed sets $F_\infty$ and $T_\infty$ in $B_1(p_\infty)$,
respectively. Then $\p F_\infty\subset T_\infty$.
\end{lemma}
\begin{proof}
For any $\ep>0$, there is an $i_0=i_0(\ep)>0$ such that $\Phi_i(\overline{F_i})\subset B_\ep(F_\infty)$ for all $i\ge i_0$.
For any $y\in \p F_\infty$, if there is no sequence in $\p F_i$ converging to $y$,
then there are a constant $\de>0$ and a sequence $y_i\in F_i$ converging to $y$ such that the distance to $\p F_i$ at $y_i$ satisfies $d(y_i,\p F_i)\ge\de$, which
implies $B_\de(y_i)\subset F_i$. Hence, we get $B_\de(y)\subset F_\infty$,
which is a contradiction to $y\in \p F_\infty$. So there is a sequence $z_i\in\p F_i$ so that $z_i\rightarrow y$, which implies $y\in\lim_{i\rightarrow\infty}\p F_i=T_\infty$.
We complete the proof.
\end{proof}

\bibliographystyle{amsplain}

\end{document}